\documentclass[smallcondensed]{svjour3}

\title{A Generic Branch-and-Bound Algorithm for $\ell_0$-Penalized Problems with Supplementary Material}

\author{Cl\'{e}ment Elvira \and Th\'{e}o Guyard \and C\'{e}dric Herzet}

\institute{
  Cl\'{e}ment Elvira \at
  CentraleSupelec, CNRS, IETR - UMR 6164, Rennes, France \\
  \email{clement.elvira@centralesupelec.fr}
  \and
  Th\'{e}o Guyard \at
  Scale-AI Chair, Polytechnique Montr\'{e}al, Canada \\
  \email{theo.guyard@polymtl.ca}
  \and
  C\'{e}dric Herzet \at
  Ensai, CNRS, CREST - UMR 9194, Rennes, France \\
  \email{cedric.herzet@ensai.fr}
}

\date{Received: date / Accepted: date}

\usepackage{ifthen}
\newboolean{cameraready}
\setboolean{cameraready}{false}
\newboolean{shortreference}
\setboolean{shortreference}{false}
\newboolean{trackref}
\setboolean{trackref}{false}

\usepackage{kmath}
\smartqed
\usepackage{amsthm}
\usepackage{amsmath}
\usepackage{amsfonts}
\usepackage{lipsum}
\usepackage{mathtools}
\usepackage{graphicx}
\usepackage{epstopdf}
\usepackage{algorithmic}
\usepackage{glossaries}
\glsdisablehyper
\usepackage{stmaryrd}
\usepackage{booktabs}
\usepackage{pgfplots}
\usepackage{enumitem}
\usepackage{soul}
\usepackage[hidelinks]{hyperref}

\pgfplotsset{compat=newest}
\usepgfplotslibrary{groupplots}
\usetikzlibrary{matrix}
\usetikzlibrary{arrows.meta}

\definecolor{cplexcolor}{RGB}{46,106,255}
\definecolor{mosekcolor}{RGB}{69,152,62}
\definecolor{gurobicolor}{RGB}{172,226,42}
\definecolor{oacolor}{RGB}{172,226,42}
\definecolor{l0bnbcolor}{RGB}{242,112,40}
\definecolor{el0pscolor}{RGB}{187,0,0}

\pgfplotscreateplotcyclelist{solver_cycle_list}{
  cplexcolor, ultra thick, smooth\\
  scipcolor, ultra thick, smooth\\
  mosekcolor, ultra thick, smooth\\
  l0bnbcolor, ultra thick, smooth\\
  sbnbcolor, ultra thick, smooth\\
  el0pscolor, ultra thick, smooth\\
  black, ultra thick, densely dashed\\
}

\pgfplotsset{
    ylabel right/.style={
        after end axis/.append code={
            \node [rotate=270, anchor=north] at (rel axis cs:1.25,0.5) {#1};
        }   
    }
}

\spnewtheorem{assumption}{Assumption}{\bfseries}{\itshape}
\spnewtheorem{convention}{Convention}{\bfseries}{\itshape}

\usepackage{cleveref}

\newcommand{\fontset}[1]{\mathcal{#1}}

\newcommand{\extendedRealLine}{{\kR\cup\{\pm\infty\}}}
\newcommand{\rextendedRealLine}{{\kR\cup\{+\infty\}}}

\newcommand{\pobj}{p}
\newcommand{\robj}{r}
\newcommand{\dobj}{d}

\newcommand{\pvletter}{x}
\newcommand{\bvletter}{z}
\newcommand{\dvletter}{u}

\newcommand{\wvletter}{w}
\newcommand{\pv}{\mathbf{\pvletter}}
\newcommand{\bv}{\mathbf{\bvletter}}
\newcommand{\dv}{\mathbf{\dvletter}}

\newcommand{\wv}{\mathbf{\wvletter}}
\newcommand{\pvi}[1]{\pvletter_{#1}}
\newcommand{\bvi}[1]{\bvletter_{#1}}

\newcommand{\pdim}{n}
\newcommand{\ddim}{m}
\newcommand{\obs}{\mathbf{y}}
\newcommand{\dic}{\mathbf{A}}
\newcommand{\atom}[1]{\mathbf{a}_{#1}}
\newcommand{\reg}{\lambda}
\newcommand{\regmax}{\reg_{\max}}
\newcommand{\datafunc}{f}
\newcommand{\pertfunc}{h}
\newcommand{\regfunc}{g}

\newcommand{\setidx}{\fontset{S}}
\newcommand{\idxentry}{i}
\newcommand{\oneSymb}{1}
\newcommand{\zeroSymb}{0}

\newcommand{\setzero}{\setidx_{\zeroSymb}}
\newcommand{\setone}{\setidx_{\oneSymb}}

\newcommand{\subzero}[1]{#1_{\setzero}}
\newcommand{\subone}[1]{#1_{\setone}}

\newcommand{\nodeSymb}{\nu}
\newcommand{\node}[1]{#1^\nodeSymb}
\newcommand{\pset}{\fontset{X}}

\newcommand{\pertslope}{\tau}
\newcommand{\pertlimit}{\mu}
\newcommand{\pertrightlope}{\kappa}

\newcommand{\domb}{\beta}

\newcommand{\relaxregfunc}{\LB{\regfunc}}
\newcommand{\bigM}{M}
\newcommand{\regone}{\sigma}
\newcommand{\regtwo}{\sigma}
\newcommand{\groundtruth}{\pv^\dagger}
\newcommand{\groundtruthi}[1]{\pvi{#1}^\dagger}

\newcommand{\corrparam}{\rho}
\newcommand{\corrmat}{\mathbf{K}}
\newcommand{\corrmatel}{K}
\newcommand{\noise}{\boldsymbol{\epsilon}}

\newcommand{\bernoulliparam}{\beta}

\newcommand{\noisestd}{\zeta}

\newcommand{\identitymatrix}{\mathbf{I}}

\newcommand{\amplitudeveci}[1]{w_#1}

\newcommand{\normaldistrib }{\mathcal{N}}

\newcommand{\densityfunc}{\phi}
\newcommand{\proba}{\mathrm{Pr}}

\newcommand{\eg}{\textit{e.g.}}

\newcommand{\elops}{\textsc{El\scalebox{1}[0.8]{0}ps}}
\newcommand{\lobnb}{\textsc{L\scalebox{1}[0.8]{0}bnb}}
\newcommand{\oa}{\textsc{Oa}}
\newcommand{\mosek}{\textsc{Mosek}}
\newcommand{\cplex}{\textsc{Cplex}}
\newcommand{\lolearn}{\textsc{L0Learn}}

\DeclareMathOperator*{\argmin}{argmin}

\DeclareMathOperator{\closure}{cl}
\DeclareMathOperator{\conv}{conv}
\DeclareMathOperator{\domFuncOp}{dom}
\newcommand{\dom}[1]{\domFuncOp{#1}}

\DeclareMathOperator{\interiorOp}{int}
\newcommand{\interior}[1]{\interiorOp({#1})}

\DeclareMathOperator{\prox}{prox}
\DeclareMathOperator{\sign}{sign}

\newcommand{\1}{\mathbf{1}}
\newcommand{\0}{\mathbf{0}}
\newcommand{\intervint}[2]{\llbracket#1,#2\rrbracket}

\newcommand{\LB}[1]{\tilde{#1}}
\newcommand{\UB}[1]{\bar{#1}}
\newcommand{\card}[1]{|#1|}
\newcommand{\transp}[1]{#1^{\mathrm{T}}}
\newcommand{\opt}[1]{#1^{\star}}

\newcommand{\conj}[1]{#1^{\star}}
\newcommand{\biconj}[1]{#1^{\star\star}}
\newcommand{\abs}[1]{|#1|}
\newcommand{\icvx}{\eta}
\newcommand{\norm}[2]{\|#1\|_#2}
\newcommand{\pospart}[1]{[#1]_+}
\newcommand{\grad}{\nabla}
\newcommand{\subdiff}{\partial}
\newcommand{\separable}[2]{#1_{#2}}
\newcommand{\cstprox}{\gamma}

\newcommand{\ball}{\mathcal{B}}

\newcommand{\defeq}{\triangleq}
\newcommand{\inlinedefeq}{\triangleq}

\newcommand{\genericfunc}{\omega}
\newcommand{\aGenericOneDfunction}{\omega}

\newacronym{bnb}{BnB}{Branch-and-Bound}
\newacronym{mip}{MIP}{Mixed Integer Programming}
\newacronym{sos}{SOS}{Special Ordered Set}
\newacronym{bfs}{BFS}{Breadth-First-Search}
\newacronym{dfs}{DFS}{Depth-First-Search}
\newacronym{snr}{SNR}{Signal-to-Noise Ratio}
\newacronym{atp}{ATP}{Arbitrage Theory of Pricing}
\newacronym{map}{MAP}{maximum \textit{a posteriori}}
\ifthenelse{\boolean{cameraready}}{
  \newcommand{\AddNote}[1]{}
  \newcommand{\AddTodo}[1]{}

  \newcommand{\ProvideEditionMacros}[2]{%
    \expandafter\NewDocumentCommand\csname Add#1\endcsname{gg}{}%
    \expandafter\NewDocumentCommand\csname Rem#1\endcsname{gg}{}%
    \expandafter\NewDocumentCommand\csname Sup#1\endcsname{gg}{}%
  }
}{
  \newcommand{\AddNote}[1]{\textcolor{red}{[Note: #1]}}
  \newcommand{\AddTodo}[1]{\textcolor{red}{[Todo: #1]}}

  \newcommand{\ProvideEditionMacros}[2]{%
    \expandafter\NewDocumentCommand\csname Add#1\endcsname{gg}{\textcolor{#2}{##1}}%
    \expandafter\NewDocumentCommand\csname Rem#1\endcsname{gg}{\textcolor{#2}{[#1: {\footnotesize ##1}]}}%
    \expandafter\NewDocumentCommand\csname Sup#1\endcsname{gg}{\textcolor{#2}{\st{##1}}}%
  }
}

\ifthenelse{\boolean{shortreference}}{
  \newcommand{\pageInRef}[1]{}
  \newcommand{\sectionInRef}[1]{\textcolor{orange}{Sec.}~#1}
  \newcommand{\equationInRef}[1]{\textcolor{orange}{Eq.}~(#1)}
  \newcommand{\definitionInRef}[1]{\textcolor{orange}{Def.}~#1}
  \newcommand{\theoremInRef}[1]{\textcolor{orange}{Th.}~#1}
  
  \newcommand{\propositionInRef}[1]{\textcolor{orange}{Prop.}~#1}
  \newcommand{\corollaryInRef}[1]{\textcolor{orange}{Cor.}~#1}

  \newcommand{\exampleInRef}[1]{\textcolor{orange}{Ex.}~#1}
  \newcommand{\chapterInRef}[1]{\textcolor{orange}{Ch.~#1}}
  \newcommand{\chaptersInRef}[2]{\textcolor{orange}{Ch.~#1 and~#2}}

}{
  \newcommand{\pageInRef}[1]{}
  \newcommand{\sectionInRef}[1]{{Section}~#1}
  \newcommand{\equationInRef}[1]{{Equation~(#1)}}
  \newcommand{\definitionInRef}[1]{{Definition~#1}}
  \newcommand{\theoremInRef}[1]{{Theorem~#1}}
  
  \newcommand{\propositionInRef}[1]{{Proposition~#1}}
  \newcommand{\corollaryInRef}[1]{{Corollary~#1}}

  \newcommand{\exampleInRef}[1]{{Example~#1}}
  \newcommand{\chapterInRef}[1]{{Chapter~#1}}
  \newcommand{\chaptersInRef}[2]{{Chapters~#1 and~#2}}
}{
}

\ifthenelse{\boolean{trackref}}{
  \usepackage{trackRefs}
}{
  \usepackage[disable]{trackRefs}
}{
}

\ProvideEditionMacros{TG}{cyan}
\ProvideEditionMacros{CH}{blue}
\ProvideEditionMacros{CE}{green}
\ProvideEditionMacros{AA}{purple}

\begin{document}

\maketitle

\begin{abstract}
  We present a generic Branch-and-Bound procedure designed to solve $\ell_0$-penalized optimization problems. Existing approaches primarily focus on quadratic losses and construct relaxations using ``Big-M'' constraints and/or $\ell_2$-norm penalties. In contrast, our method accommodates a broader class of loss functions and allows greater flexibility in relaxation design through a general penalty term, encompassing existing techniques as special cases. We establish theoretical results ensuring that all key quantities required for the Branch-and-Bound implementation admit closed-form expressions under the general blanket assumptions considered in our work. 
  Leveraging this framework, we introduce \elops, an open-source \textsc{Python} solver with a plug-and-play workflow that enables user-defined losses and penalties in $\ell_0$-penalized problems.
  Through extensive numerical experiments, we demonstrate that \elops{} achieves state-of-the-art performance on classical instances and extends computational feasibility to previously intractable ones.
\end{abstract}


\section{Introduction}
\label{sec:introduction}

Over the past decades, significant attention has been devoted to solving problems of the form
\begin{equation}
    \label{prob:prob}
    \opt{\pobj} = \min_{\pv \in \kR^{\pdim}} \ \datafunc(\dic\pv) + \sum_{\idxentry=1}^{\pdim}\regfunc(\pvi{\idxentry})
\end{equation}
where \(\datafunc:\kR^{\ddim}\to\rextendedRealLine\) is a data fitting term such that 
\begin{enumerate}[label={\textnormal{(H0)}}, topsep=12pt,parsep=0pt,leftmargin=40pt]
    \item \textit{\(\datafunc\) is closed, convex, differentiable, lower-bounded and \(\0 \in \interior{\dom{\datafunc}}\)},
    \label{assumption:f}
\end{enumerate}
and \(\kfuncdef{\regfunc}{\kR}{\rextendedRealLine}\) is a regularizer imposing some structure on the solutions. 
In this paper, we concentrate on sparsity-inducing regularizers expressed as
\begin{equation}
    \label{eq:regfunc}
    \regfunc(\pvi{}) = \reg\norm{\pvi{}}{0} + \pertfunc(\pvi{})
\end{equation}
where 
\begin{align}
    \label{eq:l0-norm}
    \|\pvi{}\|_0
    &=
    \begin{cases}
        0 & \mbox{if \(\pvi{}=0\)}\\
        1 & \mbox{if \(\pvi{}\neq 0\)}
    \end{cases}
\end{align}
is the so-called ``\(\ell_0\)-norm'' and \(\kfuncdef{\pertfunc}{\kR}{\rextendedRealLine}\) is a function with some desirable properties, see hypotheses \ref{assumption:zero-minimized}-\ref{assumption:even} below. 

Problem \eqref{prob:prob} has been of interest in many fields of applied mathematics in recent years, including machine learning~\cite{bach2011convex}, signal processing~\cite{soussen2011bernoulli} and inverse problems~\cite{tropp2010computational}. 
Unfortunately, minimization problems involving the \(\ell_0\)-norm are in general NP-hard~\cite{chen2014complexity}. 
Over the last decades, various research avenues have consequently focused on developing greedy heuristic procedures \cite{blumensath2009iterative,chen1989orthogonal,Herzet_eusipco10,needell2009cosamp} or addressing convex \cite{bach2011convex,tibshirani1996regression,zou2005regularization} and non-convex \cite{fan2009network,soubies2015continuous,zhang2010nearly} approximations of this problem. 
Recently, the landmark paper \cite{bertsimas2016best} underscored the possibility of developing tractable procedures to address non-trivial instances of \eqref{prob:prob} by using modern tools from discrete optimization. 
In particular, the authors showed that \(\ell_0\)-norm problems can be reformulated under the \gls{mip} framework and be addressed by off-the-shelf solvers \cite{kronqvist2019review}.

These solvers, however, are designed as general-purpose optimization tools capable of handling a broad range of \gls{mip} problems beyond \eqref{prob:prob}.
While they leverage sophisticated \gls{bnb} strategies and cutting-plane methods, their generic nature implies that they do not exploit the specific structure of \eqref{prob:prob}.
As a result, their computational efficiency can be severely hampered when applied to medium- to large-scale instances typically encountered in machine learning and signal processing applications. 

To overcome these limitations, researchers have developed problem-specific implementations tailored to \eqref{prob:prob}, aiming for more efficient algorithms capable of handling larger problem instances, see \eg{}, \cite{atamturk2020safe,ben2021global,guyard2022node,hazimeh2021sparse,mhenni2020sparse,samain2022techniques}.
Rather than addressing the problem in full generality, these approaches leverage particular choices of the data fitting function \(\datafunc\) and the regularization term \(\pertfunc\) to exploit the problem structure and design specialized \gls{bnb} algorithmic strategies. 
More specifically, contributions \cite{atamturk2020safe,ben2021global,guyard2022node,mhenni2020sparse,samain2022techniques}  
concentrated on data fitting terms of the form:
\begin{equation}
    \datafunc(\wv)= \tfrac{1}{2}\norm{\obs - \wv}{2}^2    
    \qquad \forall \wv\in\kR^{\ddim}
\end{equation}
and considered one of the following regularizers: 
\begin{subequations}
    \begin{alignat}{4}
        \label{eq:bigM}
        \pertfunc(\pvi{}) &= \icvx(\abs{\pvi{}} \leq \bigM) &\qquad&\forall \pvi{} \in \kR\\
        \label{eq:sql2norm}
        \pertfunc(\pvi{}) &= \tfrac{\sigma}{2}\pvi{}^2  &\qquad&\forall \pvi{} \in \kR
    \end{alignat}
\end{subequations}
where \(\bigM\), \(\regtwo\) are positive parameters and \(\icvx\) denotes the convex indicator, as defined in \cite[\exampleInRef{1.25}]{Bauschke2017}.
A linear combination of \eqref{eq:bigM}-\eqref{eq:sql2norm} was also considered in \cite{hazimeh2021sparse}.  

A central element in the success of these  methods lies in the fact that several key quantities needed for the implementation of \gls{bnb} algorithms admit simple analytical expressions in the particular setups considered in these papers.
For example, the authors in \cite{ben2021global,pilanci2015sparse} showed that
\begin{subequations}
    \begin{align}
        \label{eq:relax-bigM}
        \biconj{\regfunc}(\pvi{})&= 
        \begin{cases}
            \frac{\reg}{\bigM}|\pvi{}| & \mbox{if \(|\pvi{}|\leq \bigM\)}\\
            +\infty & \mbox{otherwise}    
        \end{cases} \\
        \label{eq:relax-sql2norm}
        \biconj{\regfunc}(\pvi{})&= 
        \begin{cases}
            \sqrt{2\reg\sigma}|\pvi{}| & \mbox{if \(\pvi{}^2\leq \tfrac{2\reg}{\sigma}\)}\\
            \frac{\sigma}{2}\pvi{}^2+\reg & \mbox{otherwise}    
        \end{cases}
    \end{align}
\end{subequations}
are convex lower-bounds on \(\regfunc\) when \(\pertfunc\) is defined by \eqref{eq:bigM} and \eqref{eq:sql2norm}, respectively.
In these works, the authors also derived closed-form expressions for other key quantities appearing in the efficient optimization of convex relaxations involving \(\biconj{\regfunc}\), namely its convex conjugate, subdifferential and/or proximal operator.\footnote{We refer the reader to \Cref{sec:bnb} for a detailed discussion on the practical interest of these quantities in the implementation of \gls{bnb} procedures.} 

While these results have led to efficient solving procedures for specific instances of \eqref{prob:prob}, they remain limited in scope. In practice, a broader range of data fitting terms \(\datafunc\) and regularization functions \(\pertfunc\) arise in applications \cite{tillmann2024cardinality}.
To date, no general framework is available to systematically handle these diverse settings, restricting the applicability of existing methods to the few instances mentioned above.

\paragraph{Contributions.}
In this paper, we consider a methodological framework allowing to derive fast numerical solution to instances of problem \eqref{prob:prob} involving functions $\datafunc$ and \(\pertfunc\) verifying a set of blanket assumptions.
In particular, $\datafunc$ is supposed to verify \ref{assumption:f} and $\pertfunc$ to satisfy the following hypotheses:\footnote{Assumption \ref{assumption:even} can be relaxed in our derivations, but is retained to enhance the readability of our exposition.} 
\begin{enumerate}[label=\textnormal{(H\arabic*)},itemsep=4pt,topsep=12pt,parsep=0pt,leftmargin=40pt]
    \item \label{assumption:zero-minimized}
        \textit{\(\pertfunc(\pvi{}) \geq \pertfunc(0)=0\) and \(\dom\pertfunc\cap\kR+\setminus\{0\}\neq\emptyset\)}.
     \item \label{assumption:closed}    
        \textit{\(\pertfunc\) is closed}. 
    \item \label{assumption:convex}
        \textit{\(\pertfunc\) is convex.}              
    \item \label{assumption:coercive}
        \textit{\(\pertfunc\) is coercive.}
    \item \label{assumption:even}
        \textit{\(\pertfunc\) is even.} 
\end{enumerate}
Our contributions are twofold.

First, we provide a systematic strategy to efficiently compute the largest proper, closed and convex function that lower-bound \(\regfunc\), as well as its convex conjugate, subdifferential, and proximal operator. 
In particular, we show that under 
\ref{assumption:zero-minimized}-\ref{assumption:closed}-\ref{assumption:convex}-\ref{assumption:coercive}-\ref{assumption:even}, 
all these quantities admit simple closed-form expressions only depending on three key parameters:
\begin{align}
    \label{eq:pertslope}
    \pertslope &\defeq  \sup \kset{\bvi{}\in\kR+}{\conj{\pertfunc}(\bvi{}) \leq \reg} \\
    \label{eq:pertlimit}
    \pertlimit &\defeq 
    \begin{cases}
        \sup \kset{\bvi{}\in\kR+}{\bvi{} \in \subdiff\conj{\pertfunc}(\pertslope)} &\text{ if } \subdiff\conj{\pertfunc}(\pertslope) \neq \emptyset \\
        +\infty &\text{ otherwise}
    \end{cases}
    \\
    \label{eq:pertrightlope}
    \pertrightlope &\defeq
    \begin{cases}
     \sup\kset{\bvi{}\in\kR+}{\bvi{} \in \subdiff\pertfunc(\pertlimit)} & \mbox{if \(\pertlimit<+\infty\)}\\
     +\infty &\mbox{otherwise,}
     \end{cases}
\end{align}
where $\subdiff$ denotes the subdifferential operator.
As outlined in \Cref{table:examples:intro}, the quantities \(\pertslope\), \(\pertlimit\), and \(\pertrightlope\) admit simple closed-form expressions for many choices of \(\pertfunc\) of practical interest. 
Interestingly, the proposed characterization unifies and extends previous results in the literature.
For instance, we show in \Cref{prop:regfunc-biconj} that the tightest proper, closed, and convex lower bound on \(\regfunc\) -- namely, its biconjugate function -- can be expressed as:
\begin{equation}
    \label{eq:expr-g**}
    \forall \pvi{}\in\kR:\ 
    \biconj{\regfunc}(\pvi{}) =
     \begin{cases}
        \pertslope\abs{\pvi{}} & \text{if} \ \abs{\pvi{}} \leq \pertlimit \\
        \pertfunc(\pvi{}) + \reg & \text{if} \ \abs{\pvi{}} \geq \pertlimit
        .
    \end{cases}
\end{equation}
In particular, this formulation directly recovers the previously known results \eqref{eq:relax-bigM}-\eqref{eq:relax-sql2norm} as special cases when \(\pertfunc\) is defined according to \eqref{eq:bigM}-\eqref{eq:sql2norm}. 
Given the generality of our assumptions on \(\datafunc\) and \(\pertfunc\), our framework thus enables the efficient implementation of \gls{bnb} algorithms for a much wider range of problem instances than was previously feasible.

\begin{table}[!t]
    \centering
    \small
    \begin{tabular}{llccc} 
        \toprule
        Function 
        & \(\pertslope\) & \(\pertlimit\)& \(\pertrightlope\) \\ [0.5ex]
        \midrule
        \(\pertfunc(\pvi{}) = \icvx(\abs{\pvi{}} \leq \bigM)\) & \(\tfrac{\reg}{\bigM}\) & \(\bigM\) & \(+\infty\) \\ [1ex]
        \(\pertfunc(\pvi{}) = \sigma\abs{\pvi{}}\) & \(\sigma\) & \(+\infty\) & \(+\infty\) \\ [1ex]
        \(\pertfunc(\pvi{}) = \tfrac{\sigma}{p}\abs{\pvi{}}^p\) with $p>1$ & \(\sigma(\tfrac{p\reg}{(p-1)\sigma})^{\tfrac{p-1}{p}}\) & \((\tfrac{p\reg}{(p-1)\sigma})^{\tfrac{1}{p}}\) & \(\sigma(\tfrac{p\reg}{(p-1)\sigma})^{\tfrac{p-1}{p}}\) \\ [1ex]
        \(\pertfunc(\pvi{}) = \sigma\abs{\pvi{}} + \tfrac{\sigma'}{2}\pvi{}^2\) & \(\sigma+\sqrt{2\reg\sigma'}\) & \(\sqrt{\tfrac{2\reg}{\sigma'}}\) & $\sigma + \sqrt{2\reg\sigma'}$ \\ [1ex]
        \(\pertfunc(\pvi{}) = \sigma\abs{\pvi{}} + \icvx(\abs{\pvi{}} \leq \bigM)\) & \(\sigma + \tfrac{\reg}{\bigM}\) & \(\bigM\) & $+\infty$ \\ [1ex]
        \(\pertfunc(\pvi{}) = \tfrac{\sigma}{2}\pvi{}^2 + \icvx(\abs{\pvi{}} \leq \bigM)\) with \(\reg < \tfrac{\sigma}{2}\bigM^2\) & \(\sqrt{2\reg\sigma}\) & \(\sqrt{\tfrac{2\reg}{\sigma}}\) & $\sqrt{2\reg\sigma}$ \\ [1ex]
        \(\pertfunc(\pvi{}) = \tfrac{\sigma}{2}\pvi{}^2 + \icvx(\abs{\pvi{}} \leq \bigM)\) with \(\reg \geq \tfrac{\sigma}{2}\bigM^2\) & \(\tfrac{\reg}{\bigM}+\tfrac{\sigma\bigM}{2}\) & \(\bigM\) & $+\infty$ \\ [1ex]
        \bottomrule
    \end{tabular}
    \caption{Examples of expression of \(\pertslope\), \(\pertlimit\), \(\pertrightlope\).  
    The parameters \(\bigM\), \(\sigma\) and \({\sigma'}\) are taken positive.}
    \label{table:examples:intro}
\end{table}

Second, we introduce \elops, an open-source \textsc{Python} toolbox\footnote{The code is available in open-access at \href{https://github.com/TheoGuyard/El0ps}{\texttt{https://github.com/TheoGuyard/El0ps}}.} implementing a state-of-the-art solving procedure for problems of the form \eqref{prob:prob}  satisfying assumptions \ref{assumption:f}-\ref{assumption:zero-minimized}-\ref{assumption:closed}-\ref{assumption:convex}-\ref{assumption:coercive}-\ref{assumption:even}. This toolbox provides a flexible and efficient alternative to commercial solvers, allowing users to handle various loss functions and regularization terms while benefiting from significant computational gains. 
Through numerical simulations, we demonstrate that the proposed solver can accelerate the resolution of \eqref{prob:prob} by several orders of magnitude across various standard problems in machine learning and signal processing. 

\paragraph{Outline.}
The rest of the paper is organized as follows. 
In \Cref{subsec:notations}, we define the notational conventions used in our derivations. 
In \Cref{sec:bnb}, we remind the main principles of \gls{bnb} algorithms and 
discuss several ingredients playing a central role in their implementation. 
In \Cref{sec:implementation}, we show that the main quantities needed for the efficient implementation of \gls{bnb} procedures for problem \eqref{prob:prob} admit simple closed-form expressions under our blanket assumptions.  
In \Cref{sec:numerics}, we evaluate the performance of our toolbox, \elops{}, which leverages our mathematical results to provide an efficient and flexible \gls{bnb} solver tailored to problem \eqref{prob:prob}. All the technical details of our derivations are postponed to the appendices.


\section{Notations and Conventions} \label{subsec:notations}
 
Classical letters (\eg,~\(x\)) represent scalars, boldface lowercase (\eg,~\(\pv\)) letters represent vectors and boldface uppercase (\eg,~\(\dic\)) letters represent matrices.
\(\0\) and \(\1\) respectively denote the all-zeros and all-ones vectors whose dimension is usually  clear from the context. 
The notation $\ktranspose{\pv}$ stands for the transpose of vector $\pv$ and $\odot$ is used to represent the element-wise product between two vectors. 
The \(\idxentry\)-th entry of a vector \(\pv\) is denoted by \(\pvi{\idxentry}\) and \(\pv_{\setidx}\) corresponds to the restriction of \(\pv\) to its elements indexed by \(\setidx\). The $\ddim$-dimensional identity matrix is denoted $\mathbf{I}_\ddim$. 

Calligraphic letters (\eg,~\(\setidx\)) are used to define sets, \(\card{\cdot}\) denotes their cardinality and \(\conv(\cdot)\) their convex hull. Given two integers \(a,b\in\kN\), the notation \(\intervint{a}{b}\) represents the set of all integers between \(a\) and \(b\). 
Given \(\epsilon>0\) and \(\pv\in\kR^\pdim\), we define 
\begin{align}
	\ball(\pv,\epsilon)
	&\defeq 
	\kset{\pv'\in\kR^\pdim}{\|\pv'-\pv\|_2<\epsilon}
	.
\end{align}
\(\icvx(\cdot)\) (resp. $1(\cdot)$) denotes the indicator function, which is equal to \(0\) (resp. $1$) if the condition in parentheses is satisfied and \(+\infty\) (resp. $0$) otherwise. 
The positive-part function is defined as \(\pospart{\pvi{}} \inlinedefeq \max(\pvi{},0)\) and is applied coordinate-wise to vectors.
Given some function \(\kfuncdef{\genericfunc}{\kR^\pdim}{\extendedRealLine}\), we let \(\conj{\genericfunc}(\bv) \inlinedefeq \sup_{\pv} \transp{\bv}\pv - \genericfunc(\pv)\) denote its convex conjugate.~Similarly, the convex biconjugate function of \(\genericfunc\) is denoted \(\biconj{\genericfunc} \inlinedefeq \conj{(\conj{\genericfunc})}\). We let \(\dom \genericfunc \inlinedefeq \kset{\pv\in\kR^\pdim}{\genericfunc(\pv)<+\infty}\) be the domain of \(\genericfunc\). 
The subdifferential set of \(\genericfunc\) at some \(\pv\in\dom \genericfunc\) is written\footnote{Following~\cite{beck2017first}, we only define the subdifferential of a proper function at points of the domain. } 
\begin{align}
	\subdiff\genericfunc(\pv) \defeq \kset{\dv\in\kR^\pdim}{\forall \pv'\in\kR^\pdim:\ \genericfunc(\pv') \geq \genericfunc(\pv) + \transp{\dv}(\pv' - \pv)}	
\end{align}
and any \(\dv \in \subdiff\genericfunc(\pv)\) is referred to as a subgradient.
The proximal operator of \(\genericfunc\) is denoted \(\prox_{\genericfunc}(\bv) \inlinedefeq \argmin_{\pv} \tfrac{1}{2}\norm{\pv - \bv}{2}^2 + \genericfunc(\pv)\). 

Given some sequence \(\{\pv^{(k)}\}_{k\in\kN}\), we note \(\lim_{k \rightarrow +\infty}\pv^{(k)}=\pv\) if the sequence converges to some limit point \(\pv\). 
Given some function \(\kfuncdef{\genericfunc}{\kR^\pdim}{\rextendedRealLine}\), we note \(\lim_{\pv' \rightarrow \pv}\genericfunc{(\pv')}=b\) if for any sequence \(\{\pv^{(k)}\}_{k\in\kN}\) converging to \(\pv\), the sequence \(\{\genericfunc(\pv^{(k)})\}_{k\in\kN}\) converges to \(b\). 
For the one-dimensional case \(\pdim=1\), we note \(\lim_{\pvi{}'\uparrow \pvi{}}\genericfunc(\pvi{})=b\), if for any scalar sequence \(\{\pvi{}^{(k)}\}_{k\in\kN}\) converging to \(\pvi{}\) and such that \(\pvi{}^{(k)}<\pvi{}\), we have that \(\{\genericfunc(\pvi{}^{(k)})\}_{k\in\kN}\) converges to \(b\). 
The notion of continuity used throughout the paper has to be understood in the sense of \cite[\chapterInRef{2}]{beck2017first}: a proper function \(\kfuncdef{\genericfunc}{\kR^\pdim}{\rextendedRealLine}\) is said to be continuous on \(\dom \genericfunc\) if for any sequence \(\{\pv^{(k)}\}_{k\in\kN}\subset\dom\genericfunc\) such that \(\lim_{k \rightarrow +\infty}\pv^{(k)}=\pv \in\dom \genericfunc\), we have \(\lim_{k \rightarrow +\infty}\genericfunc(\pv^{(k)})=\genericfunc(\pv)\). Given some subset \(\mathcal{X}\subseteq \kR^\pdim\), we let \(\mathrm{cl}(\mathcal{X})\) and \(\mathrm{int}(\mathcal{X})\) respectively denote the closure and the interior of \(\mathcal{X}\) with respect to the standard metric topology on \(\kR^\pdim\). 

Finally, we use the following conventions throughout the paper: 
\textit{i)} \(0 \cdot \infty = \infty \cdot 0 = 0\) for \(\infty \in \{-\infty,+\infty\}\),
\textit{ii)} \(\infty \leq \infty\) for \(\infty \in \{-\infty,+\infty\}\), 
\textit{iii)} \(\inf\kset{\pvi{}}{\pvi{} \in \emptyset} = +\infty\) and \(\sup\kset{\pvi{}}{\pvi{} \in \emptyset} = -\infty\).


\section{Some Key Quantities for BnB Implementation}
\label{sec:bnb}

\glsreset{bnb}\gls{bnb} refers to a family of algorithms designed to find the minimizers of an optimization problem to machine precision. 
In \Cref{sec:bnb:pruning-bounding}, we first remind the main principles of \gls{bnb} algorithms tailored to problem \eqref{prob:prob}.
We refer the reader to \cite[\chapterInRef{5}]{Locatelli:2013xr} for a thorough presentation of this type of algorithm. 
From \Cref{sec:bnb:lower-bounding,sec:bnb:numerical-tools,sec:bnb:dual-bounds}, we then discuss a set of functions and operators playing a central role in the \gls{bnb} implementation.

\subsection{Pruning and Bounding}
\label{sec:bnb:pruning-bounding}

The crux of \gls{bnb} procedures consists in identifying subsets of \(\kR^{\pdim}\) which cannot contain any minimizer of \eqref{prob:prob}.
As shown in \Cref{proof:existence-minimizer}, such a minimizer always exists in the setup considered in this paper. 
To do so, this family of methods constructs a decision tree in which each node corresponds to a particular subset of \(\kR^{\pdim}\). 
In our context, a tree node is identified by two disjoint subsets \(\setzero\) and \(\setone\) of \(\intervint{1}{\pdim}\).  
The goal at some node \(\nodeSymb \inlinedefeq{} (\setzero,\setone)\) is to detect whether a solution to \eqref{prob:prob} can be attained within the set 
\begin{align}
    \label{eq:region}
    \node{\mathcal{X}}
    \defeq \kset{\pv\in\kR^\pdim}{\subzero{\pv} = \0, \ \subone{\pv} \neq \0}
    .
\end{align} 
Letting 
\begin{equation}
    \label{eq:node_problem}
    \node{\pobj} \defeq \inf_{\pv \in \node{\pset}} 
    \ \datafunc(\dic\pv) + \sum_{\idxentry=1}^{\pdim}\regfunc(\pvi{\idxentry})
\end{equation}
be the smallest objective value achievable
over \(\node{\pset}\), we have that 
this subset does not contain any minimizer of problem \eqref{prob:prob} provided that
\begin{align}
    \label{eq:pruning condition}
    \node{\pobj} > \opt{\pobj}
    .
\end{align}
If inequality \eqref{eq:pruning condition} is satisfied, we can therefore safely discard \(\node{\pset}\) from the search space of the optimization problem without altering its solution. This operation is commonly referred to as \textit{``pruning''}.

Unfortunately, \(\node{\pobj}\) and \(\opt{\pobj}\) are NP-hard to evaluate and 
inequality \eqref{eq:pruning condition} is therefore of little practical interest.
This issue can nevertheless be circumvented by finding some tractable lower bound \(\node{\LB{\pobj}}\) on \(\node{\pobj}\) and upper bound \(\UB{\pobj}\) on \(\opt{\pobj}\). 
Pruning condition \eqref{eq:pruning condition} can then be relaxed as
\begin{align}
    \label{eq:relaxed_pruning_condition}
    \node{\LB{\pobj}} > \UB{\pobj} 
    .
\end{align}
The computation of \(\node{\LB{\pobj}}\) and \(\UB{\pobj}\) is commonly named \textit{``bounding''}.

On the one hand, finding an upper bound \(\UB{\pobj}\) is an easy task since the value of the objective function in \eqref{eq:node_problem} at any feasible point constitutes an upper bound on \(\opt{\pobj}\).
Several methods to construct relevant candidates have been proposed, see \eg{}, \cite{hazimeh2021sparse,mhenni2020sparse,pilanci2015sparse,tropp2010computational}. 
On the other hand, constructing a lower bound on \(\node{\pobj}\) is usually more computationally involved. A standard approach consists in minimizing a convex lower bound on a modified version of the objective function in \eqref{eq:node_problem}. 
We elaborate on this point in the next subsections. 
More specifically, we emphasize that the quantities \(\biconj\regfunc\), \(\subdiff\biconj\regfunc\), \(\prox_{\cstprox\biconj{\regfunc}}\) and \(\conj\regfunc\) play a central role in the construction and the minimization of these convex lower bounds. 

\subsection{Convex Relaxation}
\label{sec:bnb:lower-bounding}

In this section, we present a standard strategy to construct some lower bound \(\node{\LB{\pobj}}\). 
We first note that the infimum in \eqref{eq:node_problem} can be expressed as a minimum after a slight modification of the problem: 
\begin{proposition}
    \label{prop:prob-node}
    Under hypotheses \ref{assumption:f}-\ref{assumption:zero-minimized}-\ref{assumption:closed}-\ref{assumption:convex}-\ref{assumption:coercive}, we have
    \begin{equation}
        \label{prob:prob-node}
        \node{\pobj} = \min_{\pv \in \kR^{\pdim}} \datafunc(\dic\pv) + \sum_{\idxentry=1}^{\pdim} \node{\separable{\regfunc}{\idxentry}}(\pvi{\idxentry})
    \end{equation}
    where
    \begin{equation} \label{eq:regfunc-node} 
        \node{\separable{\regfunc}{\idxentry}}(\pvi{}) = 
        \begin{cases}
            \icvx(\pvi{} = 0) &\text{if} \ \idxentry \in \setzero \\
            \pertfunc(\pvi{}) + \reg &\text{if} \ \idxentry \in \setone \\
            \regfunc(\pvi{}) &\text{otherwise}.
        \end{cases}
    \end{equation}
\end{proposition}    
We refer the reader to \Cref{proof:equivalence-problems-l0} for a proof of this result. 
Problem \eqref{prob:prob-node} corresponds to a continuous extension of \eqref{eq:node_problem} where the function \(\regfunc(\pvi{})+\icvx(\pvi{}\neq 0)\), appearing implicitly in \eqref{eq:node_problem}, is continuously prolonged to \(\pvi{}=0\) for any \(\idxentry\in\setone\). 

A common strategy to find some tractable lower bound on \(\node{\pobj}\) then consists in minimizing a convex relaxation of the cost function in \eqref{prob:prob-node}. 
Since \(\datafunc\) is convex according to \ref{assumption:f}, finding such a convex relaxation can be achieved by identifying the tightest convex lower bound on \(\node{\separable{\regfunc}{\idxentry}}\), denoted \(\node{\separable{\relaxregfunc}{\idxentry}}\) in the sequel.
In the framework considered in this paper, the latter coincides with the convex biconjugate of \(\node{\separable{\regfunc}{\idxentry}}\).\footnote{\label{footnote:g**=tightest-convex-lower-bound}This is a consequence of~\cite[\propositionInRef{13.45}]{Bauschke2017} since \(\node{\separable{\regfunc}{\idxentry}}\) is lower-bounded under \ref{assumption:zero-minimized}.}
More precisely, standard properties of the convex biconjugate function \cite[\propositionInRef{13.30} and \theoremInRef{13.37}]{Bauschke2017} yield that 
\begin{equation} \label{eq:relax-gi}
    \node{\separable{\relaxregfunc}{\idxentry}}(\pvi{}) =
    \begin{cases}
        \icvx(\pvi{} = 0) &\text{if} \ \idxentry \in \setzero \\
        \pertfunc(\pvi{}) + \reg &\text{if} \ \idxentry \in \setone \\
        \biconj{\regfunc}(\pvi{}) &\text{otherwise.}
    \end{cases}            
\end{equation}
As a consequence, a lower bound on \(\node{\pobj}\) can be found by solving the following convex optimization problem:
\begin{equation} \label{prob:relax-node}
    \node{\robj} = \min_{\pv \in \kR^{\pdim}} 
    \datafunc(\dic\pv) + \sum_{\idxentry=1}^{\pdim}\node{\separable{\relaxregfunc}{\idxentry}}(\pvi{\idxentry})
    .
\end{equation}  
We show in \Cref{proof:existence-minimizer-relax-node} that a minimizer to this problem always exists and elaborate on standard strategies to solve this type of convex optimization problems in the next subsections. 
For now, we observe that the evaluation of \(\node{\separable{\relaxregfunc}{\idxentry}}\)  is simple as soon as the biconjugate function \(\biconj\regfunc\) admits a simple closed-form expression.
In the next section, we will show that \(\biconj\regfunc\) relates in a simple manner to \({\pertfunc}\) through the parameters \((\pertslope,\pertlimit,\pertrightlope)\) advertised in the introduction.
In particular, \Cref{prop:regfunc-biconj} in \Cref{sec:characterisation-g**} suggests that both functions can be evaluated with similar numerical complexity. 

\subsection{Dual relaxation}
\label{sec:bnb:dual-bounds}

A classical strategy to improve the running time of \gls{bnb} algorithms is to avoid solving the relaxed problem \eqref{prob:relax-node} to machine precision.
Unfortunately, this early-stopping strategy may result in lower bounds which are no longer valid and can thus potentially alter pruning decisions. 
To ensure correctness of the \gls{bnb} procedure while using an early-stopping strategy, a standard approach consists in considering the Fenchel-Rockafellar dual \cite[\definitionInRef{15.19}]{Bauschke2017} associated with \eqref{prob:relax-node}, that is 
\begin{equation} 
    \label{prob:dual-node}
    \node{\dobj} = \max_{\dv \in \kR^{\ddim}} \ -\conj{\datafunc}(-\dv) - \sum_{\idxentry=1}^{\pdim} \conj{(\node{\separable{\relaxregfunc}{\idxentry}})}(\ktranspose{\atom{}}_\idxentry \dv{})
    .
\end{equation}
Notably, we have from \cite[\propositionInRef{13.23} and \propositionInRef{13.30}]{Bauschke2017} that 
\begin{equation}
    \label{eq:conj-relaxregfunc}
    \conj{(\node{\separable{\relaxregfunc}{\idxentry}})}(\bvi{}) = 
    \left\{
        \begin{array}{ll}
            0 & \text{if} \ \idxentry \in \setzero \\
            \conj{\pertfunc}(\bvi{}) - \reg & \text{if} \ \idxentry \in \setone 
            \\
            \conj{\regfunc}(\bvi{}) &\text{otherwise.} 
        \end{array}
    \right.
\end{equation}
The couple of problems \eqref{prob:relax-node}-\eqref{prob:dual-node} at least satisfies weak duality, that is \(\node{\robj} \geq \node{\dobj}\) \cite[\propositionInRef{15.21.i}]{Bauschke2017}.
Hence, evaluating the objective of \eqref{prob:dual-node} at any \(\dv \in \kR^{\ddim}\) necessarily yields a valid lower bound on \(\node{\robj}\), and thus on \(\node{\pobj}\). 
One can therefore always stop the solving procedure addressing \eqref{prob:relax-node} before reaching machine accuracy, choose some ``judicious'' dual point\footnote{A typical choice is \(\dv \in -\subdiff\datafunc(\dic\hat{\pv})\), where \(\hat{\pv}\) is the last iterate of the numerical procedure addressing \eqref{prob:relax-node}. This is motivated by primal-dual optimality conditions \cite[\theoremInRef{19.1}]{Bauschke2017}.} and evaluate the objective of \eqref{prob:dual-node} to set the value of the lower-bound.
This ensures the validity of the \gls{bnb} process, even though \eqref{prob:relax-node} has not been solved exactly.

Recent works~\cite{atamturk2020safe,guyard2022node,deza2022safe,pmlr-v235-guyard24a} also highlighted that dual relaxation can be of interest to allow the simultaneous pruning of several regions with negligible computational overhead, thereby yielding significant computational speedup in various setups related to machine learning and signal processing.
These contributions leverage (either explicitly or implicitly) a nesting property in the definition of the dual function at different nodes of the \gls{bnb} tree resulting from the closed-form expression of \(\conj{\regfunc}\), see~\cite[\propositionInRef{2}]{pmlr-v235-guyard24a}.

We notice from \eqref{prob:dual-node}-\eqref{eq:conj-relaxregfunc} that the evaluation of the dual function requires the knowledge of \(\conj{\datafunc}\), \(\conj{\pertfunc}\) and \(\conj{\regfunc}\). In fact, \(\conj{\datafunc}\) and \(\conj{\pertfunc}\) admit simple closed-form expressions for a wide variety of convex functions \(\datafunc\) and \(\regfunc\), see \cite[\sectionInRef{4.4}]{beck2017first}.
In the next section, we will show that \(\conj{\regfunc}\) depends in simple way on \(\conj{\pertfunc}\). 
In particular, \Cref{prop:regfunc-conj} in \Cref{sec:characterisation-g*} emphasizes that \(\conj{\regfunc}\) and \(\conj{\pertfunc}\) can be evaluated with the same computational complexity. 

\subsection{Numerical Tools for Solving Relaxations}
\label{sec:bnb:numerical-tools}

Implementing the bounding step via the lower bounds presented in \Cref{sec:bnb:lower-bounding,sec:bnb:dual-bounds} necessitates access to the optimum values of \eqref{prob:relax-node} and \eqref{prob:dual-node}. As closed-form solutions are typically unavailable for these optimization problems, numerical procedures are requisite for their evaluation. 
Many methods addressing this kind of composite optimization problem have been proposed in the literature over the past decades~\cite{beck2009fast,bertrand2022beyond,boyd2011distributed,lee2006efficient,massias2017safe,wright2015coordinate}. 
They mostly require access to first-order information about the functions \(\datafunc\), \(\pertfunc\) and \(\biconj{\regfunc}\) when addressing \eqref{prob:relax-node}, and about the functions \(\conj{\datafunc}\), \(\conj{\pertfunc}\) and \(\conj{\regfunc}\) when addressing \eqref{prob:dual-node}, specifically their subdifferential and proximal operators.
In general, the operators associated with \(\datafunc\), \(\pertfunc\), \(\conj{\datafunc}\) and \(\conj{\pertfunc}\) can be obtained in closed-form from standard convex analysis results, see~\cite[\chaptersInRef{4}{6}]{beck2017first}. 
Hence, the application of standard optimization methods tailored to \eqref{prob:relax-node} and \eqref{prob:dual-node} mainly depends on the ability to characterize these operators for \(\biconj{\regfunc}\) and \(\conj{\regfunc}\).
In the next section, we show that under our blanket assumptions, these operators have simple closed-form expressions in terms of the parameters \(\pertslope{}\), \(\pertlimit{}\), and \(\pertrightlope{}\), see \Cref{prop:subdiff-gbiconj,prop:operator-prox-gbiconj} in \Cref{sec:characterisation-g**} and \Cref{prop:subdiff-conjugate of g,prop:operator-prox-gstar} in \Cref{sec:characterisation-g*}.


\section{Mathematical Characterization of the Key Ingredients}
\label{sec:implementation}

In the previous section, we highlighted that efficiently implementing \gls{bnb} solvers for problem~\eqref{prob:prob} is feasible whenever the functions \(\conj{\regfunc{}}\) and \(\biconj{\regfunc{}}\) as well as their subdifferential and proximal operators can be evaluated in a tractable manner. 
This section aims to show that, under our blanket assumptions, these quantities have simple closed-form expressions in terms of parameters \(\pertslope\), \(\pertlimit\), and \(\pertrightlope\) introduced in \eqref{eq:pertslope}-\eqref{eq:pertrightlope}.
We structure our exposition by distinguishing between quantities related to \(\biconj{\regfunc{}}\) in \Cref{sec:characterisation-g**}, and those associated with \(\conj{\regfunc{}}\) in \Cref{sec:characterisation-g*}.

\subsection{Characterization of \texorpdfstring{\(\biconj{\regfunc}\)}{\expandafter\regfunc**}}
\label{sec:characterisation-g**}

We first elaborate on the expression of the biconjugate function \(\biconj{\regfunc}\) and its subdifferential/proximal operators.
The next result provides a closed-form expression of the convex biconjugate of $\regfunc$ as a function of parameters $\pertslope$ and $\pertlimit$: 
\begin{proposition} \label{prop:regfunc-biconj}
    If~\ref{assumption:zero-minimized}-\ref{assumption:closed}-\ref{assumption:convex}-\ref{assumption:even} hold, then \(\dom{\biconj{\regfunc}}=\dom{\pertfunc}\) and 
    \begin{equation} \label{eq:regfunc-biconj} 
        \forall \pvi{} \in \kR:\
        \biconj{\regfunc}(\pvi{}) =
        \begin{cases}
            \pertslope\abs{\pvi{}} & \text{if} \ \abs{\pvi{}} \leq \pertlimit \\
            \pertfunc(\pvi{}) + \reg & \text{if} \ \abs{\pvi{}} \geq \pertlimit.
        \end{cases}       
    \end{equation}
\end{proposition}
A proof of this result is available in \Cref{proof:regfunc-biconj}. 
\Cref{fig:relaxregfunc} illustrates the result stated in \Cref{prop:regfunc-biconj} for some particular choices of function \(\pertfunc\) along with the correspond values of $\pertslope{}$ and $\pertlimit{}$. 

As outlined in \Cref{table:examples:intro}, the parameters $\pertslope{}$ and $\pertlimit$ can be evaluated in closed-form for many functions $\pertfunc{}$ encountered in the literature. 
Interestingly, when $\pertfunc{}$ is defined as in \eqref{eq:bigM}-\eqref{eq:sql2norm}, one recovers results previously published in the literature \cite{atamturk2020safe,ben2021global,guyard2022node,mhenni2020sparse,samain2022techniques} as particular cases of \Cref{prop:regfunc-biconj}. 
However, our result extends beyond these works by providing the closed-form expression of a convex lower bound on \(\regfunc{}\) for any function \(\pertfunc{}\) satisfying \ref{assumption:zero-minimized}-\ref{assumption:closed}-\ref{assumption:convex}-\ref{assumption:even}. As a byproduct of our derivations, we also point out that this bound is the largest proper, closed and convex function that lower-bound \(\regfunc\), as it corresponds to its biconjugate function, see \Cref{footnote:g**=tightest-convex-lower-bound} in \Cref{sec:bnb:lower-bounding}.

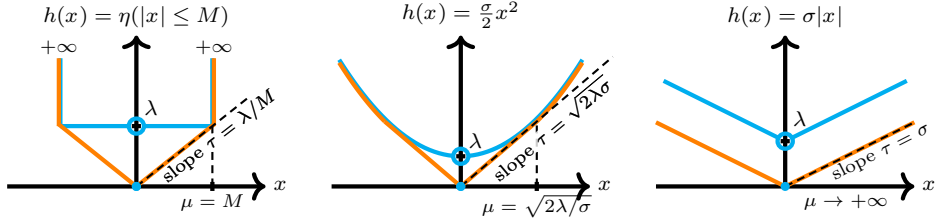
\begin{figure}[t]
    \centering
    \begin{tikzpicture}
    \begin{scope}[xscale=2,yscale=2]
        \node (origin) at (0,0) {};
        \draw[ultra thick,->] (-0.85,0) -- (0.85, 0);
        \draw[ultra thick,->] (0,0) -- (0, 1);
        \draw[ultra thick] (-0.03,0.4) -- (0.03,0.4);
        \node[above right] at (0,0.42) {$\reg$};
        \node[right] at (0.85,0) {$\pvi{}$};
        \node[above] at (0,1) {\small{$\pertfunc(\pvi{}) = \icvx(\abs{\pvi{}} \leq \bigM)$}};
        \draw[cyan,ultra thick] (-0.5,0.4) plot[domain=0.05:0.5] (\x,0.4);
        \draw[cyan,ultra thick] (-0.5,0.4) plot[domain=-0.5:-0.05] (\x,0.4) node {};
        \draw[cyan,ultra thick] (0,0.4) circle (0.05);
        \draw[cyan,ultra thick] (0.5,0.4) -- (0.5,0.85);
        \draw[cyan,ultra thick] (-0.5,0.4) -- (-0.5,0.85);
        \draw[orange,ultra thick] (0, 0) -- (0.515, 0.415);
        \draw[orange,ultra thick] (0.51,0.4) -- (0.51,0.85);
        \draw[orange,ultra thick] (0, 0) -- (-0.515, 0.415);
        \draw[orange,ultra thick] (-0.51,0.4) -- (-0.51,0.85);
        \fill[cyan] (0,0) circle (0.03);
        \node[above] at (-0.51,0.82) {\scriptsize{$+\infty$}};
        \node[above] at (0.51,0.82) {\scriptsize{$+\infty$}};
        \draw[ultra thick] (0.5,-0.03) -- (0.5,0.03);
        \node[below] at (0.5,0) {\scriptsize{$\pertlimit=\bigM$}};
        \draw[thick,densely dashed] (0.5,0) -- (0.5,0.4);
        \draw[thick,densely dashed] (0, 0) -- (0.75, 0.6);
        \node[rotate=38,anchor=center,fill=white,fill opacity=0.9,inner sep=-1] at (0.55,0.33) {\scriptsize{slope $\pertslope = \reg/\bigM$}};
        \node[rotate=38,anchor=center] at (0.55,0.33) {\scriptsize{slope $\pertslope = \reg/\bigM$}};
        \fill[cyan] (0,0) circle (0.03);
    \end{scope}
    \begin{scope}[shift={(0.35\linewidth,0)},xscale=2,yscale=2]
        \node (origin) at (0,0) {};
        \draw[ultra thick,->] (-0.85,0) -- (0.85, 0);
        \draw[ultra thick,->] (0,0) -- (0, 1);
        \draw[ultra thick] (-0.03,0.2) -- (0.03,0.2);
        \node[above right] at (0,0.2) {$\reg$};
        \node[right] at (0.85,0) {$\pvi{}$};
        \node[above] at (0,1) {\small{$\pertfunc(\pvi{}) = \tfrac{\regone}{2}\pvi{}^2$}};
        \draw[cyan,ultra thick] (-0.5,0.2) plot[domain=0.05:0.8] (\x,0.2+\x^2);
        \draw[cyan,ultra thick] (-0.5,0.2) plot[domain=-0.8:-0.05] (\x,0.2-\x^2) node {};
        \draw[cyan,ultra thick] (0,0.2) circle (0.05);
        \draw[orange,ultra thick] (0, 0) -- (0.5, 0.43);
        \draw[orange,ultra thick] (0.5, 0.43) plot[domain=0.5:0.8] (\x,0.2+\x^2-0.02) node {};
        \draw[orange,ultra thick] (0, 0) -- (-0.5, 0.43);
        \draw[orange,ultra thick] (-0.5, 0.43) plot[domain=-0.8:-0.5] (\x,0.2-\x^2-0.02) node {};
        \fill[cyan] (0,0) circle (0.03);
        \draw[thick,densely dashed] (0, 0) -- (1, 0.86);
        \node[below] at (0.5,0) {\scriptsize{$\pertlimit = \sqrt{2\reg/\regone}$}};
        \draw[ultra thick] (0.5,-0.03) -- (0.5,0.03);
        \draw[thick,densely dashed] (0.5,0) -- (0.5,0.42);
        \node[rotate=40,anchor=center,fill=white,fill opacity=0.8,inner sep=-1] at (0.6,0.38) {\scriptsize{slope $\pertslope = \sqrt{2\reg\regone}$}};
        \node[rotate=40,anchor=center] at (0.6,0.38) {\scriptsize{slope $\pertslope = \sqrt{2\reg\regone}$}};
        \fill[cyan] (0,0) circle (0.03);
    \end{scope}
    \begin{scope}[shift={(0.7\linewidth,0)},xscale=2,yscale=2]
        \node (origin) at (0,0) {};
        \draw[ultra thick,->] (-0.85,0) -- (0.85, 0);
        \draw[ultra thick,->] (0,0) -- (0, 1);
        \draw[ultra thick] (-0.03,0.3) -- (0.03,0.3);
        \node[above right] at (0,0.35) {$\reg$};
        \node[right] at (0.85,0) {$\pvi{}$};
        \node[above] at (0,1) {\small{$\pertfunc(\pvi{}) = \regone\abs{\pvi{}}$}};
        \draw[cyan,ultra thick] (-0.5,0.3) plot[domain=0.05:0.8] (\x,0.3+0.5*abs{\x});
        \draw[cyan,ultra thick] (-0.5,0.3) plot[domain=-0.8:-0.05] (\x,0.3-0.5*\x) node {};
        \draw[cyan,ultra thick] (0,0.3) circle (0.05);
        \draw[orange,ultra thick] (0, 0) -- (0.85, 0.425);
        \draw[orange,ultra thick] (0, 0) -- (-0.85, 0.425);
        \draw[thick,densely dashed] (0, 0) -- (0.85, 0.425);
        \node[below] at (0.4,0) {\scriptsize{$\pertlimit \rightarrow +\infty$}};
        \node[rotate=27,anchor=center] at (0.63,0.23) {\scriptsize{slope $\pertslope = \regone$}};
        \fill[cyan] (0,0) circle (0.03);
    \end{scope}
\end{tikzpicture}
    \caption{Graph of the functions \(\regfunc\) (blue) and \(\biconj{\regfunc}\) (orange) for different choices of function \(\pertfunc\).}
    \label{fig:relaxregfunc}
\end{figure}

As shown by \Cref{lem:pertlimite > 0} in \Cref{app:characterization_of_conj_pertfunc}, the quantity \(\pertlimit\) is always strictly positive under our working assumptions.
Consequently, we deduce from \Cref{prop:regfunc-biconj} that \(\biconj{\regfunc}\) behaves like a rescaled absolute-value function over the non-empty interval \(\kintervcc{-\pertlimit}{\pertlimit}\).
Elsewhere, we have \(\biconj{\regfunc{}} = \regfunc{}\).\footnote{The functions \(\biconj{\regfunc{}}\) and \(\regfunc{}\) may differ for any \(\pvi{} \neq 0\), as in the third example of \Cref{fig:relaxregfunc}.} 
These observations lead to the following characterization of the subdifferential of $\biconj{\regfunc{}}$:
\begin{proposition}
    \label{prop:subdiff-gbiconj}
    If~\ref{assumption:zero-minimized}-\ref{assumption:closed}-\ref{assumption:convex}-\ref{assumption:even} hold, then\footnote{We remind the reader that $\dom{\biconj{\regfunc}}=\dom{\pertfunc{}}$ from \Cref{prop:regfunc-biconj}.}
        \begin{equation} \label{eq:operators-biconj-subdiff}
            \forall \pvi{}\in\dom{\biconj{\regfunc}}:\ 
              \subdiff\biconj{\regfunc}(\pvi{})
            =
            \begin{cases}
                [-\pertslope,\pertslope]   & \mbox{if \(|\pvi{}|=0\)}\\
                \sign(\pvi{})\, \{\pertslope\}  & \mbox{if \(|\pvi{}|\in\kintervoo{0}{\pertlimit}\)}\\
                \sign(\pvi{})\,[\pertslope,\pertrightlope] & \mbox{if \(|\pvi{}|=\pertlimit\)}\\
                \subdiff\pertfunc(\pvi{}) & \mbox{if \(|\pvi{}|>\pertlimit\)}
            \end{cases}
        \end{equation}
\end{proposition}
A proof of \Cref{prop:subdiff-gbiconj} is provided in \Cref{proof:subdiff-g**}.
Interestingly, this result highlights that the subdifferential of \(\biconj{\regfunc{}}\) admits a simple analytical expression in terms of parameters \(\pertslope{}\), \(\pertlimit{}\), and \(\pertrightlope{}\).
Specifically, for \(|\pvi{}| < \pertlimit{}\), the subdifferential corresponds to that of a rescaled absolute-value function, while for \(|\pvi{}| > \pertlimit{}\), the subdifferential of \(\biconj{\regfunc{}}\) coincides with that of \(\pertfunc{}\).
This shows that the subdifferential of \(\biconj{\regfunc{}}\) can be evaluated with similar numerical complexity as that of \(\pertfunc{}\).

A similar conclusion can be drawn for the proximal operator of $\biconj{\regfunc{}}$:
\begin{proposition} \label{prop:operator-prox-gbiconj}
    Let \(\cstprox > 0\).
    If \ref{assumption:zero-minimized}-\ref{assumption:closed}-\ref{assumption:convex}-\ref{assumption:even} hold, then
    \begin{equation} \label{eq:prop:operator-prox-gbiconj}
        \forall \pvi{}\in \kR:\ 
        \prox_{\cstprox\biconj{\regfunc}}(\pvi{}) = 
        \begin{cases}
            0 & \text{if} \ \abs{\pvi{}} \in [0,\cstprox\pertslope] \\
            \pvi{} - \cstprox\pertslope\sign(\pvi{}) & \text{if} \ \abs{\pvi{}} \in \ ]\cstprox\pertslope, \cstprox\pertslope + \pertlimit] \\
            \prox_{\cstprox\pertfunc}(\pvi{}) & \text{if} \ \abs{\pvi{}} \in \ ]\cstprox\pertslope + \pertlimit,+\infty[
            .
        \end{cases}
    \end{equation}
\end{proposition}
A proof of \Cref{prop:operator-prox-gbiconj} is available  in \Cref{proof:operator-prox-g**}. 
As with our previous results, the proximal operator of \(\biconj{\regfunc{}}\) only depends on the parameters \(\pertslope{}\) and \(\pertlimit{}\). It can be directly computed from the expression in \Cref{prop:operator-prox-gbiconj}, provided that the proximal operator of \(\pertfunc{}\) is known.

\subsection{Characterization of \texorpdfstring{\(\conj{\regfunc}\)}{\expandafter\regfunc*}}
\label{sec:characterisation-g*}

In this subsection, we provide simple closed-form expressions for the conjugate function \(\conj{\regfunc}\) and its subdifferential/proximal operators. 
Our first result reads as follows:
\begin{proposition} \label{prop:regfunc-conj}
    If~\ref{assumption:zero-minimized}-\ref{assumption:closed}-\ref{assumption:convex} hold, then $\dom{\conj{\regfunc{}}}=\dom{\conj{\pertfunc{}}}$ and
    \begin{equation} \label{eq:regfunc-conj}
        \forall \bvi{}\in\kR:\quad \conj{\regfunc}(\bvi{}) = \pospart{\conj{\pertfunc}(\bvi{}) - \reg}
        .
    \end{equation}
    If~\ref{assumption:even} moreover holds, then \(\conj{\regfunc}(\bvi{}) = 0\) \text{if and only if} \(\abs{\bvi{}} \leq \pertslope\).
\end{proposition}    
A proof of \Cref{prop:regfunc-conj} is available in \Cref{proof:regfunc-conj}. 
The relationship between $\conj{\regfunc}$ and $\conj{\pertfunc}$ emphasized in this result in turn leads to the following simple characterization of the subdifferential and proximal operators:
\begin{proposition} \label{prop:subdiff-conjugate of g}
    Assume~{\ref{assumption:zero-minimized}-\ref{assumption:closed}-\ref{assumption:convex}-\ref{assumption:even}} hold.
    If \(\pertfunc\) is identically zero, then 
    \begin{equation}\label{eq:subdiff-g*-h=0} 
        \dom{\conj{\regfunc}}=\{0\} \ \text{and} \ \subdiff\conj{\regfunc}(0) = \kR.
    \end{equation}
   If \(\pertfunc\) is not identically zero, then 
    \begin{equation} \label{eq:subdiff-g*}
        \forall \bvi{} \in \dom{\conj{\regfunc}}:\ 
        \subdiff\conj{\regfunc}(\bvi{}) = 
        \begin{cases}
            \{0\} & \text{if} \ \abs{\bvi{}} < \pertslope \\
            \sign(\bvi{})\, [0,\pertlimit] \cap\kR  & \text{if} \ \abs{\bvi{}} = \pertslope \\
            \subdiff\conj{\pertfunc}(\bvi{}) & \text{if} \ \abs{\bvi{}} > \pertslope.
        \end{cases}
    \end{equation}
\end{proposition}

\begin{proposition} \label{prop:operator-prox-gstar}
    Let \(\cstprox > 0\). 
    If \ref{assumption:zero-minimized}-\ref{assumption:closed}-\ref{assumption:convex}-\ref{assumption:even} hold, then 
    \begin{equation} \label{eq:prop:operator-prox-gstar}
        \forall \bvi{} \in \kR:
        \prox_{\cstprox\conj{\regfunc}}(\bvi{}) = 
        \begin{cases}
            \bvi{}  
            &\text{if} \ \abs{\bvi{}} \in [0,\pertslope] \\
            \pertslope\sign(\bvi{}) 
            &\text{if} \ \abs{\bvi{}} \in \ ]\pertslope, \pertslope+\cstprox\pertlimit] \\
            \prox_{\cstprox\conj{\pertfunc}}(\bvi{}) 
            &\text{if} \ \abs{\bvi{}} \in \ ]\pertslope+\cstprox\pertlimit,+\infty[.
        \end{cases}
    \end{equation}
\end{proposition}
We refer the reader to \Cref{proof:subdiff-g*,proof:operator-prox-g*} for a proof of these results.  
A similar conclusion can be drawn as in the previous subsection: once the key parameters \(\pertslope{}\) and \(\pertlimit{}\) have been determined, the subdifferential and proximal operators of \(\conj{\regfunc}\) can be directly obtained from those of \(\conj{\pertfunc}\).
\Cref{prop:subdiff-conjugate of g,prop:operator-prox-gstar} thus provide a simple and valuable way of deriving the subdifferential and proximal operators of \(\conj{\regfunc}\) for any function \(\pertfunc{}\) that satisfies our working assumptions with a tractable expression for these operators.


\section{Numerical Experiments}
\label{sec:numerics}

In this section, we assess the performance of \elops{}, a \textsc{Python} toolbox accompanying this paper.
Leveraging the theoretical developments presented in \Cref{sec:bnb,sec:implementation}, it implements a generic \gls{bnb} solver tailored to any instance of problem \eqref{prob:prob} verifying assumptions \ref{assumption:f}-\ref{assumption:zero-minimized}-\ref{assumption:closed}-\ref{assumption:convex}-\ref{assumption:coercive}-\ref{assumption:even}.
Its code is available at 
\begin{center}
    \href{https://github.com/TheoGuyard/El0ps}{\texttt{https://github.com/TheoGuyard/El0ps}}
\end{center}
We refer the reader to the paper \cite{guyard2025el0ps} accompanying our toolbox for an in-depth description of its implementation. 
Some supplementary material to our numerical experiments is also provided in \Cref{sec:supp-numerics}.

\subsection{Concurrent Solvers}

Throughout this section, we compare \elops{} with several state-of-the-art methods previously proposed in the literature.
\begin{itemize}
    \item First, we consider \cplex{} \cite{cplex2009v12} and \mosek{} \cite{mosek}, two commercial solvers able to address a wide variety of \glsplural{mip}. \cplex{} supports linear and quadratic expressions whereas \mosek{} can additionally handle conic ones. Both solvers are implemented in \textsc{C/C++} and leverage various techniques such as branch-and-bound, cutting planes, heuristics, and presolving to efficiently explore the solution space.
    These solvers can handle instances of problem \eqref{prob:prob} for particular expressions of the function \(\pertfunc{}\) through the \gls{mip} formulations proposed in \cite{bourguignon2015exact,hazimeh2021sparse,pilanci2015sparse} and reminded in \Cref{sec:supp-numerics:mip}. 
    \item Second, we consider the \oa{} method introduced in \cite{bertsimas2021unified} aiming to enhance \gls{mip} solvers. Instead of directly tackling the entire \gls{mip} formulation of the problem, it considers a sequence of piecewise linear approximations.
    Since no publicly available implementation of this procedure exists, we use our own \textsc{Python} implementation, see \Cref{sec:supp-numerics:oa}. 
    \item Finally, we consider \lobnb{} \cite{hazimeh2022l0learn}, a specialized \gls{bnb} solver for problem \eqref{prob:prob} implemented in \textsc{Python}. The latter can handle instances of \eqref{prob:prob} where $\datafunc$ is a least-squares function and where $\pertfunc{}(\pvi{})=\icvx(|\pvi{}| \leq \bigM{}) + \tfrac{\regone}{2}\pvi{}^2$ for some $\bigM \in [0,+\infty]$ and $\regone\geq 0$, with at least $\bigM < +\infty$ or $\regone > 0$.
    When restricted to this setup, \elops{} and \lobnb{} implement a similar \gls{bnb} backbone and bounding strategy, but differ in the way they explore regions in the feasible space. Moreover, \elops{} implements additional acceleration strategies proposed in \cite{pmlr-v235-guyard24a,samain2022techniques}.
\end{itemize}
We note that \elops{} can handle all instances of problem \eqref{prob:prob} that were previously addressed by existing methods in the literature.
Moreover, the new theoretical contributions presented in this paper allow to tackle instances beyond these setups, thereby expanding the scope of application of problem \eqref{prob:prob}. We illustrate the versatility and efficiency of \elops{} below by addressing several problems of interest in the fields of machine learning and signal processing. 
 
\subsection{Machine Learning: Feature Selection Tasks}
\label{sec:numerics:regpath}

In this experiment, we focus on feature selection tasks \cite[\sectionInRef{2.3}]{li2017feature} that arise in machine learning applications and require solving instances of problem \eqref{prob:prob}.
Specifically, we consider three distinct feature selection problems associated with a given function $\datafunc$ and, for each, select two datasets providing a feature matrix \(\dic \in \kR^{\ddim\times\pdim}\) and a target vector \(\obs \in \kR^{\ddim}\) or \(\obs \in \{-1,+1\}^{\ddim}\):

\begin{itemize}
    \item \textit{Least-squares regression} with the loss $\datafunc(\dic\pv) = \tfrac{1}{2}\norm{\obs - \dic\pv}{2}^2$:
    We use the \textsc{Riboflavin} \cite{buhlmann2014high} and \textsc{Bctcga} \cite{liu2018integrated} datasets where $\obs \in \kR^{\ddim}$ that are related to vitamin production and cancer screening, respectively.
    \item \textit{Logistic binary classification} with the loss $\datafunc(\dic\pv) = \transp{\1}\log(\1 + \exp(-\obs \odot \dic\pv))$:  
    We use the \textsc{Colon cancer} \cite{alon1999broad} and \textsc{Leukemia} \cite{golub1999molecular} datasets where \(\obs \in \{-1,+1\}^{\ddim}\) that are both related to cancer screening.  
    \item \textit{SVM binary classification} with the loss $\datafunc(\dic\pv) = \norm{\pospart{\1 - \obs \odot \dic\pv}}{2}^2$:
    We use the \textsc{Breast cancer} \cite{chang2011libsvm} and \textsc{Arcene} \cite{guyon2003design} datasets where \(\obs \in \{-1,+1\}^{\ddim}\) that are related to tumor categorization and DNA analysis, respectively.
\end{itemize}
The dimensions of each dataset are provided in \Cref{table:numerics:regpath-datasets}.
For each family of problems, we consider the following two penalty functions:
 \begin{subequations}
	\begin{alignat}{4}
        \label{eq:numerics:regpath:bigm-l2norm}
		\pertfunc(\pvi{}) &= \tfrac{\regone}{2}\pvi{}^2 &&+ \icvx(\abs{\pvi{}} \leq \bigM) &&\quad\forall \pvi{} \in \kR \\
		\label{eq:numerics:regpath:bigm-l1norm}
		\pertfunc(\pvi{}) &= \regone\abs{\pvi{}} &&+ \icvx(\abs{\pvi{}} \leq \bigM) &&\quad\forall \pvi{} \in \kR 
	\end{alignat}
\end{subequations}
where \(\regone>0\) and \(\bigM>0\).
These choices are motivated by the statistical properties of the resulting solutions \cite[\sectionInRef{4}]{dedieu2021learning}.

\begin{table}[t]
    \begin{center}
    \begin{tabular}{lccc}
        \toprule
        Dataset & $\ddim$ & $\pdim$ \\ 
        \midrule
        \textsc{Riboflavin} & 71 & 4,088 \\
        \textsc{Bctcga} & 536 & 17,322 \\
        \textsc{Colon cancer} & 62 & 2,000 \\
        \textsc{Leukemia} & 38 & 7,129 \\
        \textsc{Breast cancer} & 44 & 7,129 \\
        \textsc{Arcene} & 100 & 10,000 \\
        \bottomrule
    \end{tabular}
    \caption{Description of the dimensions of the datasets used in our numerical simulations.}
    \label{table:numerics:regpath-datasets}
    \end{center}   
\end{table}

Our experiments are conducted as follows.
For each dataset, we first calibrate the parameters \(\regone\) and \(\bigM\) by using the cross-validation procedure detailed in \Cref{sec:supp-numerics:calibration}.
We then construct a regularization path \cite{friedman2010regularization}, meaning that we solve the resulting instance of problem \eqref{prob:prob} for different values of \(\reg\) to generate a pool of solutions with varying sparsity levels. The process starts at some value \(\regmax\) such that the solution to \eqref{prob:prob} is the all-zero vector.\footnote{For example, it can be shown that any $\lambda$ satisfying \(\norm{\transp{\dic}\grad\datafunc(\0)}{\infty} \leq \pertslope\) (where \(\pertslope\) depends on $\lambda$, see \Cref{table:examples:intro}) leads to an instance of problem \eqref{prob:prob} for which the all-zero vector is a minimizer.}
Next, we iterate over a grid of 20 logarithmically spaced values of \(\reg\), ranging from \(\reg_{\max}\) to \(10^{-2} \times \regmax\). The process terminates either when the entire grid has been explored or when no solver completes within 10 minutes for some \(\reg\). 
The solution obtained for a given \(\reg\) is used as a warm start for the next one considered in the regularization path.

In \Cref{fig:numerics:regpath}, we show the average time required by each method to solve the problem for each value of \(\reg\) during the regularization path construction.\footnote{To ensure a reliable assessment of the computational performance of the numerical procedures, we ran each procedure 10 times on the same problem instance and averaged the results. This approach minimizes the influence of external variability, such as system load fluctuations, providing a more robust estimate of the actual running time.} 
We note that \lobnb{} is only tailored to address the least-square regression problem with penalty function \eqref{eq:numerics:regpath:bigm-l2norm}. The performance of this solver (orange curve) thus only appears in the two ``top-left'' figures. Moreover, \cplex{}  cannot handle logistic data loss and its performance (blue curve) is only available for the least-squares regression and SVM binary classification problems. 
 
We observe that \elops{} allows for substantial time savings as compared to the generic solvers \cplex{} and \mosek.
As far as our simulation setups are concerned, we noticed an acceleration factor varying between two and three in average, sometimes reaching up to four orders of magnitude.   
We note that, prior to this work, all the problem instances considered here could only be solved by generic solvers, with the exception of the least-square regression problem with penalty function \eqref{eq:numerics:regpath:bigm-l2norm}. 
\elops{} thus paves the way to the application of $\ell_0$-based regularizers to a wider range of machine-learning problems.
We also observe that \elops{} outperforms the specialized \gls{bnb} solver \lobnb{} on the instances it can handle. On average, our solver achieves an acceleration factor ranging from one to two orders of magnitude. 

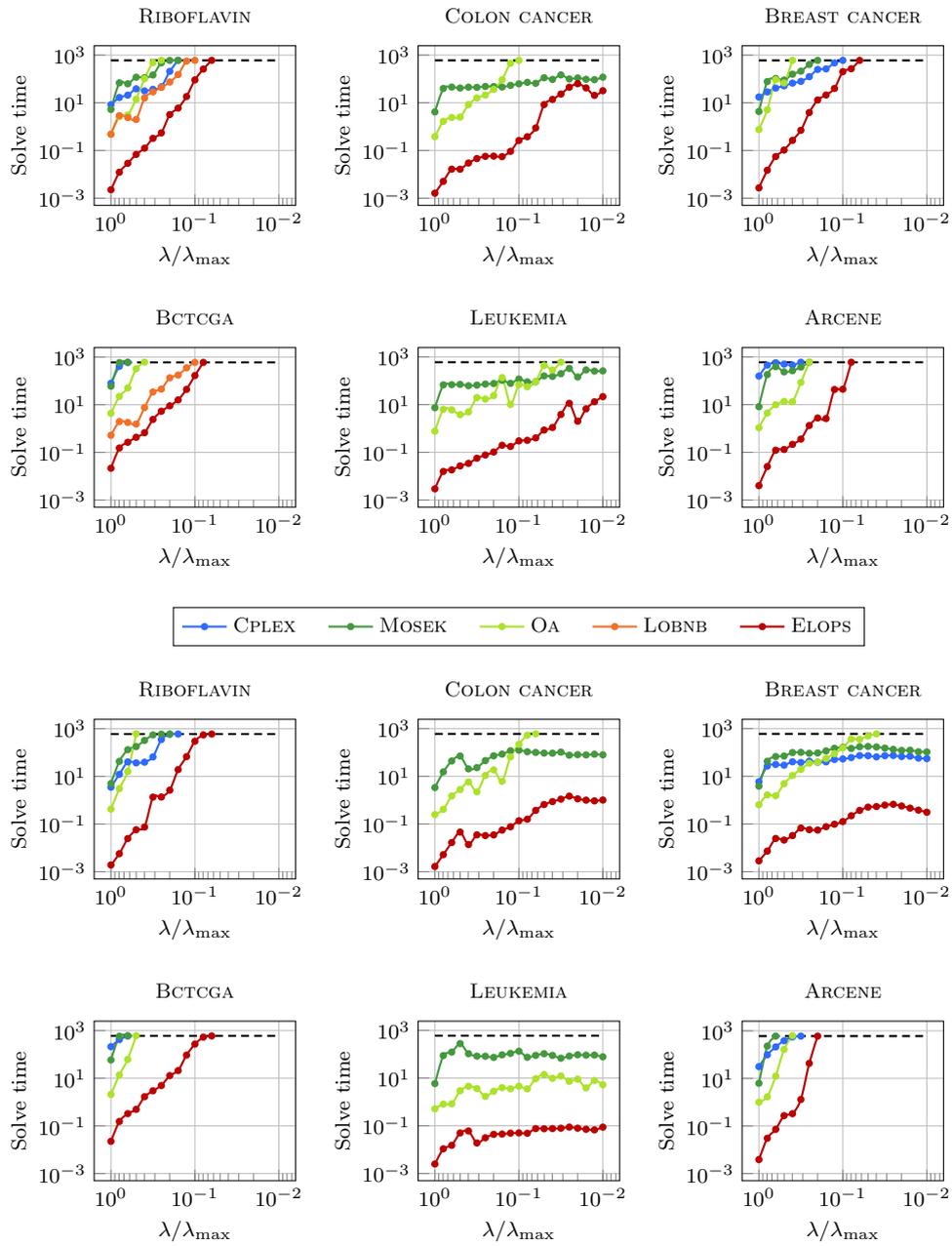
\begin{figure}[!t]
    \centering
    \pgfplotscreateplotcyclelist{cycle_list_regpath}{
    cplexcolor, smooth, thick, mark=*, mark options={scale=0.5}\\     
    mosekcolor, smooth, thick, mark=*, mark options={scale=0.5}\\
    oacolor, smooth, thick, mark=*, mark options={scale=0.5}\\
    l0bnbcolor, smooth, thick, mark=*, mark options={scale=0.5}\\
    el0pscolor, smooth, thick, mark=*, mark options={scale=0.5}\\
    black, smooth, thick, densely dashed\\
}

\begin{tikzpicture}
    \begin{groupplot}[
        group style     = {
            group size      = 3 by 2,
            vertical sep    = 55pt,
            horizontal sep  = 47pt,
        },
        height          = 3.75cm,
        width           = 4.3cm,
        ytick pos       = left,
        xtick pos       = bottom,
        grid            = major,
        xlabel          = $\reg/\regmax$,
        xmode           = log,
        x dir           = reverse,
        ylabel          = Solve time,
        ymode           = log,
        ymin            = 0.0005,
        ytick           = {0.001,0.1,10,1000},
        ylabel style    = {at={(-0.3,0.5)}},
        cycle list name = cycle_list_regpath,
    ]


        \nextgroupplot[
            title           = \textsc{Riboflavin},
            legend to name  = regpath_legend,
            legend style    = {
                legend columns = 6,
                /tikz/every even column/.style = {column sep=10pt},
            },
        ]
        \foreach \solver in {cplex,mosek,mosek_oa,l0bnb,el0ps,max}{
            \addplot table[
                x       = lgrid, 
                y       = \solver_solve_time,
                col sep = comma,
            ] {img/regpath/regpath-dataset=riboflavin-datafit=Leastsquares-penalty=BigmL2norm.csv};
        }

        \addlegendentry{\textsc{Cplex}}
        \addlegendentry{\textsc{Mosek}}
        \addlegendentry{\textsc{Oa}}
        \addlegendentry{\textsc{Lobnb}}
        \addlegendentry{\textsc{Elops}}


        \nextgroupplot[title = \textsc{Colon cancer}]
        \foreach \solver in {cplex,mosek,mosek_oa,l0bnb,el0ps,max}{
            \addplot table[
                x       = lgrid, 
                y       = \solver_solve_time,
                col sep = comma,
            ] {img/regpath/regpath-dataset=colon-cancer-datafit=Logistic-penalty=BigmL2norm.csv};
        }


        \nextgroupplot[title = \textsc{Breast cancer}]
        \foreach \solver in {cplex,mosek,mosek_oa,l0bnb,el0ps,max}{
            \addplot table[
                x       = lgrid, 
                y       = \solver_solve_time,
                col sep = comma,
            ] {img/regpath/regpath-dataset=breast-cancer-datafit=Squaredhinge-penalty=BigmL2norm.csv};
        }


        \nextgroupplot[title = \textsc{Bctcga}]
        \foreach \solver in {cplex,mosek,mosek_oa,l0bnb,el0ps,max}{
            \addplot table[
                x       = lgrid, 
                y       = \solver_solve_time,
                col sep = comma,
            ] {img/regpath/regpath-dataset=bctcga-datafit=Leastsquares-penalty=BigmL2norm.csv};
        }


        \nextgroupplot[title = \textsc{Leukemia}]
        \foreach \solver in {cplex,mosek,mosek_oa,l0bnb,el0ps,max}{
            \addplot table[
                x       = lgrid, 
                y       = \solver_solve_time,
                col sep = comma,
            ] {img/regpath/regpath-dataset=leukemia-datafit=Logistic-penalty=BigmL2norm.csv};
        }


        \nextgroupplot[title = \textsc{Arcene}]
        \foreach \solver in {cplex,mosek,mosek_oa,l0bnb,el0ps,max}{
            \addplot table[
                x       = lgrid, 
                y       = \solver_solve_time,
                col sep = comma,
            ] {img/regpath/regpath-dataset=arcene-datafit=Squaredhinge-penalty=BigmL2norm.csv};
        }

    \end{groupplot}
    \path (group c1r2.south east) -- node[below,yshift=-1.25cm]{\ref*{regpath_legend}} (group c3r2.south west);
\end{tikzpicture}

\vspace*{0.25cm}

\begin{tikzpicture}
    \begin{groupplot}[
        group style     = {
            group size      = 3 by 2,
            vertical sep    = 55pt,
            horizontal sep  = 47pt,
        },
        height          = 3.75cm,
        width           = 4.3cm,
        ytick pos       = left,
        xtick pos       = bottom,
        grid            = major,
        xlabel          = $\reg/\regmax$,
        xmode           = log,
        x dir           = reverse,
        ylabel          = Solve time,
        ymode           = log,
        ymin            = 0.0005,
        ytick           = {0.001,0.1,10,1000},
        ylabel style    = {at={(-0.3,0.5)}},
        cycle list name = cycle_list_regpath,
    ]


        \nextgroupplot[title = \textsc{Riboflavin}]
        \foreach \solver in {cplex,mosek,mosek_oa,l0bnb,el0ps,max}{
            \addplot table[
                x       = lgrid, 
                y       = \solver_solve_time,
                col sep = comma,
            ] {img/regpath/regpath-dataset=riboflavin-datafit=Leastsquares-penalty=BigmL1norm.csv};
        }


        \nextgroupplot[title = \textsc{Colon cancer}]
        \foreach \solver in {cplex,mosek,mosek_oa,l0bnb,el0ps,max}{
            \addplot table[
                x       = lgrid, 
                y       = \solver_solve_time,
                col sep = comma,
            ] {img/regpath/regpath-dataset=colon-cancer-datafit=Logistic-penalty=BigmL1norm.csv};
        }


        \nextgroupplot[title = \textsc{Breast cancer}]
        \foreach \solver in {cplex,mosek,mosek_oa,l0bnb,el0ps,max}{
            \addplot table[
                x       = lgrid, 
                y       = \solver_solve_time,
                col sep = comma,
            ] {img/regpath/regpath-dataset=breast-cancer-datafit=Squaredhinge-penalty=BigmL1norm.csv};
        }


        \nextgroupplot[title = \textsc{Bctcga}]
        \foreach \solver in {cplex,mosek,mosek_oa,l0bnb,el0ps,max}{
            \addplot table[
                x       = lgrid, 
                y       = \solver_solve_time,
                col sep = comma,
            ] {img/regpath/regpath-dataset=bctcga-datafit=Leastsquares-penalty=BigmL1norm.csv};
        }


        \nextgroupplot[title = \textsc{Leukemia}]
        \foreach \solver in {cplex,mosek,mosek_oa,l0bnb,el0ps,max}{
            \addplot table[
                x       = lgrid, 
                y       = \solver_solve_time,
                col sep = comma,
            ] {img/regpath/regpath-dataset=leukemia-datafit=Logistic-penalty=BigmL1norm.csv};
        }


        \nextgroupplot[title = \textsc{Arcene}]
        \foreach \solver in {cplex,mosek,mosek_oa,l0bnb,el0ps,max}{
            \addplot table[
                x       = lgrid, 
                y       = \solver_solve_time,
                col sep = comma,
            ] {img/regpath/regpath-dataset=arcene-datafit=Squaredhinge-penalty=BigmL1norm.csv};
        }

    \end{groupplot}
\end{tikzpicture}
    \caption{Average solving times of different methods for regularization path construction. Dots represent the considered values of $\reg$. The black dashed line indicates the maximum time budget allowed for solving a problem instance. Top chart: penalty \eqref{eq:numerics:regpath:bigm-l2norm}. Bottom chart: penalty \eqref{eq:numerics:regpath:bigm-l1norm}.}
    \label{fig:numerics:regpath}
\end{figure}

\subsection{Signal Processing: Bernoulli Mixtures Models}
\label{sec:numerics:mixtures}

A standard problem in signal processing consists in recovering some unknown vector 
$\groundtruth\in\kR^\pdim$ from partial/noisy observations: 
\begin{equation}
    \label{eq:numerics:mixtures:model}
    \obs = \dic\groundtruth + \noise
\end{equation}
where $\obs \in \kR^{\ddim}$, $\dic \in \kR^{\ddim\times\pdim}$ and $\noise \sim \normaldistrib (\0,\noisestd\identitymatrix_{\ddim})$ is a gaussian noise with variance $\noisestd > 0$. 
A widely explored assumption in the literature considers the components of \(\groundtruth\) as independent realizations of a Bernoulli mixture model, that is 
\begin{equation}
    \label{eq:numerics:mixtures:mixture}
    \groundtruthi{\idxentry} = \bvi{\idxentry} \amplitudeveci{\idxentry}
\end{equation}
where $\bvi{\idxentry} \in \{0,1\}$ is governed by a Bernoulli distribution with $\proba(\bvi{\idxentry}=1)=\bernoulliparam$ and  $\amplitudeveci{\idxentry} \in \kR^{\pdim}$ follows some law admitting a density function $\kfuncdef{\densityfunc}{\kR}{\kR+}$. 
This type of problems for example occurs in compressive sensing \cite[\sectionInRef{II}]{soussen2011bernoulli},\cite[\sectionInRef{2}]{Herzet_eusipco10}, 
electro encephalography reconstruction \cite[\sectionInRef{3.3}]{tourneret2013sparse},
Bayesian inference \cite[\sectionInRef{3.4.4}]{rover2021weakly},
microscopy applications \cite[\sectionInRef{3}]{dobigeon2009hierarchical} or 
magnetic resonance imaging \cite{chaari2014hierarchical}. \Cref{table:numerics:mixtures} gives the densities $\densityfunc$ considered in these works, along with their mathematical expressions (up to a normalization factor).

\begin{table}[!t]
    \setlength{\tabcolsep}{7pt}
    \begin{center}
    \begin{tabular}{ccl}
        \toprule
        Works & Distribution &  \multicolumn{1}{c}{$\densityfunc$} \\ 
        \midrule
        \cite{soussen2011bernoulli,Herzet_eusipco10}& \textsc{Normal}  & $\exp(-\tfrac{\pvi{}^2}{2\gamma^2})$ \\
        \cite{tourneret2013sparse}& \textsc{Laplace}  & $\exp(-\tfrac{\abs{\pvi{}}}{\gamma})$ \\
        \cite{dobigeon2009hierarchical}
        & \textsc{Exponential}  & $\exp(-\tfrac{\abs{\pvi{}}}{\gamma})\1(\pvi{} \geq 0)$ \\
        \cite{rover2021weakly}
        & \textsc{Half-Normal}  & $\exp(-\tfrac{\pvi{}^2}{2\gamma^2})\1(\pvi{} \geq 0)$ \\
        \cite{chaari2014hierarchical}& \textsc{Gauss-Laplace}  & $\exp(-\tfrac{\abs{\pvi{}}}{\gamma'}-\tfrac{\pvi{}^2}{2\gamma^2})$ \\
        \bottomrule
    \end{tabular}
    \caption{Densities $\densityfunc$ appearing in the Bayesian models considered in different works.}
    \label{table:numerics:mixtures}
    \end{center}
\end{table}

Given this probabilistic model, it can be shown (see \eg{}, \cite[\sectionInRef{II}]{soussen2011bernoulli}) that any \gls{map} estimate of $\groundtruth$ corresponds to a solution of problem \eqref{prob:prob} with the following definitions: 
\begin{subequations}
    \begin{alignat}{4}
        \datafunc(\dic\pv) &= \tfrac{1}{2}\norm{\obs - \dic\pv}{2}^2 &&\quad\forall \pv \in \kR^{\pdim}  \\
        \pertfunc(\pvi{}) &= -\noisestd^2\log \densityfunc(\pvi{}) &&\quad\forall \pvi{} \in \kR  \\
        \reg &= \noisestd^2\log(\tfrac{1 - \bernoulliparam}{\bernoulliparam})
        .
    \end{alignat}
\end{subequations}

Hereafter, we illustrate the ability of different procedures to solve instances of this problem for various choices of $\densityfunc$. 
More specifically, for each density listed in \Cref{table:numerics:mixtures}, 
we generate 100 instances of problem \eqref{prob:prob} and assess the number of instances that each solver can solve to optimality within a given time budget. 
The parameters $\obs$ and $\dic$ defining each problem instance are generated randomly as follows. 
Each row of matrix $\dic \in \kR^{\ddim\times\pdim}$ is drawn independently from a multivariate normal distribution $\normaldistrib(\0,\corrmat)$, where $\corrmatel_{ij}=\corrparam^{\abs{i-j}}$ for some $\corrparam \in [0,1]$. 
For large values of \(\corrparam{}\), the resulting matrix \(\dic\) has thus highly-correlated columns, making problem \eqref{prob:prob} particularly challenging.
$\obs \in \kR^{\ddim}$ obeys model \eqref{eq:numerics:mixtures:model} with the distributions on $\groundtruth$ and $\noise$ previously described. In our simulations, we consider the following values $(\ddim,\pdim,\bernoulliparam, \corrparam, \gamma,\gamma')= (500,1000,0.01,0.9,1,1)$. The noise standard deviation is set to $\noisestd{}=\|\dic\groundtruth{}\|_2/\sqrt{10\ddim{}}$ to have a signal-to-noise ratio equal to $10$.
 
\Cref{fig:numerics:mixtures} reports the proportion of problem instances solved to optimality by the considered solvers as a function of the time budget expressed in seconds.
We first note that \elops{} is the only procedure able to address the problems involving 
the \textsc{Laplace} and the \textsc{Exponential} densities. 
This explains why only \elops{} performance is displayed in the corresponding figures. We notice that for these setups,  \elops{} is able to solve any problem instance in less than 100 seconds. 
In the \textsc{Gaussian}, \textsc{Laplace}, and \textsc{Half-normal} setups, \elops{} significantly outperforms the other solvers.
Notably, it achieves at least one order of magnitude speedup in solving time for a given target proportion of instances solved, with this gain reaching up to two orders of magnitude in most cases.

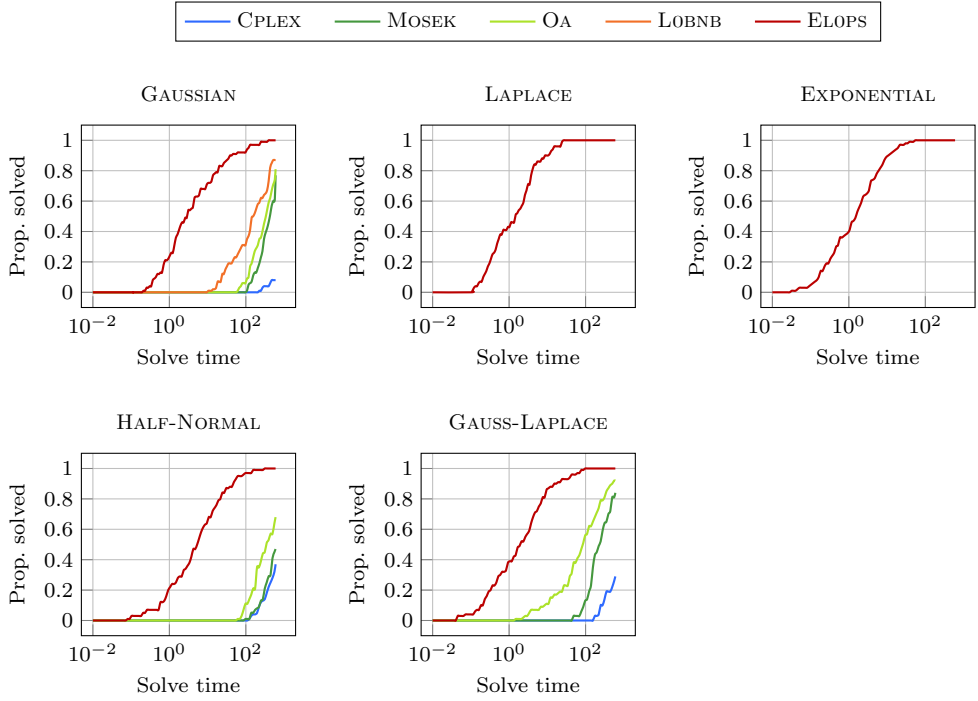
\begin{figure}[!t]
    \centering
    \pgfplotscreateplotcyclelist{cycle_list_mixtures}{
    cplexcolor, smooth, thick\\    
    mosekcolor, smooth, thick\\
    oacolor, smooth, thick\\
    l0bnbcolor, smooth, thick\\
    el0pscolor, smooth, thick\\
}

\begin{tikzpicture}
    \begin{groupplot}[
        group style     = {
            group size      = 3 by 2,
            vertical sep    = 55pt,
            horizontal sep  = 47pt,
        },
        height          = 4.0cm,
        width           = 4.4cm,
        ytick pos       = left,
        xtick pos       = bottom,
        grid            = major,
        xlabel          = Solve time,
        xmode           = log,
        ylabel          = Prop. solved,
        ytick           = {0,20,...,100},
        yticklabel      = {\pgfmathparse{\tick/100}\pgfmathprintnumber{\pgfmathresult}},
        ymin            = -10,
        ymax            = 110,
        xmax            = 2000,
        xmin            = 0.005,
        cycle list name = cycle_list_mixtures,
    ]


        \nextgroupplot[
            title           = \textsc{Gaussian},
            legend to name  = mixtures_legend,
            legend style    = {
                legend columns = 6,
                /tikz/every even column/.style = {column sep=10pt},
            },
        ]
        \foreach \solver in {cplex,mosek,mosek_oa,l0bnb,el0ps}{
            \addplot table[
                x       = tgrid, 
                y       = \solver,
                col sep = comma,
            ] {img/mixtures/mixtures-k=10-m=500-n=1000-r=0.90-s=10.00-distrib=gaussian-scale=1.csv};
        }

        \addlegendentry{\cplex}
        \addlegendentry{\mosek}
        \addlegendentry{\oa}
        \addlegendentry{\lobnb}
        \addlegendentry{\elops}


        \nextgroupplot[title = \textsc{Laplace}]
        \foreach \solver in {cplex,mosek,mosek_oa,l0bnb,el0ps}{
            \addplot table[
                x       = tgrid, 
                y       = \solver,
                col sep = comma,
            ] {img/mixtures/mixtures-k=10-m=500-n=1000-r=0.90-s=10.00-distrib=laplace-scale=1.csv};
        }


        \nextgroupplot[title = \textsc{Exponential}]
        \foreach \solver in {cplex,mosek,mosek_oa,l0bnb,el0ps}{
            \addplot table[
                x       = tgrid, 
                y       = \solver,
                col sep = comma,
            ] {img/mixtures/mixtures-k=10-m=500-n=1000-r=0.90-s=10.00-distrib=exponential-scale=1.csv};
        }


        \nextgroupplot[title = \textsc{Half-Normal }]
        \foreach \solver in {cplex,mosek,mosek_oa,l0bnb,el0ps}{
            \addplot table[
                x       = tgrid, 
                y       = \solver,
                col sep = comma,
            ] {img/mixtures/mixtures-k=10-m=500-n=1000-r=0.90-s=10.00-distrib=half-gaussian-scale=1.csv};
        }


        \nextgroupplot[title = \textsc{Gauss-Laplace}]
        \foreach \solver in {cplex,mosek,mosek_oa,l0bnb,el0ps}{
            \addplot table[
                x       = tgrid, 
                y       = \solver,
                col sep = comma,
            ] {img/mixtures/mixtures-k=10-m=500-n=1000-r=0.90-s=10.00-distrib=gauss-laplace-scale1=1-scale2=1.csv};
        }

    \end{groupplot}

    \path (group c1r1.north east) -- node[above,yshift=1cm]{\ref*{mixtures_legend}} (group c3r1.north west);
\end{tikzpicture}
    \caption{Proportion of the 100 instances solved within a given time budget expressed in seconds, with a maximum time limit of 10 minutes. A higher curve indicates a better performance.}
    \label{fig:numerics:mixtures}
\end{figure}


\section{Conclusion}
\label{sec:conclusion}

In this paper, we proposed a generic \glsreset{bnb}\gls{bnb} framework for solving $\ell_0$-penalized optimization problems under a set of general assumptions on the loss and penalty functions. Our theoretical analysis establishes closed-form expressions for all key quantities required in the \gls{bnb} process, ensuring efficient relaxations and tractable evaluations of subdifferentials and proximal operators. These results not only unify existing approaches but also significantly extend the range of problems that can be efficiently solved using \gls{bnb} methods.

To put our analytical findings into practice, we introduced \elops, an open-source \textsc{Python} toolbox that implements state-of-the-art \gls{bnb} strategies. Our numerical experiments demonstrated that \elops{} substantially outperforms both commercial and specialized solvers, achieving speedups of up to several orders of magnitude. Additionally, it enables the resolution of problem instances that were previously computationally infeasible, expanding the applicability of $\ell_0$-penalized models in machine learning and signal processing.


\begin{acknowledgements}
    Part of this work has been carried out when Th\'{e}o Guyard was affiliated to INSA Rennes, IRMAR - CNRS UMR 6625, France and funded by the ANR-11-LABX-0020.
\end{acknowledgements}

\clearpage
\appendix


\newcommand{\domrightbound}{\pvi{+}}
\newcommand{\domleftbound}{\pvi{-}}

\section{Technical Material} \label{app:technical-mat}
This appendix presents technical material that will be used in our subsequent derivations.

\subsection{Subdifferential of one-dimensional proper convex functions}

In this paragraph, we discuss some properties of the subdifferential of one-dimensional convex proper functions \(\kfuncdef{\genericfunc}{\kR}{\rextendedRealLine}\). 
First, we remind the concepts of left and right derivatives of the function \(\genericfunc\).
\begin{definition}\label{definition:left-right-derivatives}
    Let \(\kfuncdef{\genericfunc}{\kR}{\rextendedRealLine}\) be a proper function and \(\pvi{}\in\dom{\genericfunc}\).
    The left and right derivatives of \(\genericfunc\) at some \(\pvi{} \in \dom\genericfunc\) are defined as
    \begin{align}
        \label{eq:def-left-deriv}
        \partial_-\genericfunc(\pvi{}) &\defeq \lim_{\varepsilon\uparrow 0} \frac{\genericfunc(\pvi{}+\varepsilon)-\genericfunc(\pvi{})}{\varepsilon} \\
        \label{eq:def-right-deriv}
        \partial_+\genericfunc(\pvi{}) &\defeq \lim_{\varepsilon\downarrow 0} \frac{\genericfunc(\pvi{}+\varepsilon)-\genericfunc(\pvi{})}{\varepsilon}
        ,
    \end{align}
    provided that the limit exists in \(\extendedRealLine\).
\end{definition}
With this definition, the following result holds: 
\begin{lemma} 
    \label{lemma:properties-left-right-derivatives V2}
    Let \(\kfuncdef{\genericfunc}{\kR}{\rextendedRealLine}\) be a one-dimensional proper and convex function.
    Then, the left and right derivatives \(\partial_-\genericfunc\), \(\partial_+\genericfunc\) obey the following properties: 
    \begin{enumerate}[label=\textit{\alph*.},topsep = 3pt]
        \item \label{item-lemma:properties-left-right-derivatives:limits exists on domain}
        Both \(\partial_-\genericfunc(\pvi{})\) and \(\partial_+\genericfunc(\pvi{})\) exists in \(\extendedRealLine\) \textnormal{\cite[\propositionInRef{17.2.(i)}]{Bauschke2017}}.
        \item \label{item-lemma:properties-left-right-derivatives:left derivative leq right derivative}
        For all \(\pvi{}\in\dom{\genericfunc}\),  we have \(\partial_-\genericfunc(\pvi{}) \leq \partial_+\genericfunc(\pvi{})\)
        \textnormal{\cite[\propositionInRef{17.16.(i)}]{Bauschke2017}}.
        \item \label{item-lemma:properties-left-right-derivatives:finite on interior of the domain}
        For all \(\pvi{}\in\interior{\dom{\genericfunc}}\), we have \(-\infty < \partial_-\genericfunc(\pvi{}) \leq \partial_+\genericfunc(\pvi{}) < +\infty\) 
        \textnormal{\cite[\propositionInRef{17.2.(vi)}]{Bauschke2017}}.\footnote{
            The result derives from the fact that the interior of \(\dom{\genericfunc}\) is a subset of its core when $\genericfunc$ is convex, according to \cite[\propositionInRef{8.2} and \equationInRef{6.11}]{Bauschke2017}.
        }
        \item \label{item-lemma:properties-left-right-derivatives:partial omega+ <= partial omega-}
        For all \(\pvi{},\pvi{}'\in\dom{\genericfunc}\), we have \(\pvi{} < \pvi{}' \implies \partial_+\genericfunc(\pvi{}) \leq \partial_-\genericfunc(\pvi{}')\)
        \textnormal{\cite[\propositionInRef{17.16.(iii)}]{Bauschke2017}}.
    \end{enumerate}
\end{lemma}

For the sake of readability, we also reproduce the characterization of the subdifferential of a proper convex function defined on the real line provided in \cite[\propositionInRef{17.16.(ii)}]{Bauschke2017}: 
\begin{lemma} \label{lemma:subdiff-convex-fun-1D}
    Let \(\kfuncdef{\genericfunc}{\kR}{\rextendedRealLine}\) be a proper 
    convex function.
    For all \(\pvi{}\in\dom{\genericfunc}\), we have
    \begin{equation}
        \subdiff \genericfunc(\pvi{}) = [\partial_-\genericfunc(\pvi{}),\partial_{+}\genericfunc(\pvi{})]\cap \kR.
    \end{equation} 
\end{lemma}

We finally highlight the following consequence of~\Cref{lemma:properties-left-right-derivatives V2}.\ref{item-lemma:properties-left-right-derivatives:partial omega+ <= partial omega-} and \Cref{lemma:subdiff-convex-fun-1D}: 
\begin{lemma} 
    \label{lemma:subdiff-convex-fun-1D-is-nondecreasing}
    Let \(\kfuncdef{\genericfunc}{\kR}{\rextendedRealLine}\) be a proper convex function.
    For all \(\pvi{}\in\dom{\genericfunc}\) and \(\pvi{}'\in\dom{\genericfunc}\), we have\footnote{
        We remind the reader that \(\sup \emptyset = -\infty\) and \(\inf \emptyset = +\infty\) by convention.
    }
    \begin{equation}
        \pvi{} < \pvi{}' 
        \implies
        \sup\, \subdiff \genericfunc(\pvi{})
        \leq
        \inf\, \subdiff \genericfunc(\pvi{}')
        .
    \end{equation}
\end{lemma}

\subsection{Topological result}

The next result provides a variant of~\cite[\propositionInRef{11.1.(iv)}]{Bauschke2017}.
\begin{lemma} 
    \label{lemma:infima open set = infima closed set}
    If \(\kfuncdef{\aGenericOneDfunction}{\kR}{\rextendedRealLine}\) be a proper, closed and convex function. If \(\dom{\aGenericOneDfunction}\) is not reduced to a singleton, we have 
    \begin{equation}
        \label{eq:lemma:infima open set = infima closed set}
        \forall \pvi{0}\in\kR:\ 
        \inf_{
            \pvi{}\in\kR\setminus\{\pvi{0}\}
        }
        \aGenericOneDfunction(\pvi{})
        =
        \inf_{
            \pvi{}\in\kR
        }
        \aGenericOneDfunction(\pvi{})
        .
        \end{equation}
\end{lemma}
\begin{proof}
    If \(\pvi{0}\notin \dom{\aGenericOneDfunction}\), the result immediately follows from the fact that \(\aGenericOneDfunction\) is proper.         
    Now, assume that \(\pvi{0}\in \dom{\aGenericOneDfunction}\). First, we have: 
    \begin{equation}
        \label{eq:proof lemma:infima open set = infima closed set:splitting subpb in two convex sets - B}
        \inf_{
            \pvi{}\in\kR\setminus\{\pvi{0}\}
        } \aGenericOneDfunction(\pvi{})
        =
        \min\kbrace{
            \inf_{
                \pvi{}\in]-\infty, \pvi{0}[
            } \aGenericOneDfunction(\pvi{})
            ,
            \inf_{
                \pvi{}\in]\pvi{0}, +\infty[
            } \aGenericOneDfunction(\pvi{})
        }
        .
    \end{equation}
    Since \(\dom{\aGenericOneDfunction}\) is not reduced to a singleton, then 
    \(\dom{\aGenericOneDfunction}\cap]-\infty, \pvi{0}[\) and 
    \(\dom{\aGenericOneDfunction}\cap]\pvi{0}, +\infty[\) cannot be empty simultaneously. 
    Without loss of generality, assume that \(\dom{\aGenericOneDfunction}\cap]\pvi{0}, +\infty[\neq\emptyset\). 
    Then, since \(\aGenericOneDfunction\) is proper, closed and convex, invoking \cite[\propositionInRef{11.1.(iv)}]{Bauschke2017} yields 
    \begin{equation}
        \label{eq:proof lemma:infima open set = infima closed set:subpb equals whole pb - B}
        \inf_{
            \pvi{}\in]\pvi{0}, +\infty[
        } \aGenericOneDfunction(\pvi{})
        =
        \inf_{
            \pvi{}\in[\pvi{0}, +\infty[
        } \aGenericOneDfunction(\pvi{})
        .
    \end{equation}
    One finally obtains~\eqref{eq:lemma:infima open set = infima closed set} by plugging \eqref{eq:proof lemma:infima open set = infima closed set:subpb equals whole pb - B} into~\eqref{eq:proof lemma:infima open set = infima closed set:splitting subpb in two convex sets - B} and noting that \(]-\infty,\pvi{0}[\cup[\pvi{0},+\infty[=\kR\).
\end{proof}


\section{Some Properties of 
    \texorpdfstring{\(\regfunc\)}{g}, 
    \texorpdfstring{\(\conj\regfunc\)}{g*}, 
    \texorpdfstring{\(\biconj\regfunc\)}{g**},
    \texorpdfstring{\(\conj\pertfunc\)}{h*} and 
    \texorpdfstring{\(\biconj\pertfunc\)}{h**}%
} \label{app:property-fun}
This appendix gathers a series of elementary results related to the functions \(\pertfunc\), \(\conj\pertfunc\), \(\biconj\pertfunc\), \(\regfunc\), \(\conj\regfunc\) and \(\biconj\regfunc\) that will be extensively used in the proofs of the main results of the paper.
~\\

Our first lemma outlines some properties of the function \(\pertfunc\) and \(\regfunc\), from which we derive several corollaries regarding the functions \(\conj\pertfunc\), \(\biconj\pertfunc\), \(\regfunc\), \(\conj\regfunc\) and \(\biconj\regfunc\).
\begin{lemma} \label{lemma:g-even-nonneg}
    The function \(\regfunc\) defined in \eqref{eq:regfunc}:
    \begin{enumerate}[label=\textit{\alph*.},topsep = 3pt]
        \item verifies \(\regfunc(\pvi{}) \geq \regfunc(0) = 0\) under \ref{assumption:zero-minimized},
        \label{item-lemma:zero-minimized}
        \item is closed under \ref{assumption:closed},
        \label{item-lemma:g-closed}
        \item is even under \ref{assumption:even}.
        \label{item-lemma:g-even}
    \end{enumerate} 
\end{lemma}
\begin{proof}
    First, we have \(\pertfunc(\pvi{})\geq 0\)  and \(\|\pvi{}\|_0\geq 0\) with equality if \(\pvi{}=0\) from~\ref{assumption:zero-minimized} and~\eqref{eq:l0-norm}, respectively. Consequently, we have \(\regfunc(\pvi{}) \geq \regfunc(0) = 0\) since \(\reg > 0\) which yields item~\ref{item-lemma:zero-minimized}
    Second, the function \(\pertfunc\) is closed from \ref{assumption:closed} and \(\|\cdot\|_0\) is closed from \cite[\exampleInRef{2.11}]{beck2017first}. Since \(\reg > 0\), we deduce from \cite[\theoremInRef{2.7.(b)}]{beck2017first} that \(\regfunc\) is closed, which yields item~\ref{item-lemma:g-closed}
    Finally, the function \(\|\cdot\|_0\) is even by definition and \(\pertfunc\) is even from \ref{assumption:even}. Therefore, the function \(\regfunc\) corresponding to their linear combination is even, which yields item~\ref{item-lemma:g-even}
\end{proof}
\begin{corollary} \label{cor:nonneg}
    Under~\ref{assumption:zero-minimized}, any function \(\genericfunc\in\{\pertfunc,\conj\pertfunc, \biconj\pertfunc,\conj\regfunc, \biconj\regfunc\}\) verifies \(\conj\genericfunc(\pvi{})\geq \conj\genericfunc(0)=0\). 
\end{corollary}
\begin{proof}
    The result immediately follows from \ref{assumption:zero-minimized}, \Cref{lemma:g-even-nonneg}.\ref{item-lemma:zero-minimized}~and~\cite[\propositionInRef{13.22}]{Bauschke2017}.
\end{proof}
\begin{corollary} \label{remark:properness}
    Under~\ref{assumption:zero-minimized}, any function \(\genericfunc\in\{\pertfunc,\conj\pertfunc, \biconj\pertfunc,\regfunc,\conj\regfunc, \biconj\regfunc\}\) is proper. 
\end{corollary}
\begin{proof}
    Any function \(\kfuncdef{\genericfunc}{\kR}{\extendedRealLine}\) verifying \(\genericfunc(\pvi{})\geq \genericfunc(0)=0\) is proper since it is not equal to \(+\infty\) everywhere and does not take on the value \(-\infty\). 
    Hence, the function \(\pertfunc\) is proper from \ref{assumption:zero-minimized}, and \(\conj\pertfunc, \biconj\pertfunc,\regfunc, \conj\regfunc, \biconj\regfunc\) are proper as well in view of \Cref{lemma:g-even-nonneg}.\ref{item-lemma:zero-minimized} and \Cref{cor:nonneg}.
\end{proof}
\begin{corollary}
    \label{cor:subdiff properties}
    Under~\ref{assumption:zero-minimized}, any function \(\genericfunc\in\{\pertfunc, \conj\pertfunc, \biconj\pertfunc,\regfunc, \conj{\regfunc}, \biconj{\regfunc}\}\) verifies 
    \begin{align}
        &0\in\partial\genericfunc(0) \label{eq:cor:0 in subdiff} \\
        \forall \pvi{}\in\dom{\genericfunc}\cap
        &\kR+\setminus\{0\} \, :\ \inf\, \partial\genericfunc(\pvi{}) \geq0
        \label{eq:cor:nonnegative subdiff}
        .
    \end{align}
\end{corollary}
\begin{proof}
    Applying \Cref{lemma:g-even-nonneg}.\ref{item-lemma:zero-minimized} and \Cref{cor:nonneg}, we have that \(0\) is a minimizer of \(\genericfunc\) under~\ref{assumption:zero-minimized}.
    Since \(\genericfunc\) is also proper from \Cref{remark:properness}, the property~\eqref{eq:cor:0 in subdiff} then follows from Fermat's optimality condition~\cite[\theoremInRef{16.3}]{Bauschke2017}. 
    Now, let \(\pvi{}\in\dom{\genericfunc}\) be such that \(\pvi{}>0\).
    If \(\partial\genericfunc(\pvi{})=\emptyset\) then the property \eqref{eq:cor:nonnegative subdiff} trivially holds since \(\inf \emptyset = +\infty\) by convention. 
    Otherwise, if there exists some \(\bvi{}\in\partial\genericfunc(\pvi{})\), then
    \begin{equation}
        0 \leq \partial_+\genericfunc(0) \leq \partial_-\genericfunc(\pvi{}) \leq \bvi{}
    \end{equation}
    where we have used \Cref{lemma:subdiff-convex-fun-1D} and the facts that \(0\in\partial\genericfunc(0)\) and \(\bvi{}\in\partial\genericfunc(\pvi{})\) to obtain the first and third inequalities.
    The second inequality follows from \Cref{lemma:properties-left-right-derivatives V2}.\ref{item-lemma:properties-left-right-derivatives:partial omega+ <= partial omega-} since \(0<\pvi{}\) by assumption.
    Property \eqref{eq:cor:nonnegative subdiff} follows by noting that \(\inf\partial\genericfunc(\pvi{}) = \partial_-\genericfunc(\pvi{})\) from \Cref{lemma:subdiff-convex-fun-1D}.
\end{proof}
\begin{corollary} \label{cor:proper convex closed}
    Under \ref{assumption:zero-minimized}, any function \(\genericfunc\in\{\conj\pertfunc, \biconj\pertfunc,\conj\regfunc, \biconj\regfunc\}\) is proper, closed and convex.
\end{corollary}
\begin{proof}
    We have already proved in \Cref{remark:properness} that any \(\genericfunc\in\{\conj{\pertfunc}, \biconj{\pertfunc}, \conj{\regfunc}, \biconj{\regfunc}\}\) is proper under \ref{assumption:zero-minimized}.
    Closedness and convexity properties then follow from \cite[\theoremInRef{4.3}]{beck2017first}.
\end{proof}

We next provide a lemma characterizing regarding the domain of the conjugate associated one-dimensional proper, closed and convex functions minimized at the origin, as well as several byproducts regarding properties of the functions \(\pertfunc\), \(\conj\pertfunc\), \(\biconj\pertfunc\), \(\regfunc\), \(\conj\regfunc\) and \(\biconj\regfunc\).
\begin{lemma}
    \label{lemma:V2:nonzero implies support strictly includes 0}
    Let \(\kfuncdef{\aGenericOneDfunction}{\kR}{\rextendedRealLine}\) be a proper, closed, and convex function such that \(\aGenericOneDfunction\geq\aGenericOneDfunction(0)=0\). 
    Then,
    \begin{equation} \label{eq:lemma:V2:nonzero implies support strictly includes 0}
        \exists\pvi{}\in\kR+: \ \aGenericOneDfunction(\pvi{})\neq0
        \implies 
        \dom{\conj\aGenericOneDfunction}\cap\kR+\setminus\{0\}\neq\emptyset
        .
    \end{equation}
 \end{lemma}
\begin{proof}
    We prove the statement by showing the contrapositive:
    \begin{equation} \label{eq:lemma:V2:contrapositive nonzero implies support strictly includes 0}
        \dom{\conj\aGenericOneDfunction}\cap\kR+\setminus\{0\}=\emptyset
        \implies 
        \forall\pvi{}\in\kR+ :\;\aGenericOneDfunction(\pvi{}) = 0.
    \end{equation}
    Assume that \(\dom{\conj\aGenericOneDfunction}\cap\kR+\setminus\{0\}=\emptyset\).
    Since \(\aGenericOneDfunction\) satisfies \(\aGenericOneDfunction\geq \aGenericOneDfunction(0)\), then \(\conj{\aGenericOneDfunction}\geq\conj{\aGenericOneDfunction}(0)=0\) by~\cite[\propositionInRef{13.22}]{Bauschke2017}, which implies that \(\dom{\conj{\aGenericOneDfunction}}\cap\kR+ = \{0\}\).
    The implication in~\eqref{eq:lemma:V2:contrapositive nonzero implies support strictly includes 0} then follows from the combination of the following two arguments. 
    First, we have
    \begin{equation*}
        \forall \pvi{}\in\kR+:\
        \biconj\aGenericOneDfunction (\pvi{})
        = \sup_{\bvi{}\in\dom{\conj{\aGenericOneDfunction}}}\bvi{}\pvi{} - \conj{\aGenericOneDfunction}(\bvi{})
        = \sup_{\bvi{}\in\dom{\conj{\aGenericOneDfunction}}\cap\kR+}\bvi{}\pvi{} - \conj{\aGenericOneDfunction}(\bvi{})
        = 0
    \end{equation*}
    where the first equality follows from~\cite[\theoremInRef{13.10.(iv)}]{Bauschke2017},
    the second from the fact that \(\bvi{}\pvi{} - \conj{\aGenericOneDfunction}(\bvi{})\leq -\conj{\aGenericOneDfunction}(0)\) for all \(\pvi{}\in\kR+\) and \(\bvi{}\in\kR-\) since \(\conj{\aGenericOneDfunction}\geq\conj{\aGenericOneDfunction}(0)\), and the last one from \(\conj{\aGenericOneDfunction}(0)=0\) and \(\dom{\conj{\aGenericOneDfunction}}\cap\kR+ = \{0\}\).
    Second, we have from \cite[\theoremInRef{4.8}]{beck2017first} that \(\aGenericOneDfunction=\biconj\aGenericOneDfunction\) since \(\aGenericOneDfunction\) is assumed to be proper, closed and convex.
\end{proof}
\begin{corollary} \label{cor:V2:0 in interior of domh*}
    If~\ref{assumption:zero-minimized}-\ref{assumption:closed}-\ref{assumption:convex} hold and \(\pertfunc\) is not identically zero on \(\kR+\), then there exists \(\bvi{+}>0\) such that \([0,\bvi{+}]\subseteq \dom{\conj{\pertfunc}}\).
\end{corollary}
\begin{proof}
    Hypotheses~\ref{assumption:zero-minimized}-\ref{assumption:closed}-\ref{assumption:convex} and \Cref{remark:properness} imply that \(\pertfunc\) is proper, closed, convex and such that \(\pertfunc\geq\pertfunc(0)=0\). 
    Moreover, the hypothesis ``\(\pertfunc\) is not identically zero on \(\kR+\)'' exactly corresponds to the left-hand side of~\eqref{eq:lemma:V2:nonzero implies support strictly includes 0}.
    Applying \Cref{lemma:V2:nonzero implies support strictly includes 0} with \(\aGenericOneDfunction=\pertfunc\), we then obtain $\dom{\conj{\pertfunc}}\cap\kR+\setminus\{0\}\neq\emptyset$. 
    Hence, there exists \(\bvi{+}\in\dom{\conj{\pertfunc}}\) such that \(\bvi{+} >0\).
    Since \(\conj{\pertfunc}\) is convex under~\ref{assumption:zero-minimized} from \Cref{cor:proper convex closed}, \(\dom{\conj{\pertfunc}}\) is a convex set~\cite[\propositionInRef{8.2}]{Bauschke2017}. 
    We thus obtain the result by noting that \(0\in\dom{\conj{\pertfunc}}\) under \ref{assumption:zero-minimized} from \Cref{cor:nonneg}, which implies that \([0,\bvi{+}] \subseteq \dom{\conj{\pertfunc}}\) due to the convexity of \(\dom{\conj{\pertfunc}}\). 
\end{proof}
\begin{corollary} \label{cor:continuity}
    Under \ref{assumption:zero-minimized}-\ref{assumption:closed}-\ref{assumption:convex}, \(\pertfunc\) is continuous on its domain. 
    Under \ref{assumption:zero-minimized},  any function \(\genericfunc\in\{\conj\pertfunc, \biconj\pertfunc,\conj\regfunc, \biconj\regfunc\}\) is continuous on its domain. 
\end{corollary}
\begin{proof}
    We have from~\cite[\theoremInRef{2.22}]{beck2017first} that any proper, closed, convex function defined on \(\kR\) is continuous on its domain. 
    Now, the function \(\pertfunc\) verifies these hypotheses from~\ref{assumption:zero-minimized}-\ref{assumption:closed}-\ref{assumption:convex} and \Cref{remark:properness}. Similarly, any \(\genericfunc\in\{\conj\pertfunc, \biconj\pertfunc,\conj\regfunc, \biconj\regfunc\}\) satisfies these conditions under~\ref{assumption:zero-minimized} according to \Cref{cor:proper convex closed}.
\end{proof}

Our last technical lemma provides a result on the monotonicity of one-dimensional convex functions minimized at the origin, from which properties of the functions \(\conj\pertfunc\), \(\biconj\pertfunc\), \(\conj\regfunc\) and \(\biconj\regfunc\) are derived.
\begin{lemma} \label{lemma:monoticity}
	Let \(\kfuncdef{\genericfunc}{\kR}{\extendedRealLine}\) be a convex function verifying \(\genericfunc(\bvi{})\geq \genericfunc(0)=0\).
    Then \(\genericfunc\) is non-decreasing on \(\dom{\genericfunc}\cap\kR+\).~Moreover, if \(\bvi{}\in\dom{\genericfunc}\cap\kR+\) is such that \(\genericfunc{(\bvi{})}>0\), then \(\genericfunc\) is strictly increasing on \(\dom{\genericfunc} \cap [\bvi{},+\infty[\). 
\end{lemma}
\begin{proof}
	If \(\dom{\genericfunc}\cap\kR+=\{0\}\), the result is trivial.
    We thus assume that \(\dom{\genericfunc}\cap\kR+\neq \{0\}\) in the rest of the proof. 
	Let \(\bvi{}\in\dom{\genericfunc}\) and \(\bvi{}'\in\dom{\genericfunc}\) such that \(0\leq \bvi{}<\bvi{}'\). 
	Since \(\genericfunc\) is convex, we have from \cite[\propositionInRef{8.4}]{Bauschke2017} that 
	\begin{align}
		\genericfunc(\bvi{})
		= \genericfunc(\tfrac{\bvi{}}{\bvi{}'} \bvi{}'+ (1-\tfrac{\bvi{}}{\bvi{}'})\, 0 )
		\leq \tfrac{\bvi{}}{\bvi{}'}\,\genericfunc(\bvi{}') + (1-\tfrac{\bvi{}}{\bvi{}'})\,\genericfunc(0)
		= \tfrac{\bvi{}}{\bvi{}'}\,\genericfunc(\bvi{}')
		\leq \genericfunc(\bvi{}')
	\end{align}
	where the second equality follows from \(\genericfunc(0)=0\) and the second inequality from the fact that \(\genericfunc(\bvi{}')\geq 0\) and \(\bvi{}<\bvi{}'\). 
	This proves the first part of the statement. Now, the last inequality is strict as soon as \(\genericfunc(\bvi{})>0\), which gives the second part of the statement. 
\end{proof}
\begin{corollary}\label{cor:monoticity}
    Let \(\genericfunc\in\{\conj\pertfunc, \biconj\pertfunc,\conj\regfunc, \biconj\regfunc\}\). 
    If \ref{assumption:zero-minimized} holds, 
    then \(\genericfunc\) is non-decreasing on \(\dom{\genericfunc} \cap \kR+\) and \(\genericfunc\) is strictly increasing on \(\dom{\genericfunc} \cap \kintervco{\bvi{}}{+\infty}\) for any \(\bvi{}\in\dom{\genericfunc}\cap\kR+\) such that \(\genericfunc{(\bvi{})}>0\). 
	If~\ref{assumption:convex} moreover holds, these statements also hold for \(\genericfunc=\pertfunc\).
\end{corollary}
\begin{proof}
    \Cref{lemma:monoticity} can be applied with \(\genericfunc\in\{\conj\pertfunc, \biconj\pertfunc,\conj\regfunc, \biconj\regfunc\}\) under \ref{assumption:zero-minimized} from \Cref{cor:nonneg,cor:proper convex closed}.
    It can also be applied with \(\genericfunc=\pertfunc\) under \ref{assumption:zero-minimized} and \ref{assumption:convex}.
    These two observations directly give the desired result.
\end{proof}

We conclude by observing the following property.
\begin{corollary} \label{cor:even}
    Under~\ref{assumption:even}, any function \(\genericfunc\in\{\conj\pertfunc, \biconj\pertfunc,\conj\regfunc, \biconj\regfunc\}\) is even.
\end{corollary}
\begin{proof}
    This property is a direct consequence of \cite[\propositionInRef{13.21}]{Bauschke2017} since \(\pertfunc\) is even under \ref{assumption:even}, and so is \(\regfunc\) according to \Cref{lemma:g-even-nonneg}.\ref{item-lemma:g-even}
\end{proof}


\section{Some Properties of \texorpdfstring{\(\pertslope\)}{tau} and \texorpdfstring{\(\pertlimit\)}{mu}}

In this appendix, we exhibit some properties of the parameters \(\pertslope\), \(\pertlimit\) and \(\pertrightlope\) defined in \eqref{eq:pertslope}, \eqref{eq:pertlimit} and \eqref{eq:pertrightlope}. 
In our statements, we use the following short-hand notation 
\begin{equation} 
    \label{eq:domb}
    \domb \ \defeq\ \sup\ \kset{\bvi{}\in\kR+}{\bvi{} \in \dom{\conj{\pertfunc}}}
\end{equation}
to simplify our exposition.
\Cref{remark:special case pertfunc is identically zero} considers the particular case where \(\pertfunc\) is identically zero. 
\Cref{sec:beta} is dedicated to the study of some properties of \(\domb\). 
Finally, \Cref{sec:tau,sec:mu} provide some useful results about parameters \(\pertslope\) and \(\pertlimit\), respectively.

\subsection{Particular case of \texorpdfstring{\(\pertfunc\)}{\pertfunc}}
\label{remark:special case pertfunc is identically zero}

In the case where \(\pertfunc\) is identically zero on \(\kR+\) and satisfies~\ref{assumption:zero-minimized}, we have the following result: 
\begin{lemma}\label{lemma:case-id-zero}
    Assume \ref{assumption:zero-minimized} holds. If \(\pertfunc(\pvi{})=0\) for all \(\pvi{}\in\kR+\), then 
    \begin{align}\label{eq:h*-case-id-zero}
        \forall \bvi{}\in\kR+: \conj{\pertfunc}(\bvi{}) = \icvx(\bvi{}=0)
    \end{align}
    and 
    \begin{subequations}
        \begin{align}
        	\domb &\,=\, 0\\
            \pertslope &\,=\, 0 \\
            \pertlimit &\,=\, +\infty \\
            \pertrightlope & \,=\, +\infty
            .
        \end{align}
    \end{subequations}    
\end{lemma}
\begin{proof}
    If \ref{assumption:zero-minimized} holds, we have 
    \begin{align}
        \forall \bvi{}\in\kR+: 
        \conj{\pertfunc}(\bvi{}) 
        = \sup_{\pvi{}\in\kR} \bvi{}\pvi{}-\pertfunc(\pvi{})
        =\sup_{\pvi{}\in\kR+} \bvi{}\pvi{}-\pertfunc(\pvi{})
        =\sup_{\pvi{}\in\kR+} \bvi{}\pvi{} 
        = \eta(\bvi{}=0),
    \end{align}
    where the second equality follows from \ref{assumption:zero-minimized} and the third from the fact that \(\pertfunc(\pvi{})=0\) for all \(\pvi{}\in\kR+\) by assumption. 
	Now, we observe the following properties.
	\begin{itemize}
		\item First, we have \(\kset{\bvi{}\in\kR+}{\conj{\pertfunc}(\bvi{})\leq\reg}=\dom{\conj{\pertfunc}}\cap\kR+=\{0\}\) according to \eqref{eq:h*-case-id-zero}. If follows that \(\domb=0\) and \(\pertslope=0\) from their definitions given in \eqref{eq:domb} and \eqref{eq:pertslope}, respectively.
		\item Second, let \(\pvi{}\geq0\). We notice that
		\begin{equation} 
			\label{eq:proof lemma:case-id-zero:subdiff at 0 part 1}
				\forall \bvi{}\in\kR-:\quad
				\conj{\pertfunc}(0) + \pvi{}(\bvi{} - 0) \leq \conj{\pertfunc}(0) \leq \conj{\pertfunc}(\bvi{})
			\end{equation}
		where the second equality holds under~\ref{assumption:zero-minimized} in view of \Cref{cor:nonneg}. 
		Since \(\dom{\conj{\pertfunc}}\cap\kR+=\{0\}\),  the inequality \(\conj{\pertfunc}(0) + \pvi{}(\bvi{} - 0) \leq \conj{\pertfunc}(\bvi{})\) is trivially satisfied for all \(\bvi{}>0\) 
		and we therefore obtain that \(\pvi{}\in\partial\conj{\pertfunc}(0)\).  
		Consequently, we deduce that \(\kR+\subseteq\partial\conj{\pertfunc}(0)\), and it follows from definition~\eqref{eq:pertlimit} that
		\(\pertlimit=+\infty\).
		\item Finally, we obtain \(\pertrightlope=+\infty\) by definition \eqref{eq:pertrightlope} since  \(\pertlimit=+\infty\).
	\end{itemize}
\end{proof}

\subsection{Some properties of \texorpdfstring{\(\domb\)}{beta}}
\label{sec:beta}

This section presents several properties of the parameter \(\domb\) defined in \eqref{eq:domb}.
We first establish its non-negativeness.

\begin{lemma} \label{lemma:domb result}
	If~\ref{assumption:zero-minimized}-\ref{assumption:closed}-\ref{assumption:convex} hold, then \(\domb\geq 0\) with equality if and only if  \(\pertfunc\) is identically zero on \(\kR+\). 
\end{lemma}
\begin{proof}
	Since \(0\in\dom{\conj{\pertfunc}}\) under~\ref{assumption:zero-minimized} from \Cref{cor:nonneg}, we have \(\domb\geq 0\). 
	Moreover, if \(\pertfunc\) is identically zero on \(\kR+\) and that \ref{assumption:zero-minimized} holds, \Cref{lemma:case-id-zero} leads to \(\domb=0\). Conversely, if \(\pertfunc\) is not identically zero on \(\kR+\) and that \ref{assumption:zero-minimized}-\ref{assumption:closed}-\ref{assumption:convex} hold, we have \(\domb>0\) as a direct consequence of \Cref{cor:V2:0 in interior of domh*}.
\end{proof}

The next lemma provides some refined results connecting \(\domb\) to \(\dom{\conj\pertfunc}\): 
\begin{lemma} \label{lemma:dom-domb}
	If~\ref{assumption:zero-minimized}-\ref{assumption:closed}-\ref{assumption:convex} hold, then
	\begin{align}
		\label{eq:proof lemma:dom-domb:[0,domb[ subset domh*}
		\dom{\conj\pertfunc} &\supseteq \kintervco{0}{\domb}
		\\
		\closure(\dom{\conj{\pertfunc}})\cap \kR+  &= [0,\domb]\cap \kR+
		\label{eq:cl(dom-h*)}
		\\
	 	\kparen{\interior{\dom{\conj{\pertfunc}}}\setminus\{0\}}\cap \kR+ &=\,  ]0,\domb[ 
	 	.
	 	\label{eq:int(dom-h*)}
	 \end{align}
\end{lemma}
\begin{proof}
	We first note from \Cref{lemma:case-id-zero} that \eqref{eq:proof lemma:dom-domb:[0,domb[ subset domh*}-\eqref{eq:int(dom-h*)} hold when \(\pertfunc\) is identically zero on \(\kR+\).
	We thus consider the case where \(\pertfunc\) is not identically zero on \(\kR+\) in the rest of the proof.
	In this case, we have \(\domb>0\) from \Cref{lemma:domb result}. Now, the properties claimed can be proved as followed.
	\begin{itemize}
		\item We note that \(\conj\pertfunc\) is convex from \Cref{cor:proper convex closed} under \ref{assumption:zero-minimized}. Hence, \(\dom{\conj\pertfunc}\) is a convex subset of \(\kR\) according to \cite[\propositionInRef{8.2}]{Bauschke2017}. 
		Since \(0\in\dom{\conj\pertfunc}\) under~\ref{assumption:zero-minimized} by \Cref{cor:nonneg}, we thus have that \(\kintervcc{0}{\bvi{}}\subseteq \dom{\conj{\pertfunc}}\) for any \(\bvi{}\in \dom{\conj{\pertfunc}}\).
		Property \eqref{eq:proof lemma:dom-domb:[0,domb[ subset domh*} then follows from the definition of \(\domb\) as a supremum of \(\dom{\conj{\pertfunc}}\).
		\item Since \ref{assumption:zero-minimized}-\ref{assumption:closed}-\ref{assumption:convex} hold, property \eqref{eq:proof lemma:dom-domb:[0,domb[ subset domh*} implies that \(\closure(\kintervco{0}{\domb})\subseteq\closure(\dom{\conj\pertfunc})\).
		Moreover, since \(\kR+\) is a closed set and contains \(\kintervco{0}{\domb}\), it also contains its closure, which gives \(\closure(\kintervco{0}{\domb})\subseteq\closure(\dom{\conj{\pertfunc}})\cap \kR+\).
		Since \(\kintervco{0}{\domb}\neq \emptyset\) as \(\domb>0\), one observes that \(\closure(\kintervco{0}{\domb})=\kintervcc{0}{\domb}\cap \kR+\) by distinguishing between the cases ``\(\domb\in \kR+\)'' and ``\(\domb=+\infty\)''. 
		In view of our previous inclusion, this leads to \(\kintervcc{0}{\domb}\cap\kR+\subseteq\closure(\dom{\conj{\pertfunc}})\cap \kR+\).
		Conversely, we note that
		\begin{equation} \label{eq:proof lemma:dom-domb:[0,domb[ superset domh*}
			\dom{\conj{\pertfunc}} \subseteq \kintervoc{-\infty}{\domb}\cap\kR
		\end{equation}
		by definition of \(\domb\).
		This leads to \(\closure(\dom{\conj{\pertfunc}}) \subseteq \closure(\kintervoc{-\infty}{\domb}\cap \kR)=\kintervoc{-\infty}{\domb}\cap \kR\).
		Hence, \(\closure(\dom{\conj\pertfunc}) \cap \kR+ \subseteq \kintervoc{-\infty}{\domb}\cap\kR\cap\kR+=\kintervcc{0}{\domb}\cap\kR+\) which leads to the property \eqref{eq:cl(dom-h*)}.
		\item By considering the interior of the two sides of the inclusion in~\eqref{eq:proof lemma:dom-domb:[0,domb[ subset domh*}, we obtain \(\kintervoo{0}{\domb}\subseteq\interior{\dom{\conj\pertfunc}}\). 
		Since \(0\notin \kintervoo{0}{\domb}\), this leads to \(\kintervoo{0}{\domb}\subseteq \interior{\dom{\conj\pertfunc}}\setminus\{0\} \).
		Finally, as \(\kintervoo{0}{\domb}\subseteq\kR+\), we end up with \(\kintervoo{0}{\domb}\subseteq\interior{\dom{\conj\pertfunc}}\cap\kR+\). 
		Conversely, if $\domb=+\infty$, inclusion \eqref{eq:proof lemma:dom-domb:[0,domb[ superset domh*} trivially reduces to \(\dom{\conj\pertfunc}\subseteq\kR\), so that \(\interior{\dom{\conj\pertfunc}}\setminus\{0\}\subseteq\kR+\setminus\{0\}=\kintervoo{0}{+\infty}\). 
		Moreover, if $\domb<+\infty$, inclusion \eqref{eq:proof lemma:dom-domb:[0,domb[ superset domh*} leads to  \(\dom{\conj\pertfunc}\subseteq\kintervoc{-\infty}{\domb}\). 
		Hence \(\interior{\dom{\conj\pertfunc}}\subseteq\interior{\kintervoc{-\infty}{\domb}}=\kintervoo{-\infty}{\domb}\). 
		This leads to the property \eqref{eq:int(dom-h*)} by excluding \(0\) from the sets in both sides of the inclusion and taking the intersection with \(\kR+\).
	\end{itemize}
\end{proof}

\subsection{Some properties of \texorpdfstring{\(\pertslope\)}{tau}}
\label{sec:tau}

This section presents several properties of the parameter \(\pertslope\) defined in \eqref{eq:pertslope}. 
We first establish its finiteness:
\begin{proposition} \label{lemma:pertslope finite}
    If \ref{assumption:zero-minimized}-\ref{assumption:closed}-\ref{assumption:convex} hold, then \(\pertslope\in\kR+\).
\end{proposition}
\begin{proof}
   	Since \(\conj{\pertfunc}(0)=0\) under~\ref{assumption:zero-minimized} from \Cref{cor:nonneg}, we have \(\pertslope\geq 0\).     	
   	It thus remains to show that \(\pertslope<+\infty\).
   	If \(\pertslope=0\), the result is obviously true.
   	We thus assume hereafter that \(\pertslope>0\).

   	Under~\ref{assumption:zero-minimized}, \Cref{cor:nonneg,cor:monoticity} respectively imply that \(\conj{\pertfunc}(0)=0\) and \(\conj{\pertfunc}\) is non-decreasing on \(\dom{\conj{\pertfunc}}\cap\kR+\).
    Hence, by definition of \(\pertslope\) in \eqref{eq:pertslope}, we have \(\conj{\pertfunc}(\bvi{})\leq \lambda\) for any \(\bvi{}\in [0, \pertslope[\).  
    This leads to 
    \begin{equation} \label{eq:borne-sup-h*}
        \forall \bvi{}\in\kR+:\ \conj{\pertfunc}(\bvi{}) \leq \reg + \icvx(\bvi{}\in[0,\pertslope[)
        .
	\end{equation}
    Taking the conjugate of both sides, we then obtain:  
	\begin{subequations}
		\begin{align}
			\forall \pvi{}\in \kR:\ 
			\biconj{\pertfunc}(\pvi{})
			\,\geq\,&
			\conj{(\reg + \icvx(\bvi{}\in[0,\pertslope[))}(\pvi{})
			\\
			\,=\,&
			\conj{(\reg + \icvx(\bvi{}\in[0,\pertslope]))}(\pvi{})
			\label{eq:conj-ineq}
		\end{align}
	\end{subequations}     
	where the inequality follows from \cite[\propositionInRef{13.16.(ii)}]{Bauschke2017} and the equality follows from \Cref{lemma:infima open set = infima closed set}.
    On the one hand, since \(\pertfunc\) is proper, closed and convex under \ref{assumption:zero-minimized}-\ref{assumption:closed}-\ref{assumption:convex}, we have that the left-hand side of \eqref{eq:conj-ineq} is equal to \(\pertfunc(\pvi{})\) from \cite[\theoremInRef{4.8}]{beck2017first}.
    On the other hand, simple calculations \cite[\exampleInRef{7.9}]{Bauschke2017} show that the right-hand side of~\eqref{eq:conj-ineq} is equal to \([\pvi{}]_+ \pertslope - \reg\), reminding that we use the convention \(0\cdot \infty =0\).  
    These two observations lead to
    \begin{equation}
 		\forall \pvi{} \in \kR:\ 
        \pertfunc(\pvi{}) \geq [\pvi{}]_+ \pertslope - \reg
        . 
    \end{equation}        
    Finally, using the fact that \(\dom \pertfunc\cap\kR+\setminus \{0\}\neq \emptyset\) from \ref{assumption:zero-minimized}, we observe that \(\exists \pvi{+}>0\) such that \(\pertfunc(\pvi{+})\in\kR+\). Therefore, we conclude that \(\pertslope\leq (\reg + \pertfunc(\pvi{+}))/ [\pvi{+}]_+ < +\infty\).
\end{proof}
 
We now state some central results of this section, gathered in the next proposition:
\begin{proposition} \label{prop:pertslope in domain}
	If~\ref{assumption:zero-minimized}-\ref{assumption:closed}-\ref{assumption:convex} hold, the following assertions are verified:
	\begin{enumerate}[label=\textit{\alph*.},topsep = 3pt]
		\item \(0 \leq\, \pertslope \leq \domb\) with \(\pertslope=0 \iff \pertfunc\) is identically zero on \(\kR+\).
		\label{eq:prop:pertslope leq domb}
		\item \(\pertslope \in \dom{\conj{\pertfunc}}\).
		\label{eq:prop:pertslope in domain}
		\item \(\conj{\pertfunc}(\pertslope) \leq \reg\).
		\label{eq:prop:conj of h at pertslope leq reg}
		\item If \(\pertslope<\domb\) then  \(\bvi{}\in\kR+\) and \(\conj{\pertfunc}(\bvi{})=\reg\) \(\iff\) \(\bvi{}=\pertslope\). \label{eq:prop:conj of h at pertslope = domb}
	\end{enumerate}
\end{proposition}
\begin{proof}
	We first consider the case where \(\pertfunc\) is identically zero on \(\kR+\). 
	We have from \Cref{lemma:case-id-zero} that \(\pertslope=\domb=0\) so that the inequality and the converse part of item \ref{eq:prop:pertslope leq domb} are verified.
	Moreover, since \(\conj{\pertfunc}(0)=0\)
	from \Cref{cor:nonneg}, items \ref{eq:prop:pertslope in domain} and \ref{eq:prop:conj of h at pertslope leq reg} are also satisfied.
	Finally, item \ref{eq:prop:conj of h at pertslope = domb} does not apply since \(\pertslope=\domb\).
	The direct implication in the equivalence stated in item \ref{eq:prop:pertslope leq domb} is shown by contraposition in item \ref{eq:prop:pertslope leq domb} of the proof below.
	In the rest of the proof, we thus consider the case where \(\pertfunc\) is not identically zero on \(\kR+\) and prove the four items separately.
	\begin{itemize}
		\item Item \ref{eq:prop:pertslope leq domb} If \(\pertfunc\) is not identically zero on \(\kR+\) and \ref{assumption:zero-minimized}-\ref{assumption:closed}-\ref{assumption:convex} hold, \Cref{cor:V2:0 in interior of domh*} ensures the existence of \(\bvi{+}>0\) such that \([0, \bvi{+}]\subseteq\dom{\conj{\pertfunc}}\).
		Under~\ref{assumption:zero-minimized}, we also have \(\conj{\pertfunc}(0)=0\) from \Cref{cor:nonneg}  and \(\conj{\pertfunc}\) is continuous on its domain from \Cref{cor:continuity}. Hence, there exists \(\bvi{\reg}>0\) such that:
		\begin{equation}\label{eq:set-inclusion}
			\{0\} 
			\subsetneq \kintervcc{0}{\bvi{\reg}}
			\subseteq \kset{\bvi{}\in\kR+}{\conj{\pertfunc}(\bvi{}) \leq \reg} \subseteq \kset{\bvi{}\in\kR+}{\conj{\pertfunc}(\bvi{}) < +\infty}
			. 	
		\end{equation}
		Taking the supremum of each set in the above inclusion and using the definitions of \(\pertslope\) and \(\domb\) leads to \(0 < \pertslope \leq \domb\).
		We note that the strict inequality also shows (by contraposition) the direct implication in the equivalence stated in item \ref{eq:prop:pertslope leq domb}
		\item Item \ref{eq:prop:pertslope in domain} 
		Our proof leverages the following two ingredients:
		\begin{subequations}
			\begin{align}
				\pertslope &\in \ ]0,\domb] 
				\label{eq:proof prop pertslope domain:ingredient 1} \\
				\dom{\conj{\pertfunc}} &\supseteq [0,\domb[
				\label{eq:proof prop pertslope domain:ingredient 2}
			\end{align}
		\end{subequations}
		where~\eqref{eq:proof prop pertslope domain:ingredient 1} is a rewriting of item~\ref{eq:prop:pertslope leq domb} (when \(\pertfunc\) is not identically zero on \(\kR+\)) and~\eqref{eq:proof prop pertslope domain:ingredient 2} follows from \Cref{lemma:dom-domb}.
		We next distinguish different cases depending on the value of \(\domb\). 
		\begin{itemize}[label=$\bullet$]
			\item Case ``\(\domb\in\dom{\conj\pertfunc}\)''.
			Combining ``\(\domb\in\dom{\conj\pertfunc}\)'' with~\eqref{eq:proof prop pertslope domain:ingredient 2} leads to  \([0,\domb[\cup\{\domb\}=[0,\domb]\subseteq \dom{\conj\pertfunc}\). The result then follows from~\eqref{eq:proof prop pertslope domain:ingredient 1}.
			\item Case ``\(\domb\notin\dom{\conj\pertfunc}\) and \(\domb=+\infty\)''.
			Since \(\pertslope\in\kR+\) from \Cref{lemma:pertslope finite} (that is \(\pertslope\neq +\infty\)), \eqref{eq:proof prop pertslope domain:ingredient 1} reduces to \(\pertslope\in ]0,\domb[\). 
			This proves the result in view of~\eqref{eq:proof prop pertslope domain:ingredient 2}.
			\item Case ``\(\domb\notin\dom{\conj\pertfunc}\) and \(\domb<+\infty\)''.
			Let \(\{\bvi{}^{(k)}\}_{k\in\kN}\subset\dom{\conj\pertfunc}\) be some sequence such that \(\lim_{k\rightarrow+\infty}\bvi{}^{(k)}=\domb\). 
			Since \(\conj\pertfunc\) is closed under~\ref{assumption:zero-minimized} from \Cref{cor:proper convex closed} (and therefore lower semi-continuous, see \cite[\theoremInRef{2.6}]{beck2017first}), we have:
			\begin{equation}
				\liminf_{k\rightarrow+\infty} \ \conj\pertfunc(\bvi{}^{(k)})
				\geq \conj\pertfunc(\domb) = +\infty
			\end{equation}
			where the last equality follows from \(\domb\notin\dom{\conj\pertfunc}\).
			This implies that there exists \(k\in\kN\) such that \(\reg<\conj\pertfunc(\bvi{}^{(k)}) < +\infty\) where the second inequality holds because \(\bvi{}^{(k)}\in\dom{\conj\pertfunc}\). 
			Since \(\conj\pertfunc\) is non-decreasing on \(\dom{\conj{\pertfunc}}\cap\kR+\) from \Cref{cor:monoticity}, we deduce that
			\begin{equation*}
				\forall\bvi{}\in\kR+:\
				\bvi{}\geq \bvi{}^{(k)}
				\;\implies\;
				\conj\pertfunc(\bvi{}) > \reg
				.
			\end{equation*}
			Hence, by using the contrapositive of the latter result, we observe that
			\begin{equation}
				\kset{
					\bvi{}\in\kR+
				}{
					\conj\pertfunc(\bvi{}) \leq \reg
				}
				\subseteq
				[0,\bvi{}^{(k)}[
			\end{equation}
			and we obtain \(\pertslope\leq \bvi{}^{(k)}\) by taking the supremum of each side of the inclusion.
			This leads to \(\pertslope\in[0, \bvi{}^{(k)}]\subset[0,\domb[\subseteq\dom(\conj\pertfunc)\) where the last inclusion holds by~\eqref{eq:proof prop pertslope domain:ingredient 2}.
		\end{itemize}
		\item Item \ref{eq:prop:conj of h at pertslope leq reg} 
		Let \(\{\bvi{}^{(k)}\}_{k\in\kN}\) be a non-decreasing sequence of \([0, \pertslope[\) which converges to \(\pertslope\).
		Using item~\ref{eq:prop:pertslope leq domb}, we note that such a sequence exists since \(\pertslope>0\) as soon as \(\pertfunc\) is not identically zero on \(\kR+\). 
		Since \(\conj{\pertfunc}\) is non-decreasing on \(\dom{\conj{\pertfunc}}\cap\kR+\) from \Cref{cor:monoticity} under~\ref{assumption:zero-minimized}, \(\{\conj{\pertfunc}(\bvi{}^{(k)})\}_{k\in\kN}\) is also non-decreasing. 
		Moreover, we have that \(\conj{\pertfunc}(\bvi{}^{(k)}) \leq \reg\) for all \(k\in\kN\) by definition of \(\pertslope\) in \eqref{eq:pertslope}. 
		Hence, \(\{\conj{\pertfunc}(\bvi{}^{(k)})\}_{k\in\kN}\) is a non-decreasing upper-bounded sequence and thus converges to a limit.
		In addition, this limit satisfies:
		\begin{equation} \label{eq:proof prop:conj of h at pertslope leq reg:ingredient 1}
			\lim_{k\to+\infty} \conj{\pertfunc}(\bvi{}^{(k)}) \leq \reg
			.
		\end{equation}
		Since \(\conj{\pertfunc}\) is continuous on its domain under~\ref{assumption:zero-minimized} from \Cref{cor:continuity} and that \(\pertslope\in\dom{\conj{\pertfunc}}\) by item~\ref{eq:prop:pertslope in domain}, we deduce that
		\begin{equation} \label{eq:proof prop:conj of h at pertslope leq reg:ingredient 2}
			\lim_{k\to+\infty} \conj{\pertfunc}(\bvi{}^{(k)})
			=
			\conj{\pertfunc}(\pertslope)
			.
		\end{equation}
		The result 
		follows by combining~\eqref{eq:proof prop:conj of h at pertslope leq reg:ingredient 1} and~\eqref{eq:proof prop:conj of h at pertslope leq reg:ingredient 2}.
		\item Item \ref{eq:prop:conj of h at pertslope = domb} The two implications can be shown as follows:
		\begin{itemize}
			\item[(\(\Rightarrow\))] Assume that \(\bvi{}\in\kR+\) is such that \(\conj{\pertfunc}(\bvi{})=\reg\). 
			By definition of \(\pertslope\), we must have \(\bvi{}\leq \pertslope\).
			Moreover, under~\ref{assumption:zero-minimized} we have from \Cref{cor:monoticity} that \(\bvi{}'>\bvi{}\) implies \(\conj{\pertfunc}(\bvi{}')>\lambda\).
			Since \(\conj{\pertfunc}(\pertslope)\leq\lambda\) from item~\ref{eq:prop:conj of h at pertslope leq reg}, we thus obtain by contraposition that \(\pertslope \leq \bvi{}\).
			This leads to \(\bvi{}=\pertslope\).
			\item[(\(\Leftarrow\))] Assume that \(\bvi{}=\pertslope\) with \(\pertslope<\domb\).
			Combined with item~\ref{eq:prop:pertslope leq domb}, the assumption ``\(\pertslope<\domb\)'' implies that there exists \(\bvi{+}\in\dom{\conj{\pertfunc}}\) such that \(\pertslope < \bvi{+}\).
			By definition of \(\pertslope\), we must have \(\conj{\pertfunc}(\bvi{+}) > \reg\). 
			Since \(\conj{\pertfunc}(0)=0\) from \Cref{cor:nonneg} and \(\conj{\pertfunc}\) is continuous non-decreasing on its domain from \Cref{cor:continuity,cor:monoticity}, \(\exists\,\bvi{\reg}\leq \bvi{+}\) such that \(\conj{\pertfunc}(\bvi{\reg})=\reg\).  
			Finally, using the same arguments as in the direct part, we must have \(\bvi{\reg}=\pertslope\), which leads to the desired result. 
		\end{itemize}
	\end{itemize}
\end{proof}

\subsection{Properties of \texorpdfstring{\(\pertlimit\)}{mu}}
\label{sec:mu}

This section establishes some properties of parameter \(\pertlimit\) defined in~\eqref{eq:pertlimit}. 
\begin{lemma}\label{lemma:properties-related-to-mu}
	If \ref{assumption:zero-minimized}-\ref{assumption:closed}-\ref{assumption:convex} hold and \(\pertlimit<+\infty\), then 
	\begin{enumerate}[label=\textit{\alph*.},topsep = 3pt]
		\item \label{item:prop:pertilimit in intrrior of domain if finite:subdif result}
		\(\pertlimit\in\partial\conj{\pertfunc}(\pertslope)\).
		\item \label{item:tau-in-interior-dom}
		\(0<\pertslope<\domb\).
		\item \label{item:prop:pertilimit in intrrior of domain if finite:decomposition}
		\(\pertfunc(\pertlimit) = \pertslope\pertlimit - \reg\).
	\end{enumerate}
\end{lemma}
\begin{proof}
	We prove each item separately.
	\begin{itemize}
		\item Item \ref{item:prop:pertilimit in intrrior of domain if finite:subdif result}
		By contraposition of \Cref{lemma:case-id-zero}, \(\pertlimit<+\infty\) implies that \(\pertfunc\) is not identically zero on \(\kR+\), and therefore \(\pertslope>0\) from item~\ref{eq:prop:pertslope leq domb} of \Cref{prop:pertslope in domain}. 
		Moreover, we have from item \ref{eq:prop:pertslope in domain}~of \Cref{prop:pertslope in domain} that \(\pertslope\in\dom{\conj{\pertfunc}}\).
		Since \(\conj{\pertfunc}\) is proper and convex under~\ref{assumption:zero-minimized} by \Cref{cor:proper convex closed}, we have from \Cref{lemma:subdiff-convex-fun-1D} that its (possibly empty) subdifferential at \(\pertslope\) writes 
		\begin{equation}
			\partial \conj{\pertfunc}(\pertslope) = 
			[\partial_-\conj{\pertfunc}(\pertslope), \partial_+\conj{\pertfunc}(\pertslope)] \cap \kR
			.
		\end{equation}
		Since \(\pertslope>0\), we have from property~\eqref{eq:cor:nonnegative subdiff} of \Cref{cor:subdiff properties} that \(\partial_-\conj{\pertfunc}(\pertslope)\geq0\), and therefore \(\partial_+\conj{\pertfunc}(\pertslope)\geq0\).
		In view of the definition of \(\pertlimit\) in~\eqref{eq:pertlimit}, and since \(\pertlimit<+\infty\) by hypothesis, we can conclude that \(\pertlimit=\partial_+\conj{\pertfunc}(\pertslope) \in\partial \conj{\pertfunc}(\pertslope)\).	
		\item Item \ref{item:tau-in-interior-dom} 
		We first recall that \(\conj{\pertfunc}\) is proper, closed and convex under~\ref{assumption:zero-minimized} from  \Cref{cor:proper convex closed}.
		Following the arguments in the proof of item~\ref{item:prop:pertilimit in intrrior of domain if finite:subdif result}, we have that \(\pertslope>0\) which establishes the first part of the result.
		In addition, the latter inequality leads to \(\interior{\dom{\conj{\pertfunc}}}\neq\emptyset\) as it guarantees that \(\interior{\dom{\conj{\pertfunc}}}\) contains the non-empty open interval \(\kintervoo{0}{\pertslope}\) by convexity of \(\dom{\conj{\pertfunc}}\).
		Since \(\pertlimit\) is assumed to be finite, we have that \(\partial\conj{\pertfunc}(\pertslope)\) is non-empty and bounded by item~\ref{item:prop:pertilimit in intrrior of domain if finite:subdif result} and definition of \(\pertlimit\) in~\eqref{eq:pertlimit}.
		The contraposition of~\cite[\propositionInRef{16.17.(i)}]{Bauschke2017} then yields \(\pertslope\in\interior{\dom{\conj{\pertfunc}}}\), and one finally obtains \(\pertslope<\domb\) using~\eqref{eq:int(dom-h*)} in \Cref{lemma:dom-domb}.
		\item Item \ref{item:prop:pertilimit in intrrior of domain if finite:decomposition}
		Since \(\pertfunc\) is proper, closed and convex from \ref{assumption:zero-minimized}-\ref{assumption:closed}-\ref{assumption:convex}, and \(\pertlimit\in\partial\conj{\pertfunc}(\pertslope)\) according to item~\ref{item:prop:pertilimit in intrrior of domain if finite:subdif result}, we have from~\cite[\theoremInRef{4.20}]{beck2017first} that
		\begin{equation}
			\pertfunc(\pertlimit)+\conj{\pertfunc}(\pertslope) = \pertslope\pertlimit
			.
		\end{equation}
		The result then follows from \Cref{prop:pertslope in domain}.\ref{eq:prop:conj of h at pertslope = domb} by noting that \(\pertslope<\domb\) from item~\ref{item:tau-in-interior-dom}
	\end{itemize}
\end{proof}

\begin{proposition} \label{lem:pertlimite > 0}
	If \ref{assumption:zero-minimized}-\ref{assumption:closed}-\ref{assumption:convex} hold, then \(\pertlimit>0\).
\end{proposition}
\begin{proof}
	If \(\pertlimit=+\infty\) then the result trivially holds.
	We thus restrict our attention to the case where \(\pertlimit<+\infty\).
	If \ref{assumption:zero-minimized}-\ref{assumption:closed}-\ref{assumption:convex}  hold and \(\pertlimit<+\infty\), we have that \(\pertfunc(\pertlimit) = \pertslope\pertlimit - \reg\)
	from \Cref{lemma:properties-related-to-mu}.\ref{item:prop:pertilimit in intrrior of domain if finite:decomposition}
	Since \(\pertfunc\geq 0\) from \ref{assumption:zero-minimized}, this leads to \(\pertslope\pertlimit \geq \reg\). Finally, since \(\pertslope>0\) from \Cref{lemma:properties-related-to-mu}.\ref{item:tau-in-interior-dom} and that \(\reg > 0\), we obtain  \(\pertlimit \geq \reg \pertslope^{-1}>0\). 
\end{proof}

\section{Characterization of \texorpdfstring{\(\conj{\pertfunc}\)}{h*}}
\label{app:characterization_of_conj_pertfunc}

In this section, we present a result characterizing the behavior of \(\conj{\pertfunc}\) on \(\kintervoo{0}{\domb}\):
\begin{lemma}\label{lemma:partition-int-dom}
	If~\ref{assumption:zero-minimized}-\ref{assumption:closed}-\ref{assumption:convex} hold and \(\pertfunc\) is not identically zero on \(\kR+\), then 
	\begin{equation} \label{eq:level-set}
		\forall \bvi{}\in]0,\domb[:\ 
		\conj{\pertfunc}({\bvi{}})		
        \begin{cases}
            < \reg  &\mbox{ if  \(\bvi{}\in \kintervoo{0}{\pertslope}\)}\\
            = \reg  &\mbox{ if  \(\bvi{}\in \{\pertslope\}\)}\\ 
            > \reg  &\mbox{ if \(\bvi{}\in \kintervoo{\pertslope}{\domb}\)}
            .
        \end{cases}		
	\end{equation}
\end{lemma}
\begin{proof}	
	Since~\ref{assumption:zero-minimized}-\ref{assumption:closed}-\ref{assumption:convex} hold and \(\pertfunc\) is not identically zero on \(\kR+\), we have that \(0<\pertslope\leq \domb\) from \Cref{prop:pertslope in domain}.\ref{eq:prop:pertslope leq domb}
	In particular, this ensures that \(\kintervoo{0}{\domb}\neq \emptyset\). 
	We next distinguish between the following cases.
	\begin{itemize}
		\item Case \(\pertslope<\domb\): In view of \Cref{prop:pertslope in domain}.\ref{eq:prop:conj of h at pertslope = domb}, we have that
		\(\conj{\pertfunc}(\pertslope)=\reg\) and \(\conj{\pertfunc}(\bvi{})\neq\reg\ \forall \bvi{}\in\kintervoo{0}{\domb}\setminus\{\pertslope\}\). The result is then a direct consequence of \Cref{cor:monoticity}.
		\item Case \(\pertslope=\domb\): In this case, the result particularizes to:  \(\conj{\pertfunc}(\bvi{}) < \reg\) \(\forall \bvi{}\in\kintervoo{0}{\domb}\). Hence, let \(\bvi{}\in ]0,\domb[\) and assume that \(\conj{\pertfunc}(\bvi{})>0\), since otherwise we trivially have \(0=\conj{\pertfunc}(\bvi{})<\reg\).   
		We then have \(\pertfunc(\bvi{})< \pertfunc(\domb)\leq\reg\) where we used the second part of \Cref{cor:monoticity} to obtain the first inequality, and \Cref{prop:pertslope in domain}.\ref{eq:prop:conj of h at pertslope leq reg} to obtain the second one.
	\end{itemize}
\end{proof}


\section{Characterization of \texorpdfstring{\(\conj\regfunc\)}{\expandafter\regfunc*}}  \label{app:proofs-g*}

In this appendix, we give the proofs of results claimed in the main text of the paper regarding the convex conjugate \(\conj\regfunc\) of the function \(\regfunc\) defined in~\eqref{eq:regfunc}-\eqref{eq:l0-norm}.

\subsection{Proof of \texorpdfstring{\Cref{prop:regfunc-conj}}{Proposition~\ref{prop:regfunc-conj}}} \label{proof:regfunc-conj}

From definitions~\eqref{eq:regfunc}-\eqref{eq:l0-norm}, we have that
\begin{align}
    \conj{\regfunc}(\bvi{}) &= \sup_{\pvi{} \in \kR} \bvi{}\pvi{} - \regfunc(\pvi{}) 
    \nonumber\\
        &= \sup_{\pvi{} \in \kR} \bvi{}\pvi{} - \pertfunc(\pvi{}) - \reg\norm{\pvi{}}{0} \nonumber\\
        &= \max \big\{
            \sup_{\pvi{} = 0} \bvi{}\pvi{} - \pertfunc(\pvi{}) - \reg\norm{\pvi{}}{0} , \ 
            \sup_{\pvi{} \neq 0} \bvi{}\pvi{} - \pertfunc(\pvi{}) - \reg\norm{\pvi{}}{0}
        \big\} \nonumber\\
        &= \max \big\{
            0
            , \ 
            \sup_{\pvi{} \neq 0} \bvi{}\pvi{} - \pertfunc(\pvi{}) - \reg 
        \big\}\nonumber
        \\
        &= \max \big\{
            0
            , \ 
            -\kparen{
                \inf_{\pvi{} \neq 0} \pertfunc(\pvi{}) - \bvi{}\pvi{}
            } - \reg 
        \big\}
        \label{eq:proof prop:regfunc-conj:decomposition g*}
        , 
\end{align}
where the fourth equality follows from the fact that \(\pertfunc(0)=0\) by \ref{assumption:zero-minimized}. 
Note that \(\pertfunc\) is proper under~\ref{assumption:zero-minimized} from \Cref{remark:properness} and closed, convex from \ref{assumption:closed}-\ref{assumption:convex}.
Hence, the function \(\aGenericOneDfunction:\pvi{}\mapsto\pertfunc(\pvi{})- \bvi{}\pvi{}\) defined on \(\kR\) is also proper, closed and convex.
Moreover, assumption~\ref{assumption:zero-minimized} ensures that \(\dom\pertfunc\) is not a singleton. 
Applying \Cref{lemma:infima open set = infima closed set} with \(\aGenericOneDfunction(\pvi{})=\pertfunc(\pvi{})-\bvi{}\pvi{}\) and \(\pvi{0}=0\) then leads to
\begin{equation}\label{eq:proof:g*:intermediate-goal}
    \inf_{\pvi{} \neq 0} \pertfunc(\pvi{}) - \bvi{}\pvi{}
    =
    \inf_{\pvi{}\in\kR} \pertfunc(\pvi{}) - \bvi{}\pvi{}
    =
    - \sup_{\pvi{} \in \kR}\ \bvi{}\pvi{} - \pertfunc(\pvi{})
    =
    -\conj\pertfunc(\bvi{})
    .
\end{equation}
Plugging this result into~\eqref{eq:proof prop:regfunc-conj:decomposition g*} leads to \(\conj{\regfunc}(\bvi{}) = \pospart{\conj{\pertfunc}(\bvi{}) - \reg}\). 

It remains to show that under~\ref{assumption:even}, we have
\begin{equation} \label{eq:proof:regfunc-conj:3}
    \conj{\regfunc}(\bvi{})=0
    \iff \abs{\bvi{}} \leq \pertslope
\end{equation}
for all $\bvi{} \in \kR$.
Since \(\conj{\pertfunc}\) is even under~\ref{assumption:even} according to \Cref{cor:even}, it is sufficient to establish equivalency \eqref{eq:proof:regfunc-conj:3} for \(\bvi{}\geq0\).
\begin{itemize}[leftmargin=30pt]
    \item[\((\Rightarrow)\)] We note that \(\kset{\bvi{}\in\kR+}{\pospart{\conj{\pertfunc}(\bvi{}) - \reg} = 0}=\kset{\bvi{}\in\kR+}{\conj{\pertfunc}(\bvi{}) \leq \reg}\).
    The direct implication then follows from the definition of \(\pertslope\) in~\eqref{eq:pertslope}. 
    \item[\((\Leftarrow)\)] Under~\ref{assumption:zero-minimized}, \(\conj{\pertfunc}\) is non-negative and non-decreasing on \(\dom{\conj{\pertfunc}}\cap \kR+\) from \Cref{cor:monoticity,cor:nonneg}.
    Therefore, we have \(0\leq \bvi{}\leq\pertslope \implies 0\leq\conj{\pertfunc}(\bvi{})\leq\conj{\pertfunc}(\pertslope)\). 
    Finally, if \ref{assumption:zero-minimized}-\ref{assumption:closed}-\ref{assumption:convex} hold, we have from \Cref{prop:pertslope in domain}.\ref{eq:prop:conj of h at pertslope leq reg} that \(\conj{\pertfunc}(\pertslope)\leq\reg\).
\end{itemize}

\subsection{Proof of \texorpdfstring{\Cref{prop:subdiff-conjugate of g}}{Proposition~\ref{prop:subdiff-conjugate of g}}} \label{proof:subdiff-g*}

We first emphasize the following direct consequence of the expression of \(\conj{\regfunc}\) in \Cref{prop:regfunc-conj}: 
\begin{corollary} \label{cor:equality domain gstar and hstar}
    If \ref{assumption:zero-minimized}-\ref{assumption:closed}-\ref{assumption:convex} hold, then \(\dom{\conj{\regfunc}} = \dom{\conj{\pertfunc}}\).
\end{corollary}

We now turn to the proof of \Cref{prop:subdiff-conjugate of g}.
As an initial remark, we mention that certain parts of the proof involve the parameter \(\domb\), which is defined in \eqref{eq:domb}.
First, let us consider the case where \(\pertfunc\) is identically zero.
Using \Cref{lemma:case-id-zero} and \Cref{cor:even}, we obtain that \(\conj{\pertfunc}(\bvi{}) = \icvx(\bvi{}=0)\) and therefore 
\(\dom{\conj{\regfunc}}=\{0\}\). Standard subdifferential calculus then leads to 
\(\partial\conj{\regfunc}(0)=\kR\), see \eg, \cite[\exampleInRef{3.5}]{beck2017first} with \(S=\{0\}\). 
This corresponds to \eqref{eq:subdiff-g*-h=0}. 
It thus remains to consider the general case where \(\pertfunc\) is not identically zero and show that \eqref{eq:subdiff-g*} holds.
To simplify our reasoning, let us note the following properties: 
\begin{enumerate}[label=\textit{\alph*.}]
    \item Since~\ref{assumption:zero-minimized} holds, \(\conj{\regfunc}\) is proper, closed and convex from \Cref{cor:proper convex closed}.     
    Hence, \(\forall \bvi{}\in\dom{\conj{\regfunc}}\),  \(\subdiff\conj{\regfunc}(\bvi{})\) is a (possibly empty) closed and convex subset of \(\kR\) by \Cref{lemma:subdiff-convex-fun-1D}.
    Moreover, since \(\conj{\regfunc}{}\) is even from \Cref{cor:even} under~\ref{assumption:even}, then \(-\bvi{}\in\dom{\conj{\regfunc}}\) and the definition of a subgradient implies
    \begin{equation} \label{eq:proof prop:subdiff-conjugate of g:symmetry}
        \pvi{}\in\subdiff\conj{\regfunc}(\bvi{}) \iff -\pvi{}\in\subdiff\conj{\regfunc}(-\bvi{})
        .
    \end{equation}
    Hence, it is sufficient to concentrate on the case where \(\bvi{}\in\dom{\conj{\regfunc}}\cap\kR+\).
    \item We deduce from \Cref{lemma:dom-domb} and \Cref{cor:equality domain gstar and hstar} that 
    \begin{equation}
        [0,\domb[ \subseteq \dom{\conj{\regfunc}}\cap\kR+ \subseteq [0,\domb]\cap \kR+.
    \end{equation}
    Hence, it is sufficient to concentrate on the two case \(\bvi{}\in[0,\domb[\) and \(\bvi{}= \domb \in \dom{\conj{\regfunc}}\). 
    We note that \(\domb>0\) from \Cref{lemma:domb result} since we assume that \(\pertfunc\) is not identically zero. The first case is thus never empty. The second case may occur depending on the nature of \(\pertfunc\).
\end{enumerate}

~\\
In view of the above remarks, we thus have two cases to treat.
~\\

\begin{itemize}
    \item Case \(\bvi{}\in\kintervco{0}{\domb}\): 
    Since \(\pertfunc\) is not identically zero and~\ref{assumption:zero-minimized}-\ref{assumption:closed}-\ref{assumption:convex} hold, we have \(\pertslope>0\) from \Cref{prop:pertslope in domain}.\ref{eq:prop:pertslope leq domb} and we can partition \(\kintervco{0}{\domb}\) as
    \begin{equation}
        \kintervco{0}{\domb} \,=\, \kintervco{0}{\pertslope} 
        \, 
        \cup
        \, 
        \{\pertslope\}\setminus\{\domb\}
        \, 
        \cup
        \, 
        \kintervoo{\pertslope}{\domb}
        .
    \end{equation}
    Under \ref{assumption:zero-minimized}-\ref{assumption:closed}-\ref{assumption:convex},  we also have from \Cref{prop:regfunc-conj} that \(\conj{\regfunc}\) is the point-wise maximum of two proper convex functions, namely \(\bvi{} \mapsto 0\) and \(\bvi{} \mapsto \conj{\pertfunc}(\bvi{}) - \reg\) with domain \(\kR\) and \(\dom{\conj{\pertfunc}}=\dom{\conj\regfunc}\) from \Cref{cor:equality domain gstar and hstar}, respectively.
    Combining the result in \cite[\theoremInRef{3.50}]{beck2017first} and \Cref{lemma:partition-int-dom}, we then obtain: 
    \begin{equation}\label{eq:subdiff-g*-interior}
        \forall \bvi{} \in \kintervoo{0}{\domb}:\ 
        \subdiff \conj{\regfunc}(\bvi{}) = 
        \begin{cases}
            \{0\} &\text{if} \ \bvi{} < \pertslope \\
            \conv(\{0\}\cup \subdiff\conj{\pertfunc}(\pertslope)) &\text{if} \ \bvi{} = \pertslope \\
            \subdiff\conj{\pertfunc}(\bvi{})  &\text{if} \ \bvi{} > \pertslope.
        \end{cases}
    \end{equation}
    We next give an alternative expression to the subdifferential for the cases \(\conj{\regfunc}\) at \(\bvi{}=\pertslope\) and will treat the case \(\bvi{}=0\) separately. 
    \begin{itemize}[label=$\bullet$]
        \item Regarding the expression of \(\partial\conj{\regfunc}(\pertslope)\), let us first recall that $0<\pertslope<\domb$ where the second inequality is a consequence of our case hypothesis. Using, \Cref{prop:pertslope in domain}.\ref{eq:prop:pertslope in domain} and \Cref{cor:equality domain gstar and hstar}, we thus obtain that $\pertslope\in\interior{\dom{\conj\regfunc}}$. Since \(\conj{\pertfunc}\) is proper, closed and convex under~\ref{assumption:zero-minimized} by \Cref{cor:proper convex closed}, we have by applying \Cref{lemma:properties-left-right-derivatives V2}.\ref{item-lemma:properties-left-right-derivatives:finite on interior of the domain} together with \Cref{lemma:subdiff-convex-fun-1D} that \(\subdiff\conj{\pertfunc}(\pertslope)\) is a nonempty compact interval of \(\kR\).
        Therefore, the definition of \(\pertlimit\) implies that \(\pertlimit=\partial_+\conj{\pertfunc}(\pertslope)<+\infty\). 
        Finally, since \(\pertslope>0\), item~\eqref{eq:cor:nonnegative subdiff} of \Cref{cor:subdiff properties} implies that \(\partial_-\conj{\pertfunc}(\pertslope)\geq0\). Therefore, we obtain
        \begin{equation}
            \conv(\{0\}\cup\subdiff\conj{\pertfunc}(\pertslope))
            =
            \kintervcc{0}{\pertlimit} = \kintervcc{0}{\pertlimit}\cap \kR
        \end{equation}
        where the last equality follows from \(\pertlimit<+\infty\).
        \item Let now treat the case \(\bvi{}=0\). 
        Since \(\pertslope>0\), there exists \(\bvi{+}\in\kintervoo{0}{\domb}\) such that \(0<\bvi{+}<\pertslope\). We then have from \Cref{lemma:subdiff-convex-fun-1D-is-nondecreasing} that \(\sup\ \partial\conj{\regfunc}(0) \leq \inf \ \partial\conj{\regfunc}(\bvi{+}) = 0\) where the last equality holds in view of \eqref{eq:subdiff-g*-interior}. 
        Using the fact that \(0\in\partial\conj{\regfunc}(0)\) by item~\eqref{eq:cor:0 in subdiff} of \Cref{cor:subdiff properties}, we obtain \(\partial\conj{\regfunc}(0)\cap\kR+=\{0\}\).
        We finally deduce that \(\partial\conj{\regfunc}(0)=\{0\}\) by using~\eqref{eq:proof prop:subdiff-conjugate of g:symmetry}.
    \end{itemize}
    \item Case \(\bvi{}= \domb \in\dom{\conj{\regfunc}}\): We observe that \(0<\beta<+\infty\). The first inequality follows from \Cref{prop:pertslope in domain}.\ref{eq:prop:pertslope leq domb} and the fact that \(\pertfunc\) is not identically zero by hypothesis. The second inequality is a consequence of \(\domb \in\dom{\conj{\regfunc}}\). By definition of \(\domb\), we have \(\bvi{}'\notin\dom{\conj{\regfunc}}=\dom{\conj{\pertfunc}}\) for any \(\bvi{}'>\domb\). Using \Cref{lemma:subdiff-convex-fun-1D} and \Cref{definition:left-right-derivatives}, we then obtain:
    \begin{align} 
        \subdiff\conj{\pertfunc}(\domb) \,=\,& \kintervcc{\partial_-\conj{\pertfunc}(\domb)}{+\infty} \cap \kR 
        \label{eq:form-subdiff-h*-boundary}
        \\
        \subdiff\conj{\regfunc}(\domb) \,=\,& \kintervcc{\partial_-\conj{\regfunc}(\domb)}{+\infty} \cap \kR 
        \label{eq:form-subdiff-g*-boundary}
        .
    \end{align}
    Moreover, since \(\domb\in\dom{\conj{\regfunc}}\), we have that \(\subdiff_-\conj{\regfunc}(\domb)\) is well defined (but not necessarily finite) from \Cref{lemma:properties-left-right-derivatives V2}.\ref{item-lemma:properties-left-right-derivatives:limits exists on domain} 
    Since \(\dom{\conj{\pertfunc}}=\dom{\conj{\regfunc}}\), the same holds for \(\subdiff_-\conj{\pertfunc}(\domb)\).
    We next distinguish between two cases depending on the value of \(\pertslope\). 
    \begin{itemize}[label=$\bullet$]
        \item If \(\pertslope < \domb\), we have for any sequence of scalars \(\{\varepsilon^{(k)}\}_{k\in\kN}\) that converges to \(0\) and such that \(\domb + \varepsilon^{(k)}\in\kintervoo{\pertslope}{\domb}\) \(\forall k\in\kN\):
        \begin{align}
            \partial_-\conj{\regfunc}(\domb)
            \,=\,&
            \lim_{k\to+\infty}
            \frac{
                \conj{\regfunc}(\domb + \varepsilon^{(k)}) - \conj{\regfunc}(\domb)
            }{
                \varepsilon^{(k)}
            }
            \nonumber \\
            \,=\,&
            \lim_{k\to+\infty}
            \frac{
                \conj{\pertfunc}(\domb + \varepsilon^{(k)}) - \reg - \conj{\pertfunc}(\domb) + \reg
            }{
                \varepsilon^{(k)}
            }
            \,=\,
            \partial_-\conj{\pertfunc}(\domb)
            \label{eq:proof form-subdiff-g*-boundary:case beta>tau equality leftderivative at beta}
        \end{align}
        where the second line holds in view of the definition of \(\{\varepsilon^{(k)}\}_{k\in\kN}\) and \Cref{prop:regfunc-conj}. 
        Combining \eqref{eq:form-subdiff-h*-boundary}, \eqref{eq:form-subdiff-g*-boundary} and \eqref{eq:proof form-subdiff-g*-boundary:case beta>tau equality leftderivative at beta}, we thus have \(\subdiff\conj{\regfunc}(\domb)=\subdiff\conj{\pertfunc}(\domb)\) as mentioned in the last case of \eqref{eq:subdiff-g*}. 
        \item If \(\pertslope=\domb\), we need to show that the second case in \eqref{eq:subdiff-g*} holds. 
        On the one hand, using \eqref{eq:form-subdiff-h*-boundary} and the fact that \(\pertslope=\domb\), the definition of \(\pertlimit\) in~\eqref{eq:pertlimit} leads {in both cases ``\(\partial_-\conj{\pertfunc}(\pertslope)\in\kR\)'' and ``\(\partial_-\conj{\pertfunc}(\pertslope)=+\infty\)''} to \(\pertlimit=+\infty\).
        The second case in \eqref{eq:subdiff-g*} thus particularizes to \(\subdiff\conj{\regfunc}(\domb)=\kintervcc{0}{+\infty}\cap\kR\). 
        Taking \eqref{eq:form-subdiff-g*-boundary} into account, it thus remains to show that \(\partial_-\conj{\regfunc}(\domb)=0\) to prove the result. 
        Now, 
        \begin{equation} \label{eq:proof form-subdiff-g*-boundary:case beta=tau equality leftderivative at beta}
            \partial_-\conj{\regfunc}(\domb)
            =
            \lim_{\varepsilon\uparrow 0}
            \frac{
                \conj{\regfunc}(\pertslope + \varepsilon) - \conj{\regfunc}(\pertslope)
            }{
                \varepsilon
            }
            = 
            0
        \end{equation}
       where the last equality follows from the second part of \Cref{prop:regfunc-conj}. 
    \end{itemize}
\end{itemize}

\subsection{Proof of \texorpdfstring{\Cref{prop:operator-prox-gstar}}{Proposition~\ref{prop:operator-prox-gstar}}} \label{proof:operator-prox-g*}

The proximal operator of \(\cstprox\conj{\regfunc}\) is the operator defined for all \(\bvi{}\in\kR\) by
\begin{equation} \label{eq:def-prox-g*}
    \prox_{\cstprox\conj{\regfunc}}(\bvi{}) = \argmin_{\bvi{}' \in \kR} \cstprox\conj{\regfunc}(\bvi{}') + \tfrac{1}{2}(\bvi{} - \bvi{}')^2
    .
\end{equation}
Since \(\conj{\regfunc}\) is proper, closed and convex under~\ref{assumption:zero-minimized} from \Cref{cor:proper convex closed}, the objective function in \eqref{eq:def-prox-g*} admits a unique minimizer for each value of \(\bvi{}\in\kR\), according to \cite[\theoremInRef{6.3}]{beck2017first}. 
Moreover, since \(\conj{\regfunc}\) is proper, this minimizer must verify Fermat's necessary and sufficient optimality condition, that is 
\begin{equation} \label{eq:proof:operator-prox-gstar:corollary second prox theorem-B}
    \kinv{\cstprox}(\bvi{} - \prox_{\cstprox\conj{\regfunc}}(\bvi{})) \in \subdiff\conj{\regfunc}(\prox_{\cstprox\conj{\regfunc}}(\bvi{}))
\end{equation}    
in view of \cite[\theoremInRef{6.39}]{beck2017first}. 
Now, since \(\conj{\regfunc}\) is even under our assumption~\ref{assumption:even} from \Cref{cor:even}, we have \(\subdiff\conj{\regfunc}(-\bvi{})=-\subdiff\conj{\regfunc}(\bvi{})\) and \eqref{eq:proof:operator-prox-gstar:corollary second prox theorem-B} yields:
\begin{equation}
    \prox_{\cstprox\conj{\regfunc}}(-\bvi{}) = - \prox_{\cstprox\conj{\regfunc}}(\bvi{}).
\end{equation}
It is thus sufficient to focus on the case \(\bvi{}\in\kR+\) and distinguish between the following subcases \(\bvi{}\in\kintervcc{0}{\pertslope}\), \(\bvi{}\in\kintervoc{\pertslope}{\pertslope+\cstprox\pertlimit}\cap\kR\), \(\bvi{}\in\kintervoo{\pertslope+\cstprox\pertlimit}{+\infty}\cap\kR\).
\begin{itemize}
    \item Case \(\bvi{}\in\kintervcc{0}{\pertslope}\): If \(\pertfunc\) is not identically zero, we note from the first two cases in~\eqref{eq:subdiff-g*} that \(\kinv{\cstprox}(\bvi{} - \bvi{})=0\in\partial\conj{\regfunc}(\bvi{})\). We thus have that \(\prox_{\cstprox\conj{\regfunc}}(\bvi{}) = \bvi{}\) from~\eqref{eq:proof:operator-prox-gstar:corollary second prox theorem-B}. If \(\pertfunc\) is identically zero, we have from \Cref{lemma:case-id-zero} that \(\pertslope=0\) so that \(\kintervcc{0}{\pertslope}=\{0\}\). In this case, we also have \(0\in\partial\conj{\regfunc}(0)\) from \eqref{eq:subdiff-g*-h=0} and therefore \(\prox_{\cstprox\conj{\regfunc}}(\bvi{}) = \bvi{}\) verifies \eqref{eq:proof:operator-prox-gstar:corollary second prox theorem-B} for \(\bvi{}=0\) as well.
    \item Case \(\bvi{}\in\kintervoc{\pertslope}{\pertslope+\cstprox\pertlimit}\cap\kR\): 
    This is equivalent to \(\kinv{\cstprox}(\bvi{} - \pertslope) \in ]0, \pertlimit]\cap\kR\).
    If \(\pertfunc\) is not identically zero, using the second case in~\eqref{eq:subdiff-g*}, we have that \(\kinv{\cstprox}(\bvi{} - \pertslope) \in \partial\conj{\regfunc}(\pertslope)\). The latter expression is equivalent to \(\prox_{\cstprox\conj{\regfunc}}(\bvi{}) = \pertslope\) in view of~\eqref{eq:proof:operator-prox-gstar:corollary second prox theorem-B}. If \(\pertfunc\) is identically zero, we have from \Cref{lemma:case-id-zero} that \(\pertslope=0\), \(\pertlimit=+\infty\) and therefore \(\kintervoc{\pertslope}{\pertslope+\cstprox\pertlimit}\cap\kR=\kR+\). Since \(\subdiff\conj\regfunc(0)=\kR\) from \eqref{eq:subdiff-g*-h=0}, we have that \(\prox_{\cstprox\conj{\regfunc}}(\bvi{})=0\) verifies \eqref{eq:proof:operator-prox-gstar:corollary second prox theorem-B}.
    \item Case \(\bvi{}\in\kintervoo{\pertslope+\cstprox\pertlimit}{+\infty}\cap\kR\): 
    Since \(\pertslope\in\kR+\) from \Cref{lemma:pertslope finite}, we first mention that this case necessarily implies that \(\pertlimit<+\infty\) since we must have \(\bvi{}\in\kR\). 
    In particular, \(\pertfunc\) cannot be identically zero if this case occurs.     
    Using the firm non-expansivity of proximal operators (see \cite[\theoremInRef{6.42.(a)}]{beck2017first} with \(\mathbf{x}=\bvi{}\) and \(\mathbf{y}=\pertslope+\cstprox\pertlimit\)), we obtain
    \begin{equation}
        (\bvi{}-(\pertslope+\cstprox\pertlimit))(\prox_{\cstprox\conj{\regfunc}}(\bvi{})-\pertslope)\geq \|\prox_{\cstprox\conj{\regfunc}}(\bvi{})-\pertslope\|_2^2
    \end{equation}
    where we used the fact that \(\prox_{\cstprox\conj{\regfunc}}(\pertslope+\cstprox\pertlimit)=\pertslope\) from the previous case. Since \(\bvi{}-(\pertslope+\cstprox\pertlimit)>0\) by definition, this leads to \(\prox_{\cstprox\conj{\regfunc}}(\bvi{})\geq \pertslope\).
    We, assume that \(\prox_{\cstprox\conj{\regfunc}}(\bvi{})=\pertslope\). Then, we have from \eqref{eq:proof:operator-prox-gstar:corollary second prox theorem-B} and \Cref{prop:subdiff-conjugate of g} that \(\kinv{\cstprox}(\bvi{} - \pertslope)\in \kintervcc{0}{\pertlimit}\), which leads to \(\bvi{}\leq \pertslope+\cstprox\pertlimit\). 
    This is in contradiction with our initial hypothesis ``\(\bvi{}> \pertslope+\cstprox\pertlimit\)'' and therefore \(\prox_{\cstprox\conj{\regfunc}}(\bvi{}) > \pertslope\). 
    In view of this strict inequality, we have from \Cref{prop:subdiff-conjugate of g} that
    \begin{equation}
        \partial\conj{\regfunc}(\prox_{\conj{\regfunc}}(\bvi{})) = \partial\conj{\pertfunc}(\prox_{\conj{\regfunc}}(\bvi{})).
    \end{equation}
    Plugging this expression into~\eqref{eq:proof:operator-prox-gstar:corollary second prox theorem-B} then leads to 
    \begin{equation}
        \kinv{\cstprox}(\bvi{} - \prox_{\cstprox\conj{\regfunc}}(\bvi{}))
    \in \cstprox\partial\conj{\pertfunc}(\prox_{\cstprox\conj{\regfunc}}(\bvi{})),
    \end{equation}
    which is equivalent to \(\prox_{\cstprox\conj{\regfunc}}(\bvi{}) = \prox_{\cstprox\conj{\pertfunc}}(\bvi{})\) by virtue of \cite[\theoremInRef{6.39}]{beck2017first}.
\end{itemize}


\section{Characterization of \texorpdfstring{\(\biconj\regfunc\)}{\expandafter\regfunc**}} \label{app:proofs-g**}

In this appendix, we give the proofs of results claimed in the main text of the paper regarding the convex biconjugate \(\biconj\regfunc\) of the function \(\regfunc\) defined in~\eqref{eq:regfunc}-\eqref{eq:l0-norm}.

\subsection{Proof of \texorpdfstring{\Cref{prop:regfunc-biconj}}{Proposition~\ref{prop:regfunc-biconj}} }
\label{proof:regfunc-biconj}

Let first characterize the domain of \(\biconj{\regfunc}\).
By definition of the $\ell_0$-norm given in \eqref{eq:l0-norm}, we observe that \(\pertfunc\leq \regfunc \leq \reg + \pertfunc\).
Applying \cite[\propositionInRef{13.16.(ii)}]{Bauschke2017} twice, we thus obtain that \(\biconj{\pertfunc}\leq \biconj{\regfunc} \leq \reg + \biconj{\pertfunc}\) and immediately deduce that \(\dom{\biconj{\regfunc}}=\dom{\biconj{\pertfunc}}\).
One then concludes that
\begin{equation}
    \dom{\biconj\regfunc}=\dom\pertfunc
\end{equation}
by noting that \(\biconj{\pertfunc}=\pertfunc\) from \cite[\theoremInRef{4.8}]{beck2017first} since \(\pertfunc\) is closed, proper and convex under~\ref{assumption:zero-minimized}-\ref{assumption:closed}-\ref{assumption:convex}.

We can now establish the correctness of the closed-form expression of \(\biconj{\regfunc}\) given in~\eqref{eq:regfunc-biconj}.
As an initial remark, we recall that \(\conj{\pertfunc}\), \(\conj{\regfunc}\) and \(\biconj{\regfunc}\) are proper, closed and convex functions under \ref{assumption:zero-minimized} from \Cref{cor:proper convex closed}. 
Moreover, \(\biconj{\regfunc}\) is even under \ref{assumption:even} from \Cref{cor:even}. 
In the sequel, we thus restrict our attention to \(\pvi{}\in\kR+\). 
We divide our proof into several cases depending on the value of \(\pvi{}\). 
\begin{itemize}
    \item Case \(\pvi{}\in\kintervcc{0}{\pertlimit}\cap \kR\): We have \(\pvi{}\in\partial\conj{\regfunc}(\pertslope)\) from \Cref{prop:subdiff-conjugate of g}. 
    Since \(\conj{\regfunc}\) is proper, closed and convex, we obtain from \cite[\theoremInRef{4.20}]{beck2017first} that 
        \begin{align}\label{eq:expr-g**-x<=mu}
            \biconj{\regfunc}(\pvi{})= \pertslope\pvi{} - \conj{\regfunc}(\pertslope)=\pertslope\pvi{}
        \end{align}
    where the last equality follows from \Cref{prop:regfunc-conj}. This shows the first case in \eqref{eq:regfunc-biconj}.  
    We note that, in view of \Cref{lemma:case-id-zero}, this case also encompasses the scenario where $\pertfunc$ is identically zero, which does therefore not require any special treatment. 
    \item Case \(\pvi{}\geq \pertlimit\): We have to show that \(\biconj{\regfunc}(\pvi{})={\pertfunc}(\pvi{})+\reg\).  
    We note that since our statement assumes \(\pvi{}\in\kR\), \(\pvi{}\geq \pertlimit\) implies that \(\pertlimit<+\infty\) since a contradiction occurs otherwise.
    In particular, \(\pertfunc\) is not identically equal to zero, see \Cref{lemma:case-id-zero}. 
    If \(\pvi{}=\pertlimit\), we have from \eqref{eq:expr-g**-x<=mu} that \(\biconj\regfunc(\pertlimit)=\pertslope\pertlimit\).
    Since \(\pertlimit<+\infty\), we further observe  from item~\ref{item:prop:pertilimit in intrrior of domain if finite:decomposition}  of \Cref{lemma:properties-related-to-mu} that \(\pertfunc(\pertlimit)=\pertslope\pertlimit-\reg\), which shows the result in the case \(\pvi{}=\pertlimit\).
    In the rest of the proof, we thus concentrate on \(\pvi{}> \pertlimit\). 
    We distinguish between three subcases.
    \begin{itemize}[label=\(\bullet\)]
        \item Case \(\pvi{}\in\interior{\dom{\biconj{\regfunc}}}\): 
        Since \(\biconj{\regfunc}\) is proper convex, and \(\pvi{}\) belongs to the interior of the domain, we have from \cite[\theoremInRef{3.14}]{beck2017first} that \(\partial\biconj{\regfunc}(\pvi{})\neq\emptyset\). 
        Moreover, if \(\bvi{}\in\partial\biconj{\regfunc}(\pvi{})\), it follows from~\cite[\theoremInRef{4.20}]{beck2017first} that 
            \begin{align}
                \pvi{} &\in \subdiff\conj{\regfunc}(\bvi{}) \label{eq:cond-FY-A} \\
                \biconj{\regfunc}(\pvi{}) &= \bvi{}\pvi{} - \conj{\regfunc}(\bvi{}) \label{eq:cond-FY-B} 
            \end{align} 
        since \(\conj{\regfunc}\) is proper, closed and convex. 
        \eqref{eq:cond-FY-A} leads to \(\bvi{}>\pertslope\) since otherwise we have from \Cref{prop:subdiff-conjugate of g} that our working assumption ``\(\pvi{} > \pertlimit\)'' is violated. 
        Using the expression of \(\conj{\regfunc}(\bvi{})\) for \(\bvi{}>\pertslope\) in \Cref{prop:regfunc-conj},~\eqref{eq:cond-FY-B} then reduces to
        \begin{equation}
            \biconj{\regfunc}(\pvi{}) = \bvi{}\pvi{} + \reg - \conj{\pertfunc}(\bvi{})
            . 
        \end{equation}
        It thus remains to show that \(\bvi{}\pvi{} - \conj{\pertfunc}(\bvi{})=\pertfunc(\pvi{})\) to prove the result. 
        We first note that \(\partial\conj{\regfunc}(\bvi{})=\partial\conj{\pertfunc}(\bvi{})\) by \Cref{prop:subdiff-conjugate of g} since \(\bvi{}>\pertslope\). Combining this observation with \eqref{eq:cond-FY-A}, we then obtain \(\pvi{}\in\partial\conj{\pertfunc}(\bvi{})\). Since \(\conj{\pertfunc}\) is proper, closed and convex, applying \cite[\theoremInRef{4.20}]{beck2017first} leads to \(\biconj{\pertfunc}(\pvi{}) = \bvi{}\pvi{} - \conj{\pertfunc}(\bvi{})\). 
        We finally obtain the desired result by noticing that \(\biconj{\pertfunc}(\pvi{}) = \pertfunc(\pvi{})\) as \(\pertfunc\) is proper, closed and convex under~\ref{assumption:zero-minimized}-\ref{assumption:closed}-\ref{assumption:convex}. 
        \item Case \(\pvi{}\in\dom{\biconj{\regfunc}}\setminus\interior{\dom{\biconj{\regfunc}}}\): We will use the fact that \(\pertfunc\) and \(\biconj{\regfunc}\) are continuous on their domain under \ref{assumption:zero-minimized}-\ref{assumption:closed}-\ref{assumption:convex} (see \Cref{cor:continuity}), and \(\dom{\biconj{\regfunc}}=\dom\pertfunc\) (see first part of the result). 
        Specifically, let \(\{\pvi{}^{(k)}\}_{k\in\kN}\) be a sequence of non-negative scalars that converges to \(\pvi{}\) and such that \(\pertlimit < \pvi{}^{(k)} < \pvi{}\) for all \(k\in\kN\).
        According to the previous case, such a construction implies that \(\biconj{\regfunc}(\pvi{}^{(k)}) = \pertfunc(\pvi{}^{(k)}) + \reg\)  for all \(k\in\kN\).
        Taking the limit in both sides of the equality and using the continuity of \(\pertfunc\) and \(\biconj{\regfunc}\) over \(\dom{\biconj{\regfunc}}=\dom\pertfunc\), we obtain: 
        \begin{equation}
            \biconj{\regfunc}(\pvi{})
            =
            \lim_{k\to+\infty}
            \biconj{\regfunc}(\pvi{}^{(k)})
            =
            \lim_{k\to+\infty}
            \pertfunc(\pvi{}^{(k)}) + \reg
            =
            \pertfunc(\pvi{}) + \reg.
        \end{equation}
        \item Case \(\pvi{}\notin\dom{\biconj{\regfunc}}\): Since \(\dom{\biconj{\regfunc}}=\dom\pertfunc\), we have that \(\pvi{}\notin\dom\pertfunc\) and therefore \(+\infty=\biconj{\regfunc}(\pvi{})=\pertfunc(\pvi{})+\reg\).
    \end{itemize}
\end{itemize}

\subsection{Proof of \texorpdfstring{\Cref{prop:subdiff-gbiconj}}{Proposition~\ref{prop:subdiff-gbiconj}}} \label{proof:subdiff-g**}

Since \(\biconj{\regfunc}{}\) is even under \ref{assumption:even} from \Cref{cor:even}, the definition of a subgradient implies
\begin{equation} \label{eq:proof prop:subdiff-biconjugate of g:symmetry}
    \bvi{}\in\subdiff\biconj{\regfunc}(\pvi{}) \iff -\bvi{}\in\subdiff\biconj{\regfunc}(-\pvi{})
    .
\end{equation}
In the sequel, we thus concentrate on the case where \(\pvi{}\geq0\). We consider separately the case where \(\pertfunc\) is identically zero on \(\kR+\). 

\paragraph{\(\pertfunc\) is identically zero on \(\kR+\).}
In view of \Cref{prop:regfunc-biconj} and \Cref{lemma:case-id-zero}, we have that \(\biconj{\regfunc}\) is also identically zero on \(\kR+\).
\(\biconj{\regfunc}\) is thus differentiable for any \(\pvi{}>0\) with \(\subdiff\biconj{\regfunc}(\pvi{}) = \{0\}\). 
If \(\pvi{}=0\), one easily shows that \(\partial_+\biconj{\regfunc}(0)=0\) which leads to \(\partial\biconj{\regfunc}(0)\cap\kR+=\{0\}\) in view of \Cref{lemma:subdiff-convex-fun-1D}.
We finally deduce that \(\partial\biconj{\regfunc}(0)=\{0\}\) using~\eqref{eq:proof prop:subdiff-biconjugate of g:symmetry}.
This corresponds to the two first cases in~\eqref{eq:operators-biconj-subdiff} since we have \(\pertslope=0\) and \(\pertlimit=+\infty\) from \Cref{lemma:case-id-zero}.

\paragraph{\(\pertfunc\) is not identically zero on \(\kR+\).} 
We consider \(\pvi{}\in\dom{\biconj{\regfunc}}\) and distinguish between different cases. Before proceeding, let us mention that \(\pertlimit>0\) from \Cref{lem:pertlimite > 0} under \ref{assumption:zero-minimized}-\ref{assumption:closed}-\ref{assumption:convex}.
\begin{itemize}
    \item Case \(\pvi{}=0\):
    Since \(\pertlimit>0\), let us consider a sequence \(\{\varepsilon^{(k)}\}_{k\in\kN}\) converging to \(0\) and such that \(0<\varepsilon^{(k)}<\pertlimit\) for all \(k\in\kN\).
    We then have from \Cref{definition:left-right-derivatives}:  
    \begin{equation}
        \partial_+\biconj{\regfunc}(0)
        =
        \lim_{k\to+\infty}
        \frac{
            \biconj{\regfunc}(\varepsilon^{(k)}) -  \biconj{\regfunc}(0)
        }{
            \varepsilon^{(k)}
        }
        =
        \lim_{k\to+\infty} \pertslope
        =
        \pertslope
        ,
    \end{equation}
    where the second equality follows from \Cref{prop:regfunc-biconj} and the definition of \(\varepsilon^{(k)}\). 
    Using the fact that \(0\in\partial\biconj{\regfunc}(0)\) from \eqref{eq:cor:0 in subdiff} in \Cref{cor:subdiff properties}, we obtain that \(\partial\biconj{\regfunc}(0)\cap\kR+=[0,\pertslope]\).
    This leads to \(\partial\biconj{\regfunc}(0)=[-\pertslope,\pertslope]\) in view of~\eqref{eq:proof prop:subdiff-biconjugate of g:symmetry}.
    \item Case \(0<\pvi{}<\pertlimit\): 
    We have from the second case in \Cref{prop:regfunc-biconj} that \(\biconj{\regfunc}\) is linear with slope \(\pertslope\) on \(]0, \pertlimit[\). 
    Hence \(\biconj{\regfunc}\) is differentiable at \(\pvi{}\) and therefore \(\partial\biconj{\regfunc}(\pvi{}) = \{\nabla\biconj{\regfunc}(\pvi{})\} = \{\pertslope\}\) by~\cite[\theoremInRef{3.33}]{beck2017first}.
    \item Case \(\pertlimit<+\infty\) and \(\pvi{}=\pertlimit\):
    Since \(\pvi{}\in\dom{\biconj{\regfunc}}\), we have that \(\pertlimit\in\dom{\biconj{\regfunc}}\).
    Now \(\biconj{\regfunc}\) is proper and convex by \Cref{cor:proper convex closed}, so we have from \Cref{lemma:subdiff-convex-fun-1D} that \(\subdiff\biconj{\regfunc}(\pertlimit)=\kintervcc{\partial_-\biconj{\regfunc}(\pertlimit)}{\partial_+\biconj{\regfunc}(\pertlimit)}\cap\kR\). We next show that \(\partial_-\biconj{\regfunc}(\pertlimit)=\pertslope\) and \(\partial_+\biconj{\regfunc}(\pertlimit)=\pertrightlope\), as stated in \Cref{prop:subdiff-gbiconj}. 
    \begin{itemize}[label=$\bullet$]
        \item On the one hand, let us consider a sequence \(\{\varepsilon^{(k)}\}_{k\in\kN}\) converging to \(0\) and such that \(-\pertlimit<\varepsilon^{(k)}<0\) for all \(k\in\kN\), which exists since \(\pertlimit>0\). 
        Using \Cref{definition:left-right-derivatives}, we have   
        \begin{equation}
            \partial_-\biconj{\regfunc}(\pertlimit)
            =
            \lim_{k\to+\infty}
            =
            \frac{
                \pertslope(\pertlimit + \varepsilon^{(k)}) -  \pertslope\pertlimit
            }{
                \varepsilon^{(k)}
            }
            =
            \pertslope,        
        \end{equation}
        where the second equality follows from \(\pertlimit + \varepsilon^{(k)}\in \kintervoo{0}{\pertlimit}\) and \Cref{prop:regfunc-biconj}. 
        \item On the other hand, consider a sequence \(\{\varepsilon^{(k)}\}_{k\in\kN}\) converging to \(0\) and such that \(\varepsilon^{(k)}>0\) for all \(k\in\kN\). We have 
        \begin{equation}\label{eq:def-g+-prop4.5}
            \partial_+\biconj{\regfunc}(\pertlimit)
            =
            \lim_{k\to+\infty}
            \frac{
                \pertfunc(\pertlimit + \varepsilon^{(k)}) -  \pertfunc(\pertlimit)
            }{
                \varepsilon^{(k)}
            }
            =
            \partial_+\pertfunc(\pertlimit)
            ,
        \end{equation}
        where the first equality is a consequence of \(\pertlimit + \varepsilon^{(k)}>\pertlimit\) and \Cref{prop:regfunc-biconj}. We note that the limit \(\partial_+\pertfunc(\pertlimit)\in\extendedRealLine\) is always well-defined from \Cref{lemma:properties-left-right-derivatives V2} since \(\pertlimit\in\dom{\biconj{\regfunc}}=\dom{\pertfunc}\). 
        One concludes by noting that \(\partial_+\pertfunc(\pertlimit) = \sup \partial \pertfunc(\pertlimit)\) and therefore exactly corresponds to the definition of \(\pertrightlope\) in~\eqref{eq:pertrightlope} {when \(\pertlimit<+\infty\)}.
    \end{itemize}
    \item Case \(\pertlimit < +\infty\) and \(\pvi{}>\pertlimit\): This case can be treated along the same lines as the previous case. More specifically, considering a sequence \(\{\varepsilon^{(k)}\}_{k\in\kN}\) which converges to \(0\) and such that \(\pertlimit-\pvi{} < \varepsilon^{(k)}<0\) (resp. \(\varepsilon^{(k)}>0\)) for all \(k\in\kN\), we easily obtain from \Cref{definition:left-right-derivatives} and \Cref{prop:regfunc-biconj} that \(\partial_-\biconj{\regfunc}(\pvi{})=\partial_-\pertfunc(\pvi{})\) (resp. \(\partial_+\biconj{\regfunc}(\pvi{})=\partial_+\pertfunc(\pvi{})\)). This leads to \(\subdiff\biconj{\regfunc}(\pvi{})=\subdiff\pertfunc(\pvi{})\) in view of \Cref{lemma:subdiff-convex-fun-1D}.
\end{itemize}

\subsection{Proof of \texorpdfstring{\Cref{prop:operator-prox-gbiconj}}{Proposition~\ref{prop:operator-prox-gbiconj}}}
\label{proof:operator-prox-g**}

We first recall that \(\conj{\regfunc}\) is proper, closed and convex under~\ref{assumption:zero-minimized} from \Cref{cor:proper convex closed}.
Therefore, the convex conjugate of \(\biconj{\regfunc}\) is equal to \(\conj{\regfunc}\)~\cite[\theoremInRef{4.8}]{beck2017first}.
Applying the extended Moreau decomposition in \cite[\theoremInRef{6.45}]{beck2017first}, we then have 
\begin{equation} \label{eq:operators-biconj-prox:proof-1}
    \forall \pvi{}\in\kR:\
    \prox_{\cstprox\biconj{\regfunc}}(\pvi{}) = \pvi{} - \cstprox\prox_{\cstprox^{-1}\conj{\regfunc}}({\cstprox}^{-1}\pvi{})
    .
\end{equation}
The result then directly follows from our characterization of \(\prox_{\cstprox\conj{\regfunc}}\) in \Cref{prop:operator-prox-gstar} which holds under~\ref{assumption:zero-minimized}-\ref{assumption:closed}-\ref{assumption:convex} and~\ref{assumption:even}.
In particular, the third case in~\eqref{eq:prop:operator-prox-gbiconj} follows by noticing that one can decompose the proximal operator of \(\biconj{\pertfunc}\) as in~\eqref{eq:operators-biconj-prox:proof-1} since \(\conj{\pertfunc}\) is also proper, closed and convex under~\ref{assumption:zero-minimized} from \Cref{cor:proper convex closed}.


\newcommand{\objectiveFunction}[1]{F^{#1}}
\newcommand{\rootnode}{\nodeSymb_r}
\newcommand{\aMinimizerInProofi}[1]{\widehat{\pvi{}}_{#1}^\nodeSymb}
\newcommand{\aMinimizerInProof}{\widehat{\pv}^\nodeSymb}
\newcommand{\elementInDomPertfunc}{\pvi{\pertfunc}}

\section{Existence of Minimizers} 

This section gathers the proofs pertaining to the existence of minimizers to optimization problems encountered in the implementation of the \gls{bnb}. 
To simplify our exposition, we first provide the following result which will appear in our subsequent derivations. 
\begin{lemma}\label{lemma:properties-F-node}
    Let \(\nodeSymb=(\setzero,\setone)\) with \(\setzero\cup\setone\subseteq \intervint{1}{\pdim}\) and \(\setzero\cap\setone=\emptyset\). 
    If \ref{assumption:f}-\ref{assumption:zero-minimized}-\ref{assumption:closed}-\ref{assumption:coercive} hold, then 
    the function
    \begin{equation} \label{eq:proof:equivalence-problems-l0:ingredient:szd Fnu}
        \objectiveFunction{\nodeSymb}(\pv)
        \defeq 
        \datafunc(\dic\pv) 
        + \sum_{\idxentry=1}^\pdim  \separable{\regfunc}{\idxentry}^{\nodeSymb}(\pvi{\idxentry})
        ,
    \end{equation}
    where \(\separable{\regfunc}{\idxentry}^{\nodeSymb}\) is defined in~\eqref{eq:regfunc-node}, is proper, closed and coercive. 
    Moreover, if~\ref{assumption:convex} holds and \(\setzero\cup\setone=\intervint{1}{\pdim}\), then \(\objectiveFunction{\nodeSymb}\) is convex.
\end{lemma}
\begin{proof}
    The four statements of the result are proved as follows.
    \begin{itemize}
        \item We establish the properness of \(\objectiveFunction{\nodeSymb}\) by proving that its domain is not empty and that the function cannot take on the value \(-\infty\). 
        Using~\eqref{eq:regfunc-node}, we have 
        \begin{align}
            \objectiveFunction{\nodeSymb}(\0)\leq \datafunc(\0)+\reg |\setone|<+\infty
        \end{align}
        where the strict inequality follows from \(\0 \in \dom{\datafunc}\) by assumption~\ref{assumption:f}. Hence \(\dom{\objectiveFunction{\nodeSymb}}\neq \emptyset\). 
        Moreover, since \(\datafunc\) is proper by \ref{assumption:f} and \(\separable{\regfunc}{\idxentry}^{\nodeSymb}\geq 0\) by definition, we also have that
        \(\objectiveFunction{\nodeSymb}>-\infty\). 
        \item Closedness of \(\objectiveFunction{\nodeSymb}\) is proved by using \cite[\theoremInRef{2.7}]{beck2017first} and showing that all the functions appearing in the right-hand side of \eqref{eq:proof:equivalence-problems-l0:ingredient:szd Fnu} are closed. 
        We first have that \(\datafunc\) is closed by assumption~\ref{assumption:f}. 
        Moreover, \(\node{\separable{\regfunc}{\idxentry}}\) is closed in the three possible subcases:         
        \textit{i)} \(\icvx(\pvi{} = 0)\) is closed since \(\{0\}\) is a closed set~\cite[\propositionInRef{2.3}]{beck2017first};
        \textit{ii)} \(\pertfunc\) is closed by~\ref{assumption:closed};
        \textit{iii)} \(\regfunc\) is closed by \Cref{lemma:g-even-nonneg} under~\ref{assumption:closed}.
        \item Since \(\datafunc\) is lower-bounded by~\ref{assumption:f} and \(\dom{\objectiveFunction{\nodeSymb}}\neq\emptyset\) by properness of \(\objectiveFunction{\nodeSymb}\), it is sufficient to show the coercivity of the functions \(\{\separable{\regfunc}{\idxentry}^{\nodeSymb}\}_{\idxentry=1}^\pdim\) to establish the coercivity of \(\objectiveFunction{\nodeSymb}\) according to \cite[\corollaryInRef{11.16}]{Bauschke2017}. 
        Coercivity of \(\icvx(\pvi{} = 0)\) is trivial; coercivity of \(\pertfunc+\reg\) follows from~\ref{assumption:coercive}; coercivity of \(\regfunc\) follows from \(\regfunc(\pvi{})\geq \pertfunc(\pvi{})\) \(\forall \pvi{}\in\kR\). 
        \item Finally, assume that~\ref{assumption:convex} holds and that \(\setzero\cup\setone=\intervint{1}{\pdim}\).
        Then~\eqref{eq:regfunc-node} simplifies to
        \begin{equation}
            \forall \pvi{}\in\kR:\quad        
            \node{\separable{\regfunc}{\idxentry}}(\pvi{}) = 
            \begin{cases}
                \icvx(\pvi{} = 0) &\text{if} \ \idxentry \in \setzero \\
                \pertfunc(\pvi{}) + \reg &\text{if} \ \idxentry \in \setone
            \end{cases}
        \end{equation}
        which is convex since \(\{0\}\) is a convex set and \(\pertfunc\) is convex from~\ref{assumption:convex}.
        Since \(\datafunc\) is convex from~\ref{assumption:f}, we have that \(\objectiveFunction{\nodeSymb}\) is convex.  
    \end{itemize}
\end{proof}

\subsection{Existence of a Minimizer to \texorpdfstring{\eqref{prob:prob}}{(\ref{prob:prob})}} \label{proof:existence-minimizer}

Letting \(\nodeSymb=(\emptyset,\emptyset)\),  we note that the function \(\objectiveFunction{\nodeSymb}\) defined in \eqref{eq:proof:equivalence-problems-l0:ingredient:szd Fnu} corresponds to the objective function in~\eqref{prob:prob}. 
Since \ref{assumption:f}-\ref{assumption:zero-minimized}-\ref{assumption:closed}-\ref{assumption:coercive} hold, we then obtain from \Cref{lemma:properties-F-node} that \(\objectiveFunction{\nodeSymb}\) is proper, closed and coercive.
We finally have that \(\objectiveFunction{\nodeSymb}\) attains its infimum over \(\kR^{\pdim}\) as a consequence of the Weierstrass theorem for coercive functions~\cite[\theoremInRef{2.14}]{beck2017first}.

\subsection{Proof of \texorpdfstring{\Cref{prop:prob-node}}{Proposition~\ref{prop:prob-node}}}
\label{proof:equivalence-problems-l0}

We first establish a preliminary result.
\begin{lemma}\label{lemma:compararison function fnu}    
    Let \(\nodeSymb=(\setzero,\setone)\) with \(\setzero\cup\setone\subseteq \intervint{1}{\pdim}\) and \(\setzero\cap\setone=\emptyset\). 
    Let moreover \(\nodeSymb'=(\setzero',\setone')\) such that \(\setzero{\subseteq}\setzero'\) and \(\setone{\subseteq}\setone'\).
    If \ref{assumption:zero-minimized} holds, then 
    \begin{equation} \label{eq:lemma:compararison function fnu}
        \forall \pv\in \kR^{\pdim}:\, 
        \objectiveFunction{\nodeSymb}(\pv) \leq \objectiveFunction{\nodeSymb'}(\pv)
        .
    \end{equation}
    Moreover, equality holds if at least one of the following conditions is satisfied:
    \begin{enumerate}[label=\textit{\alph*.}]
        \item \(\pv\in \pset^{\nodeSymb'}\),
        \label{item:lemma:compararison function fnu:equality case}
        \item \(\pv\in\closure(\pset^{\nodeSymb'})\) and \(\pvi{\idxentry}\neq0\) \(\forall\idxentry\in\setone'\setminus\setone\). 
        \label{item:lemma:compararison function fnu:equality case closure}\\
    \end{enumerate}
\end{lemma}
\begin{proof} 
    By definition of \(\objectiveFunction{\nodeSymb}\) in \eqref{eq:proof:equivalence-problems-l0:ingredient:szd Fnu}, inequality~\eqref{eq:lemma:compararison function fnu} is verified provided that 
    \begin{equation} \label{eq:proof lemma:compararison function fnu:it remains to show this inequality}
        \forall \idxentry\in\intervint{1}{\pdim},
        \forall \pvi{\idxentry}\in\kR
        :\quad
        \separable{\regfunc}{\idxentry}^{\nodeSymb}(\pvi{\idxentry}) \leq \separable{\regfunc}{\idxentry}^{\nodeSymb'}(\pvi{\idxentry})
        .
    \end{equation}  
    By distinguishing between the cases ``\(\idxentry\in\setzero\)'', ``\(\idxentry\in\setone\)'' and ``\(\idxentry\notin\setzero\cup\setone\)'', one first observes that \(\dom{\separable{\regfunc}{\idxentry}^{\nodeSymb'}} \subseteq \dom{\separable{\regfunc}{\idxentry}^{\nodeSymb}}\).
    Inequality~\eqref{eq:proof lemma:compararison function fnu:it remains to show this inequality} thus trivially holds for all \(\pvi{\idxentry}\notin\dom{\separable{\regfunc}{\idxentry}^{\nodeSymb}}\) as a consequence of the convention ``\(\infty \leq \infty\)''. 
    It remains to establish the inequality for \(\pvi{\idxentry}\in\dom{\separable{\regfunc}{\idxentry}^{\nodeSymb}}\).
    This condition can be shown to be met by verifying that the function
    \begin{equation} \label{eq:proof lemma:compararison function fnu:diff gnu' minus gnu}
        \separable{\regfunc}{\idxentry}^{\nodeSymb'}(\pvi{\idxentry}) - \separable{\regfunc}{\idxentry}^{\nodeSymb}(\pvi{\idxentry})
        =
        \begin{cases}
            \icvx(\pvi{\idxentry} = 0) - \regfunc(\pvi{\idxentry})  &\text{if} \ \idxentry\in\setzero'\setminus\setzero \\
            \reg - \reg\Vert\pvi{\idxentry}\Vert_0 &\text{if} \ \idxentry\in\setone'\setminus\setone \\
            0 & \text{otherwise}
        \end{cases}
    \end{equation}
    is non-negative. 
    If \(\idxentry\in\setzero'\setminus\setzero\), this property holds since \(\separable{\regfunc}{\idxentry}^{\nodeSymb'}(\pvi{\idxentry}) - \separable{\regfunc}{\idxentry}^{\nodeSymb}(\pvi{\idxentry})=+\infty\) as soon as \(\pvi{\idxentry}\neq0\) and \(\regfunc(0)=0\) from item~\ref{item-lemma:zero-minimized} of \Cref{lemma:g-even-nonneg} under~\ref{assumption:zero-minimized}.
    If \(\idxentry\in\setone'\setminus\setone\), the result directly follows from the fact that \(\Vert\pvi{\idxentry}\Vert_0\in\{0,1\}\) for all \(\pvi{\idxentry}\in\kR\).
    The last case is trivial.

    Let us finally focus on some conditions ensuring that equality holds in \eqref{eq:lemma:compararison function fnu}. 
    In view of the inclusions \(\dom{\separable{\regfunc}{\idxentry}^{\nodeSymb'}} \subseteq \dom{\separable{\regfunc}{\idxentry}^{\nodeSymb}}\) established at the beginning of the proof, we note that the equality is trivially verified for \(\pvi{\idxentry}\notin\dom{\separable{\regfunc}{\idxentry}^{\nodeSymb}}\).
    We thus assume in the rest of the proof that \(\pvi{\idxentry}\in\dom{\separable{\regfunc}{\idxentry}^{\nodeSymb}}\). 
    Under this assumption, we have from our previous discussion that establishing equality in~\eqref{eq:lemma:compararison function fnu} is equivalent to showing that the functions in~\eqref{eq:proof lemma:compararison function fnu:diff gnu' minus gnu} are equal to zero \(\forall \idxentry\in\intervint{1}{\pdim}\).
    Let us concentrate on the conditions stated in the lemma:
    \begin{itemize}
        \item Item~\ref{item:lemma:compararison function fnu:equality case} Assume \(\pv\in\pset^{\nodeSymb'}\). We observe that \(\pvi{\idxentry}=0\) \(\forall\idxentry\in\setzero'\setminus\setzero\) and \(\pvi{\idxentry}\neq 0\) \(\forall\idxentry\in\setone'\setminus\setone\). 
        It follows that all the functions in \eqref{eq:proof lemma:compararison function fnu:diff gnu' minus gnu} are equal to zero by using the fact \(\regfunc(0)=0\) and \(\|\pvi{}\|_0=1\) \(\forall\pvi{}\neq 0\). 
        \item Item~\ref{item:lemma:compararison function fnu:equality case closure} 
        Assume \(\pv\in\closure(\pset^{\nodeSymb'})\) and \(\pvi{\idxentry}\neq0\) \(\forall\idxentry\in\setone'\setminus\setone\). We observe that
        \begin{equation}
            \closure(\pset^{\nodeSymb'})
            =
            \kset{
                \pv\in\kR^\pdim
            }{
                \pvi{\idxentry} = 0 \ \forall\idxentry\in\setzero'
            }
            .
        \end{equation}        
        Therefore, the first case in~\eqref{eq:proof lemma:compararison function fnu:diff gnu' minus gnu} 
        is identically equal to zero and the expression simplifies to 
        \begin{equation}
            \separable{\regfunc}{\idxentry}^{\nodeSymb'}(\pvi{\idxentry}) - \separable{\regfunc}{\idxentry}^{\nodeSymb}(\pvi{\idxentry})
            =
            \begin{cases}
                \reg - \reg\Vert\pvi{\idxentry}\Vert_0 &\text{if} \ \idxentry\in\setone'\setminus\setone \\
                0 &\text{otherwise}. 
            \end{cases}
        \end{equation}
        These functions are clearly equal to zero \(\forall \idxentry\in\intervint{1}{\pdim}\) since we assumed \(\pvi{\idxentry}\neq0\) \(\forall\idxentry\in\setone'\setminus\setone\).
    \end{itemize}
\end{proof}

We now turn to the proof of \Cref{prop:prob-node}.
We will proceed by breaking problems~\eqref{eq:node_problem} -- the minimization problem involved in the definition of \(\node{\pobj}\) -- and~\eqref{prob:prob-node} into a series of subproblems which will eventually appear to be equal. 
To that end, let us first notice that for any node\footnote{Remind that a node of the \gls{bnb} decision tree must verify \(\setzero\cup\setone\subseteq \intervint{1}{\pdim}\), \(\setzero\cap\setone=\emptyset\).} \(\nodeSymb=(\setzero,\setone)\) of the \gls{bnb} decision tree, the set \(\pset^{\nodeSymb}\) can be partitioned as follows: 
\begin{equation} \label{eq:proof:equivalence-problems-l0:ingredient:covering pset}
    \pset^{\nodeSymb}
    =
    \bigcup_{\nodeSymb'\in\node{\mathcal{N}}} \pset^{\nodeSymb'}
\end{equation}
where 
\begin{equation}
    \node{\mathcal{N}} \defeq \kset{\nodeSymb'=(\setzero',\setone')}{\setzero\subseteq\setzero',\setone\subseteq\setone', \setzero'\cap\setone'=\emptyset, \setzero'\cup\setone'=\intervint{1}{\pdim}}
    .
\end{equation}
Based on this observation, we next re-express problems~\eqref{eq:node_problem}-\eqref{prob:prob-node} as detailed below, and then show that these reformulations are equivalent.

\paragraph{Equivalent formulation of~\eqref{eq:node_problem}.}
We have
\begin{align} 
    \inf_{\pv \in \pset^{\nodeSymb}}     
    \ \datafunc(\dic\pv) + \sum_{\idxentry=1}^{\pdim}\regfunc(\pvi{\idxentry})
    \,=\,
    & 
    \inf_{\pv \in \pset^{\nodeSymb}}
    \objectiveFunction{\nodeSymb}(\pv)
    \\
    \,=\,
    & 
    \min_{\nodeSymb'\in\node{\mathcal{N}}}
    \kparen{
        \inf_{\pv \in \pset^{\nodeSymb'}} 
        \objectiveFunction{\nodeSymb}(\pv)
    }
    \\
    \,=\,
    & 
    \min_{\nodeSymb'\in\node{\mathcal{N}}}
    \kparen{
        \inf_{\pv \in \pset^{\nodeSymb'}} 
        \objectiveFunction{\nodeSymb'}(\pv)
    }
    \label{eq:proof:equivalence-problems-l0:ingredient:equivalent formulation pb1}
\end{align}
where the first equality is obtained by considering \Cref{lemma:compararison function fnu}.\ref{item:lemma:compararison function fnu:equality case} with \(\nodeSymb=(\emptyset,\emptyset)\) and \(\nodeSymb'=\nodeSymb\); 
the second is due to~\eqref{eq:proof:equivalence-problems-l0:ingredient:covering pset}; 
the third follows from \Cref{lemma:compararison function fnu}.\ref{item:lemma:compararison function fnu:equality case} again.
We note that the minimum in the last two equalities is achieved since \(\node{\mathcal{N}}\) contains a finite number of elements. 

\paragraph{Equivalent formulation of~\eqref{prob:prob-node}.}
First notice that the objective function in~\eqref{prob:prob-node} corresponds to  function \(\objectiveFunction{\nodeSymb}\) in~\eqref{eq:proof:equivalence-problems-l0:ingredient:szd Fnu}.
Second, we observe that any \(\nodeSymb'=(\setzero',\setone')\in\node{\mathcal{N}}\) verifies by definition the hypotheses of \Cref{lemma:compararison function fnu}. Using the latter result, we then obtain:
\begin{equation} \label{eq:proof:equivalence-problems-l0:ingredient:towards equivalent formulation pb2}
    \inf_{\pv \in \kR^{\pdim}}
        \objectiveFunction{\nodeSymb}(\pv)        
    \leq
    \min_{\nodeSymb'\in\node{\mathcal{N}}}
    \kparen{
        \inf_{\pv \in \kR^{\pdim}} \objectiveFunction{\nodeSymb'}(\pv)
    }
    .
\end{equation}
We next show that equality holds in~\eqref{eq:proof:equivalence-problems-l0:ingredient:towards equivalent formulation pb2}. 
To that end, we first notice that all the functions involved in~\eqref{eq:proof:equivalence-problems-l0:ingredient:towards equivalent formulation pb2} 
are proper, closed and coercive under \ref{assumption:f}-\ref{assumption:zero-minimized}-\ref{assumption:closed}-\ref{assumption:coercive} from \Cref{lemma:properties-F-node}.  
Therefore, all infima in~\eqref{eq:proof:equivalence-problems-l0:ingredient:towards equivalent formulation pb2} are finite and attained~\cite[\theoremInRef{2.14}]{beck2017first}.\footnote{In view of our initial remark on the connection between the objective function in~\eqref{prob:prob-node} and \(\objectiveFunction{\nodeSymb}\), this also shows the existence of a minimizer to \eqref{prob:prob-node} and therefore justifies the use of ``\(\min{}\)'' rather than ``\(\inf{}\)'' in our formulation.}
Let \(\pv^\nodeSymb \in \kR^{\pdim}\) be a minimizer to the left-hand side of~\eqref{eq:proof:equivalence-problems-l0:ingredient:towards equivalent formulation pb2}.
To prove the equality in~\eqref{eq:proof:equivalence-problems-l0:ingredient:towards equivalent formulation pb2}, it is sufficient to identify some \(\nodeSymb''\in \mathcal{N}^{\nodeSymb}\) such that \(\objectiveFunction{\nodeSymb''}(\pv^\nodeSymb)=\objectiveFunction{\nodeSymb}(\pv^\nodeSymb)\), since this leads to 
\begin{equation}
    \min_{\nodeSymb'\in\node{\mathcal{N}}}
    \kparen{
        \min_{\pv \in \kR^{\pdim}} \objectiveFunction{\nodeSymb'}(\pv)
    }
    \leq
    \objectiveFunction{\nodeSymb''}(\pv^\nodeSymb)
    =
    \objectiveFunction{\nodeSymb}(\pv^\nodeSymb)
    =
    \min_{\pv \in \kR^{\pdim}} \objectiveFunction{\nodeSymb}(\pv)
    .
\end{equation}
We thus restrict our focus to the construction of a suitable \(\nodeSymb''\) hereafter. 
Specifically, we let        
\begin{subequations}
    \begin{align}
        \setzero'' \,\defeq\,& \kset{\idxentry\in \intervint{1}{\pdim}}{\pvi{\idxentry}^\nodeSymb=0} \setminus \setone \\
        \setone''  \,\defeq\,& \kset{\idxentry\in \intervint{1}{\pdim}}{\pvi{\idxentry}^\nodeSymb\neq0} \cup \setone
    \end{align}
\end{subequations}
and show that \(\nodeSymb''\in \mathcal{N}^{\nodeSymb}\) with \(\objectiveFunction{\nodeSymb''}(\pv^\nodeSymb)=\objectiveFunction{\nodeSymb}(\pv^\nodeSymb)\). 
As a preliminary remark, let us notice that \(\pvi{\idxentry}^\nodeSymb\in\dom{\separable{\regfunc}{\idxentry}^{\nodeSymb}}\)  \(\forall \idxentry\in\intervint{1}{\pdim}\) since the minimum of \(\objectiveFunction{\nodeSymb}\) is finite. 
Hence,
\begin{equation} \label{eq:proof:equivalence-problems-l0:ingredient:towards equivalent formulation pb2:property solution}
    \forall \idxentry\in\setzero:\quad
    \pvi{\idxentry}^\nodeSymb
    =
    0
\end{equation}        
as \(\dom{\separable{\regfunc}{\idxentry}^{\nodeSymb}}=\{0\}\) for all \(\idxentry\in\setzero\) by definition.
\begin{itemize}
    \item First, using~\eqref{eq:proof:equivalence-problems-l0:ingredient:towards equivalent formulation pb2:property solution} and the fact that \(\setzero\cap\setone=\emptyset\), we observe that \(\setzero\subseteq\setzero''\).
    Moreover, the definition of \(\setone''\)  immediately leads to \(\setone\subseteq\setone''\).
    Hence, we have \(\nodeSymb''\in\mathcal{N}^{\nodeSymb}\). 
    \item Second, our observation in~\eqref{eq:proof:equivalence-problems-l0:ingredient:towards equivalent formulation pb2:property solution} ensures that \(\pv^\nodeSymb\in\closure(\pset^{\nodeSymb''})\) and the definition of \(\setone''\) implies that \(\pvi{\idxentry}^{\nodeSymb}\neq0\) for all \(\idxentry\in\setone''\setminus\setone\). We then obtain \(\objectiveFunction{\nodeSymb''}(\pv^\nodeSymb)=\objectiveFunction{\nodeSymb}(\pv^\nodeSymb)\) using \Cref{lemma:compararison function fnu}.\ref{item:lemma:compararison function fnu:equality case closure} with \(\nodeSymb=\nodeSymb''\).
\end{itemize}

\paragraph{Equivalence between~\eqref{eq:node_problem} and~\eqref{prob:prob-node}.}
In view of our reformulation of problems~\eqref{eq:node_problem} and~\eqref{prob:prob-node} above, it is sufficient to show that 
\begin{equation} \label{eq:proof:equivalence-problems-l0:ingredient:targeted equality}
    \forall \nodeSymb'\in\node{\mathcal{N}}:\quad 
    \min_{\pv \in \kR^{\pdim}} \objectiveFunction{\nodeSymb'}(\pv)
    =
    \inf_{\pv \in \pset^{\nodeSymb'}} \objectiveFunction{\nodeSymb'}(\pv)
\end{equation}
to prove the statement of \Cref{prop:prob-node}.
We first remark that when \(\nodeSymb'=(\setzero',\setone')\in\node{\mathcal{N}}\), then
\begin{equation} \label{eq:proof:equivalence-problems-l0:ingredient:first inequality}
    \min_{\pv \in \kR^{\pdim}} \objectiveFunction{\nodeSymb'}(\pv)
    \leq
    \inf_{\pv \in \pset^{\nodeSymb'}} \objectiveFunction{\nodeSymb'}(\pv)
\end{equation}
since \(\pset^{\nodeSymb'}\subseteq\kR^\pdim\). 
It thus remains to show the reverse inequality.
To this end, we observe that
\begin{equation} \label{eq:proof:equivalence-problems-l0:ingredient:inclusion dom Fnuprime closure Xnuprime}
    \dom{\objectiveFunction{\nodeSymb'}}
    \subseteq
    \bigtimes_{\idxentry=1}^\pdim 
    \dom{\separable{\regfunc}{\idxentry}^{\nodeSymb'}}
    \subseteq
    \kset{\pv\in\kR^\pdim}{\pvi{\idxentry}=0 \ \forall\idxentry\in\setzero'}
    =
    \closure(\pset^{\nodeSymb'})
\end{equation}
where the first inclusion holds since \(\objectiveFunction{\nodeSymb'}\) is defined as a sum of functions involving \(\separable{\regfunc}{\idxentry}^{\nodeSymb'}\) and the second inclusion is a consequence of the definition of \(\separable{\regfunc}{\idxentry}^{\nodeSymb'}\) given in~\eqref{eq:regfunc-node}. 
Hence, we deduce that
\begin{equation} \label{eq:proof:equivalence-problems-l0:ingredient:targeted equality reformulation}
    \min_{\pv \in \kR^{\pdim}} \objectiveFunction{\nodeSymb'}(\pv)
    =
    \min_{\pv \in \dom{\objectiveFunction{\nodeSymb'}}} \objectiveFunction{\nodeSymb'}(\pv)
    =
    \min_{\pv \in\closure(\pset^{\nodeSymb'})} \objectiveFunction{\nodeSymb'}(\pv)
    .
\end{equation}
Now, let \(\aMinimizerInProof \in \kR^{\pdim}\) be a minimizer of~\eqref{eq:proof:equivalence-problems-l0:ingredient:targeted equality reformulation}.
In view of~\eqref{eq:proof:equivalence-problems-l0:ingredient:inclusion dom Fnuprime closure Xnuprime}, we have \(\aMinimizerInProof\in\dom{\objectiveFunction{\nodeSymb'}}\cap\closure(\pset^{\nodeSymb'})\).
We next distinguish between two cases.
\begin{itemize}
    \item Case \(\aMinimizerInProof\in\pset^{\nodeSymb'}\): We immediately obtain
    \begin{equation} \label{eq:proof:equivalence-problems-l0:ingredient:second inequality}
        \inf_{
            \pv\in\pset^{\nodeSymb'}
        } \
        \objectiveFunction{\nodeSymb'}(\pv)
        \leq
        \objectiveFunction{\nodeSymb'}(\aMinimizerInProof)
        =
        \min_{\pv \in \kR^{\pdim}} \objectiveFunction{\nodeSymb'}(\pv)
    \end{equation}
    by definition of the infimum, which yields the desired inequality. 
    \item Case \(\aMinimizerInProof\notin\pset^{\nodeSymb'}\): We have \(\aMinimizerInProof\in\closure(\pset^{\nodeSymb'})\setminus\pset^{\nodeSymb'}\). 
    In view of the definition of \(\closure(\pset^{\nodeSymb'})\) and \(\pset^{\nodeSymb'}\), we deduce that there exists at least one \(\idxentry\in\setone'\) such that \(\aMinimizerInProofi{\idxentry}=0\).
    Upon this observation, let define \(\overline{\pv}\in\pset^{\nodeSymb'}\) entry-wise as
    \begin{equation}
        \forall \idxentry\in\intervint{1}{\pdim}:\quad
        \overline{\pvi{}}_\idxentry
        \defeq
        \begin{cases}
            \aMinimizerInProofi{\idxentry} & \text{ if } \idxentry\in\setzero' 
            \\
            \aMinimizerInProofi{\idxentry} & \text{ if } \idxentry\in\setone' \text{ and } \aMinimizerInProofi{\idxentry}\neq 0 \\ 
            \elementInDomPertfunc &\text{ otherwise}
        \end{cases}
    \end{equation}
    where \(\elementInDomPertfunc\) denotes any nonzero element of \(\dom\pertfunc\).\footnote{We note that such an element exists under~\ref{assumption:zero-minimized}.} 
    Letting \(\pv_{\alpha}\defeq(1-\alpha)\aMinimizerInProof + \alpha\overline{\pv}\) for all \(\alpha\in]0,1[\), we have \(\pv_{\alpha}\in \pset^{\nodeSymb'}\). As a byproduct, we deduce that
    \begin{equation} \label{eq:proof:equivalence-problems-l0:ingredient:partial second inequality}
        \inf_{
            \pv\in\pset^{\nodeSymb'}
        } \
        \objectiveFunction{\nodeSymb'}(\pv)
        \leq
        \lim_{
            \substack{
                \alpha\downarrow 0 \\
                \alpha < 1
            }
        } \objectiveFunction{\nodeSymb'}(\pv_{\alpha})
        .
    \end{equation}
    Since \(\objectiveFunction{\nodeSymb'}\) is proper, closed and convex under~\ref{assumption:f}-\ref{assumption:zero-minimized}-\ref{assumption:closed}-\ref{assumption:convex}-\ref{assumption:coercive} from \Cref{lemma:properties-F-node}, we can apply~\cite[\propositionInRef{9.14}]{Bauschke2017} with \((f, x, y) =(\objectiveFunction{\nodeSymb'}, \aMinimizerInProof,\overline{\pv})\) to deduce that the right-hand side of~\eqref{eq:proof:equivalence-problems-l0:ingredient:partial second inequality} is equal to \(\objectiveFunction{\nodeSymb'}(\aMinimizerInProof)\).
    Hence, inequality~\eqref{eq:proof:equivalence-problems-l0:ingredient:second inequality} holds when \(\aMinimizerInProof\notin\pset^{\nodeSymb'}\). 
\end{itemize}

\subsection{Existence of a minimizer to \texorpdfstring{\eqref{prob:relax-node}}{(\ref{prob:relax-node})}}
\label{proof:existence-minimizer-relax-node}

Let \(\nodeSymb=(\setzero,\setone)\) be a node of the \gls{bnb} tree. 
The existence of a minimizer to \eqref{prob:relax-node} under \ref{assumption:f}-\ref{assumption:zero-minimized}-\ref{assumption:closed}-\ref{assumption:convex}-\ref{assumption:coercive} can be shown by applying the Weierstrass' theorem for coercive functions, see \cite[\theoremInRef{2.14}]{beck2017first}. 
To do so, it is sufficient to show that the objective function in \eqref{prob:relax-node} is proper, closed, coercive, and has a non-empty domain.
The proof of this result can be done exactly along the same lines as in the proof of \Cref{lemma:properties-F-node}, 
the only difference here being that the definitions of the functions \(\node{\separable{\regfunc}{\idxentry}}\) and \(\node{\separable{\relaxregfunc}{\idxentry}}\) differ when \(\idxentry\notin\setzero\cup\setone\). 
In the proof at hand, we will thus concentrate specifically on establishing that \(\node{\separable{\relaxregfunc}{\idxentry}}\) is proper, closed, and coercive for \(\idxentry \notin \setzero \cup \setone\). This aspect aside, the remainder of the arguments follow directly and identically from those utilized in the proof of \Cref{lemma:properties-F-node}.

If \(\idxentry\notin\setzero\cup\setone\), then \(\node{\separable{\relaxregfunc}{\idxentry}}=\biconj{\regfunc}\) which is proper and closed under~\ref{assumption:zero-minimized} from \Cref{cor:proper convex closed}.
We also note that \(\regfunc\geq\pertfunc\) by nonnegativity of the \(\ell_0\)-norm so that \(\biconj{\regfunc}\geq \biconj{\pertfunc}\) from~\cite[\propositionInRef{13.16.(ii)}]{Bauschke2017}.
Since \(\biconj{\pertfunc}=\pertfunc\) under~\ref{assumption:zero-minimized}-\ref{assumption:closed}-\ref{assumption:convex} and \(\pertfunc\) is coercive by~\ref{assumption:coercive}, we obtain that \(\biconj{\pertfunc}\) and also coercive which concludes the proof.

\section{Supplementary Material for Numerical Experiments}
\label{sec:supp-numerics}

This section provides additional details on the numerical experiments presented in \Cref{sec:numerics}.

\subsection{Hardware and Software Specifications}
\label{sec:supp-numerics:hardware}

The code used for our experiments is open-sourced at
\begin{center}
    \vspace*{10pt}
    \href{https://github.com/TheoGuyard/l0exp}{\texttt{https://github.com/TheoGuyard/l0exp}}
    \vspace*{10pt}
\end{center}
and all the datasets used are publicly available.
Computations were carried out using the Grid'5000 testbed, supported by a scientific interest group hosted by Inria and including CNRS, RENATER and several universities as well as other organizations.
Experiments were run on a Debian 10 operating system, featuring one Intel Xeon E5-2660 v3 CPU clocked at 2.60 GHz with 16 GB of RAM.
We used \textsc{Python} v3.9.2, \textsc{Mosek} v9.3, \textsc{Cplex} v20.1 and \textsc{L0bnb} v1.0.

\subsection{Mixed-Integer Programming Formulations}
\label{sec:supp-numerics:mip}

Problem \eqref{prob:prob} is tackled by \cplex{} and \mosek{} through a \gls{mip} reformulation instantiated using \textsc{Pyomo} \cite{bynum2021pyomo}.
We use a baseline model expressed as
\begin{equation}
    \label{eq:supp-numerics:mip:baseline}
    \left\{
        \begin{array}{rl}
            \min & \datafunc_{\text{val}} + \regfunc_{\text{val}} \\
            \text{s.t.} & \wv = \dic\pv \\
            & \pv \in \kR^{\pdim}, \ \bv \in \{0,1\}^{\pdim}, \ \wv \in \kR^{\ddim} \\
            & \datafunc_{\text{val}} \in \kR, \ \regfunc_{\text{val}} \in \kR
        \end{array}
    \right.
\end{equation}
and add new linear, quadratic or conic constraints to bind value of the variables $\datafunc_{\text{val}}$ and $\pertfunc_{\text{val}}$ depending on the expression of the functions $\datafunc$ and $\pertfunc$ considered, as specified in \Cref{table:mip-datafunc,table:mip-pertfunc}.
Our implementation is also part of the \elops{} toolbox.\footnote{See \href{https://github.com/TheoGuyard/El0ps/blob/main/src/el0ps/solver/mip.py}{\texttt{https://github.com/TheoGuyard/El0ps/blob/main/src/el0ps/solver/mip.py}}.}

\begin{table}[!t]
    \centering
    \begin{tabular*}{0.95\linewidth}{l|l}
        \toprule
        Function $\datafunc$ & Constraints added to baseline model \eqref{eq:supp-numerics:mip:baseline} \\
        \midrule
        $\datafunc(\wv) = \tfrac{1}{2}\norm{\obs - \wv}{2}^2$ &
        $
        \begin{array}{rl}
            (\datafunc_{\text{val}},1,\mathbf{r}) &\in \mathcal{Q}_{\text{rot2}} \\
            \wv - \obs &= \mathbf{r} \\
            \mathbf{r} &\in \kR^{\ddim}
        \end{array}
        $ \\ \midrule
        $\datafunc(\wv) = \transp{\1}\log(\1 + \exp(-\obs \odot \wv))$ &
        $
        \begin{array}{rl}
            -\transp{\1}\mathbf{s}' &\leq \datafunc_{\text{val}} \\
            \obs \odot \wv + \mathbf{s}' &= \mathbf{s} \\
            \mathbf{r} + \mathbf{r}' &\leq \1 \\
            \forall \idxentry, \ (r_{\idxentry},1,s_{\idxentry}) &\in \mathcal{K}_{\text{exp}} \\
            \forall \idxentry, \ (r_{\idxentry}',1,s_{\idxentry}') &\in \mathcal{K}_{\text{exp}} \\
            (\mathbf{r},\mathbf{r}',\mathbf{s},\mathbf{s}') &\in \kR^{\ddim}
        \end{array}
        $
        \\ \midrule
        $\datafunc(\wv) = \norm{\pospart{\1 - \obs \odot \wv}}{2}^2$ &
        $
        \begin{array}{rl}
            (\datafunc_{\text{val}},\tfrac{1}{2},\mathbf{r}) &\in \mathcal{Q}_{\text{rot2}} \\
            \1 - \obs \odot \wv &\leq \mathbf{r} \\
            \mathbf{r} &\in \kR+^{\ddim}
        \end{array}
        $
        \\ \bottomrule
    \end{tabular*}
    \caption{MIP model specifications for the function $\datafunc$. Vectorial inequalities are taken component-wise. $\mathcal{Q}_{\text{rot}}$ and $\mathcal{K}_{\text{exp}}$ denote Rotated-Quadratic Cone constraints \cite[\sectionInRef{3.1.2}]{aps2020mosek} and Exponential Cone constraints \cite[\sectionInRef{5.1}]{aps2020mosek}, respectively.}
    \label{table:mip-datafunc}
\end{table}

\begin{table}[!t]
    \centering
    \begin{tabular*}{0.95\linewidth}{l|l}
        \toprule
        Function $\pertfunc$ & Constraints added to baseline model \eqref{eq:supp-numerics:mip:baseline} \\ \midrule
        $\pertfunc(\pvi{}) = \regone\abs{\pvi{}}$ & 
        $
        \begin{array}{rl}
            \transp{\1}(\reg\bv + \regone\mathbf{t}) &\leq \regfunc_{\text{val}} \\
            \mathbf{t} \odot \bv &\geq \pv \\
            -\mathbf{t} \odot \bv &\leq \pv \\
            \mathbf{t} &\in \kR+^{\pdim} \\
        \end{array}
        $
        \\ \midrule
        $\pertfunc(\pvi{}) = \tfrac{\regone}{2}\pvi{}^2$ & 
        $
        \begin{array}{rl}
            \transp{\1}(\reg\bv + \regone\mathbf{t}) &\leq \regfunc_{\text{val}} \\
            \forall \idxentry, \ (t_{\idxentry},\bvi{\idxentry},\pvi{\idxentry}) &\in \mathcal{Q}_{\text{rot2}} \\
            \mathbf{t} &\in \kR+^{\pdim} \\
        \end{array}
        $
        \\ \midrule
        $\pertfunc(\pvi{}) = \tfrac{\regone}{2}\pvi{}^2 + \icvx(\abs{\pvi{}} \leq \bigM)$ & 
        $
        \begin{array}{rl}
            \transp{\1}(\reg\bv + \regone\mathbf{t}) &\leq \regfunc_{\text{val}} \\
            \forall \idxentry, \ (t_{\idxentry},\bvi{\idxentry},\pvi{\idxentry}) &\in \mathcal{Q}_{\text{rot2}} \\
            -\bigM\bv &\leq \pv \\
            \bigM\bv &\geq \pv \\
            \mathbf{t} &\in \kR+^{\pdim} \\
        \end{array}
        $
        \\ \midrule
        $\pertfunc(\pvi{}) = \regone\abs{\pvi{}} + \icvx(\abs{\pvi{}} \leq \bigM)$ & 
        $
        \begin{array}{rl}
            \transp{\1}(\reg\bv + \regone\mathbf{t}) &\leq \regfunc_{\text{val}} \\
            \pv &\leq \mathbf{t} \\
            -\pv &\leq \mathbf{t} \\
            \bigM\bv &\geq \pv \\
            -\bigM\bv &\leq \pv \\
            \mathbf{t} &\in \kR+^{\pdim} \\
        \end{array}
        $
        \\ \midrule
        $\pertfunc(\pvi{}) = \regone\abs{\pvi{}} + \icvx(\pvi{} \geq 0)$ & 
        $
        \begin{array}{rl}
            \transp{\1}(\reg\bv + \regone\mathbf{t}) &\leq \regfunc_{\text{val}} \\
            \pv &\leq \mathbf{t} \odot \bv \\
            \pv &\geq \0 \\
            \mathbf{t} &\in \kR+^{\pdim} \\
        \end{array}
        $
        \\ \midrule
        $\pertfunc(\pvi{}) = \tfrac{\regone}{2}\pvi{}^2 + \icvx(\pvi{} \geq 0)$ & 
        $
        \begin{array}{rl}
            \transp{\1}(\reg\bv + \regone\mathbf{t}) &\leq \regfunc_{\text{val}} \\
            \forall \idxentry, \ (t_{\idxentry},\bvi{\idxentry},\pvi{\idxentry}) &\in \mathcal{Q}_{\text{rot2}} \\
            \pv &\geq \0 \\
            \mathbf{t} &\in \kR+^{\pdim} \\
        \end{array}
        $
        \\ \midrule
        $\pertfunc(\pvi{}) = \regone\abs{\pvi{}} + \tfrac{\regone'}{2}\pvi{}^2$ &
        $
        \begin{array}{rl}
            \transp{\1}(\reg \bv + \regone\mathbf{t}^1 + \regone'\mathbf{t}') &\leq \regfunc_{\text{val}} \\
            \forall \idxentry, \ (t_{\idxentry}',\bvi{\idxentry},\pvi{\idxentry}) &\in \mathcal{Q}_{\text{rot2}} \\
            \pv &\leq \mathbf{t} \\
            -\pv &\leq \mathbf{t} \\
            \mathbf{t} &\in \kR+^{\pdim} \\
        \end{array}
        $
        \\ \bottomrule
    \end{tabular*}
    \caption{MIP model specifications for the function $\pertfunc$. Vectorial inequalities are taken component-wise and $\mathcal{Q}_{\text{rot}}$ denotes Rotated-Quadratic Cone constraints \cite[\sectionInRef{3.1.2}]{aps2020mosek}.}
    \label{table:mip-pertfunc}
\end{table}

\subsection{Outer-Approximation Method Implementation}
\label{sec:supp-numerics:oa}

In our numerical experiments, we use our own \textsc{Python} implementation of the \oa{} method specified in \cite[Algorithm~3.1]{bertsimas2021unified}.
We use \textsc{Mosek} as \gls{mip} solver to address the outer-loop problems which are modified in place to model construction cost at each iteration.
Moreover, we use a coordinate-descent method to solve the inner-loop problems with a prox-linear update scheme as specified in \cite[Algorithm~2]{wright2015coordinate}.
Our implementation is also part of the \elops{} toolbox.\footnote{See \href{https://github.com/TheoGuyard/El0ps/blob/main/src/el0ps/solver/oa.py}{\texttt{https://github.com/TheoGuyard/El0ps/blob/main/src/el0ps/solver/oa.py}}.}

\subsection{Hyperparameters Calibration}
\label{sec:supp-numerics:calibration}

In our experiments, we calibrate hyperparameters for  instances of problem \eqref{prob:prob} where the penalty function is given by
\begin{equation}
    \label{eq:supp-numerics:calibration:penalty}
    \pertfunc(\pvi{}) = \reg_p\norm{\pv}{p}^p + \icvx(\abs{\pvi{}} \leq \bigM)
\end{equation}
for some $p \in \{1,2\}$, $\reg_p \in [0,+\infty[$, and $\bigM \in \ ]0,+\infty]$.
To this end, we rely on the \lolearn{} package \cite{hazimeh2022l0learn} which allows selecting some statistically relevant values.
More precisely, we use the cross-validation procedure implemented in the \texttt{cv.fit} function to approximately solve the problem 
\begin{equation}
    \label{eq:supp-numerics:calibration:prob}
    \min_{\pv \in \kR^{\pdim}} \datafunc(\dic\pv) + \reg\norm{\pv}{0} + \reg_p\norm{\pv}{p}^p
\end{equation}
over a grid of values of $\reg$ and $\reg_p$.
We select the hyperparameters leading to the best cross-validation score over 10 random folds.
Moreover, we set $\bigM = 1.5 \norm{\hat{\pv}}{\infty}$ where $\hat{\pv}$ corresponds to the approximate solution associated with these hyperparameters.
\clearpage
\bibliographystyle{spmpsci}
\bibliography{references}

\begin{thebibliography}{10}
\providecommand{\url}[1]{{#1}}
\providecommand{\urlprefix}{URL }
\expandafter\ifx\csname urlstyle\endcsname\relax
  \providecommand{\doi}[1]{DOI~\discretionary{}{}{}#1}\else
  \providecommand{\doi}{DOI~\discretionary{}{}{}\begingroup
  \urlstyle{rm}\Url}\fi

\bibitem{alon1999broad}
Alon, U., Barkai, N., Notterman, D.A., Gish, K., Ybarra, S., Mack, D., Levine,
  A.J.: Broad patterns of gene expression revealed by clustering analysis of
  tumor and normal colon tissues probed by oligonucleotide arrays.
\newblock Proceedings of the National Academy of Sciences \textbf{96}(12),
  6745--6750 (1999)

\bibitem{aps2020mosek}
ApS, M.: Mosek modeling cookbook (2020)

\bibitem{mosek}
ApS, M.: The MOSEK optimization toolbox for MATLAB manual. Version 10.1. (2024)

\bibitem{atamturk2020safe}
Atamturk, A., G{\'o}mez, A.: Safe screening rules for l0-regression from
  perspective relaxations.
\newblock In: Proceedings of the International Conference on Machine Learning,
  pp. 421--430. PMLR (2020)

\bibitem{bach2011convex}
Bach, F., Jenatton, R., Mairal, J., Obozinski, G., et~al.: Convex optimization
  with sparsity-inducing norms.
\newblock Optimization for Machine Learning \textbf{5}, 19--53 (2011)

\bibitem{Bauschke2017}
Bauschke, H.H., Combettes, P.L.: Convex Analysis and Monotone Operator Theory
  in Hilbert Spaces.
\newblock Springer (2017)

\bibitem{beck2017first}
Beck, A.: First-Order Methods in Optimization.
\newblock SIAM (2017)

\bibitem{beck2009fast}
Beck, A., Teboulle, M.: A fast iterative shrinkage-thresholding algorithm for
  linear inverse problems.
\newblock SIAM Journal on Imaging Sciences \textbf{2}(1), 183--202 (2009)

\bibitem{ben2021global}
Ben~Mhenni, R., Bourguignon, S., Ninin, J.: Global optimization for sparse
  solution of least squares problems.
\newblock Optimization Methods and Software \textbf{37}(5), 1740--1769 (2022)

\bibitem{bertrand2022beyond}
Bertrand, Q., Klopfenstein, Q., Bannier, P.A., Gidel, G., Massias, M.: Beyond
  l1: Faster and better sparse models with skglm.
\newblock Proceedings of the Advances in Neural Information Processing Systems
  \textbf{35}, 38950--38965 (2022)

\bibitem{bertsimas2021unified}
Bertsimas, D., Cory-Wright, R., Pauphilet, J.: A unified approach to
  mixed-integer optimization problems with logical constraints.
\newblock SIAM Journal on Optimization \textbf{31}(3), 2340--2367 (2021)

\bibitem{bertsimas2016best}
Bertsimas, D., King, A., Mazumder, R.: Best subset selection via a modern
  optimization lens.
\newblock The Annals of Statistics \textbf{44}(2), 813--852 (2016)

\bibitem{blumensath2009iterative}
Blumensath, T., Davies, M.E.: Iterative hard thresholding for compressed
  sensing.
\newblock Applied and Computational Harmonic Analysis \textbf{27}(3), 265--274
  (2009)

\bibitem{bourguignon2015exact}
Bourguignon, S., Ninin, J., Carfantan, H., Mongeau, M.: Exact sparse
  approximation problems via mixed-integer programming: Formulations and
  computational performance.
\newblock IEEE Transactions on Signal Processing \textbf{64}(6), 1405--1419
  (2015)

\bibitem{boyd2011distributed}
Boyd, S., Parikh, N., Chu, E., Peleato, B., Eckstein, J., et~al.: Distributed
  optimization and statistical learning via the alternating direction method of
  multipliers.
\newblock Foundations and Trends{\textregistered} in Machine learning
  \textbf{3}(1), 1--122 (2011)

\bibitem{buhlmann2014high}
B{\"u}hlmann, P., Kalisch, M., Meier, L.: High-dimensional statistics with a
  view toward applications in biology.
\newblock Annual Review of Statistics and Its Application \textbf{1}, 255--278
  (2014)

\bibitem{bynum2021pyomo}
Bynum, M.L., Hackebeil, G.A., Hart, W.E., Laird, C.D., Nicholson, B.L.,
  Siirola, J.D., Watson, J.P., Woodruff, D.L.: Pyomo--optimization modeling in
  python.
\newblock Springer Science \& Business Media (2021)

\bibitem{chaari2014hierarchical}
Chaari, L., Batatia, H., Dobigeon, N., Tourneret, J.Y.: A hierarchical
  sparsity-smoothness bayesian model for l0 + l1 + l2 regularization.
\newblock In: Proceedings of the International Conference on Acoustics, Speech
  and Signal Processing, pp. 1901--1905. IEEE (2014)

\bibitem{tourneret2013sparse}
Chaari, L., Tourneret, J.Y., Batatia, H.: Sparse bayesian regularization using
  bernoulli-laplacian priors.
\newblock In: Proceedings of the European Signal Processing Conference, pp.
  1--5. IEEE (2013)

\bibitem{chang2011libsvm}
Chang, C.C., Lin, C.J.: Libsvm: a library for support vector machines.
\newblock ACM Transactions on Intelligent Systems and Technology \textbf{2}(3),
  1--27 (2011)

\bibitem{chen1989orthogonal}
Chen, S., Billings, S.A., Luo, W.: Orthogonal least squares methods and their
  application to non-linear system identification.
\newblock International Journal of Control \textbf{50}(5), 1873--1896 (1989)

\bibitem{chen2014complexity}
Chen, X., Ge, D., Wang, Z., Ye, Y.: Complexity of unconstrained l2-lp
  minimization.
\newblock Mathematical Programming \textbf{143}(1), 371--383 (2014)

\bibitem{cplex2009v12}
Cplex, I.I.: V12. 1: User's manual for cplex.
\newblock International Business Machines Corporation \textbf{46}(53), 157
  (2009)

\bibitem{dedieu2021learning}
Dedieu, A., Hazimeh, H., Mazumder, R.: Learning sparse classifiers: Continuous
  and mixed integer optimization perspectives.
\newblock The Journal of Machine Learning Research \textbf{22}(1), 6008--6054
  (2021)

\bibitem{deza2022safe}
Deza, A., Atamt{\"u}rk, A.: Safe screening for logistic regression with l0-l2
  regularization.
\newblock In: arXiv preprint (2022)

\bibitem{dobigeon2009hierarchical}
Dobigeon, N., Hero, A.O., Tourneret, J.Y.: Hierarchical bayesian sparse image
  reconstruction with application to mrfm.
\newblock IEEE Transactions on Image Processing \textbf{18}(9), 2059--2070
  (2009)

\bibitem{fan2009network}
Fan, J., Feng, Y., Wu, Y.: Network exploration via the adaptive lasso and scad
  penalties.
\newblock The Annals of Applied Statistics \textbf{3}(2), 521--541 (2009)

\bibitem{friedman2010regularization}
Friedman, J., Hastie, T., Tibshirani, R.: Regularization paths for generalized
  linear models via coordinate descent.
\newblock Journal of Statistical Software \textbf{33}(1), 1--22 (2010)

\bibitem{golub1999molecular}
Golub, T.R., Slonim, D.K., Tamayo, P., Huard, C., Gaasenbeek, M., Mesirov,
  J.P., Coller, H., Loh, M.L., Downing, J.R., Caligiuri, M.A., et~al.:
  Molecular classification of cancer: class discovery and class prediction by
  gene expression monitoring.
\newblock Science \textbf{286}(5439), 531--537 (1999)

\bibitem{guyard2022node}
Guyard, T., Herzet, C., Elvira, C.: Node-screening tests for the l0-penalized
  least-squares problem.
\newblock In: Proceedings of the International Conference on Acoustics, Speech
  and Signal Processing, pp. 5448--5452. IEEE (2022)

\bibitem{guyard2025el0ps}
Guyard, T., Herzet, C., Elvira, C.: \textsc{El0ps}: An exact l0-regularized
  problems solver.
\newblock arXiv preprint  (2025)

\bibitem{pmlr-v235-guyard24a}
Guyard, T., Herzet, C., Elvira, C., Arslan, A.N.: A new branch-and-bound
  pruning framework for l0-regularized problems.
\newblock In: Proceedings of the International Conference on Machine Learning,
  pp. 48077--48096. PMLR (2024)

\bibitem{guyon2003design}
Guyon, I.: Design of experiments of the nips 2003 variable selection benchmark.
\newblock In: NIPS Workshop on Feature Extraction and Feature Selection, vol.
  253, p.~40 (2003)

\bibitem{hazimeh2022l0learn}
Hazimeh, H., Mazumder, R., Nonet, T.: L0learn: A scalable package for sparse
  learning using l0 regularization.
\newblock Journal of Machine Learning Research \textbf{24}(205), 1--8 (2023)

\bibitem{hazimeh2021sparse}
Hazimeh, H., Mazumder, R., Saab, A.: Sparse regression at scale:
  Branch-and-bound rooted in first-order optimization.
\newblock Mathematical Programming \textbf{196}(1), 347--388 (2022)

\bibitem{Herzet_eusipco10}
Herzet, C., Dr{\'e}meau, A.: Bayesian pursuit algorithms.
\newblock In: Proceedings of the European Signal Processing Conference, pp.
  1474--1478. IEEE (2010)

\bibitem{kronqvist2019review}
Kronqvist, J., Bernal, D.E., Lundell, A., Grossmann, I.E.: A review and
  comparison of solvers for convex minlp.
\newblock Optimization and Engineering \textbf{20}, 397--455 (2019)

\bibitem{lee2006efficient}
Lee, H., Battle, A., Raina, R., Ng, A.Y.: Efficient sparse coding algorithms.
\newblock In: Proceedings of the Advances in Neural Information Processing
  Systems, pp. 801--808 (2006)

\bibitem{li2017feature}
Li, J., Cheng, K., Wang, S., Morstatter, F., Trevino, R.P., Tang, J., Liu, H.:
  Feature selection: A data perspective.
\newblock Computing Surveys \textbf{50}(6), 1--45 (2017)

\bibitem{liu2018integrated}
Liu, J., Lichtenberg, T., Hoadley, K.A., Poisson, L.M., Lazar, A.J., Cherniack,
  A.D., Kovatich, A.J., Benz, C.C., Levine, D.A., Lee, A.V., et~al.: An
  integrated tcga pan-cancer clinical data resource to drive high-quality
  survival outcome analytics.
\newblock Cell \textbf{173}(2), 400--416 (2018)

\bibitem{Locatelli:2013xr}
Locatelli, M., Schoen, F.: Global Optimization: Theory, Algorithms, and
  Applications.
\newblock SIAM (2013)

\bibitem{massias2017safe}
Massias, M., Gramfort, A., Salmon, J.: From safe screening rules to working
  sets for faster lasso-type solvers.
\newblock In: NIPS Workshop on Optimization for Machine Learning (2017)

\bibitem{mhenni2020sparse}
Mhenni, R.B., Bourguignon, S., Mongeau, M., Ninin, J., Carfantan, H.: Sparse
  branch and bound for exact optimization of l0-norm penalized least squares.
\newblock In: Proceedings of the International Conference on Acoustics, Speech
  and Signal Processing, pp. 5735--5739. IEEE (2020)

\bibitem{needell2009cosamp}
Needell, D., Tropp, J.A.: Cosamp: Iterative signal recovery from incomplete and
  inaccurate samples.
\newblock Applied and Computational Harmonic Analysis \textbf{26}(3), 301--321
  (2009)

\bibitem{pilanci2015sparse}
Pilanci, M., Wainwright, M.J., El~Ghaoui, L.: Sparse learning via boolean
  relaxations.
\newblock Mathematical Programming \textbf{151}(1), 63--87 (2015)

\bibitem{rover2021weakly}
R{\"o}ver, C., Bender, R., Dias, S., Schmid, C.H., Schmidli, H., Sturtz, S.,
  Weber, S., Friede, T.: On weakly informative prior distributions for the
  heterogeneity parameter in bayesian random-effects meta-analysis.
\newblock Research Synthesis Methods \textbf{12}(4), 448--474 (2021)

\bibitem{samain2022techniques}
Samain, G., Bourguignon, S., Ninin, J.: Techniques for accelerating
  branch-and-bound algorithms dedicated to sparse optimization.
\newblock Optimization Methods and Software \textbf{39}(1), 4--41 (2024)

\bibitem{soubies2015continuous}
Soubies, E., Blanc-F{\'e}raud, L., Aubert, G.: A continuous exact l0-penalty
  (cel0) for least squares regularized problem.
\newblock SIAM Journal on Imaging Sciences \textbf{8}(3), 1607--1639 (2015)

\bibitem{soussen2011bernoulli}
Soussen, C., Idier, J., Brie, D., Duan, J.: From bernoulli--gaussian
  deconvolution to sparse signal restoration.
\newblock IEEE Transactions on Signal Processing \textbf{59}(10), 4572--4584
  (2011)

\bibitem{tibshirani1996regression}
Tibshirani, R.: Regression shrinkage and selection via the lasso.
\newblock Journal of the Royal Statistical Society Series B: Statistical
  Methodology \textbf{58}(1), 267--288 (1996)

\bibitem{tillmann2024cardinality}
Tillmann, A.M., Bienstock, D., Lodi, A., Schwartz, A.: Cardinality
  minimization, constraints, and regularization: a survey.
\newblock SIAM Review \textbf{66}(3), 403--477 (2024)

\bibitem{tropp2010computational}
Tropp, J.A., Wright, S.J.: Computational methods for sparse solution of linear
  inverse problems.
\newblock Proceedings of the IEEE \textbf{98}(6), 948--958 (2010)

\bibitem{wright2015coordinate}
Wright, S.J.: Coordinate descent algorithms.
\newblock Mathematical programming \textbf{151}(1), 3--34 (2015)

\bibitem{zhang2010nearly}
Zhang, C.H.: Nearly unbiased variable selection under minimax concave penalty.
\newblock The Annals of Statistics \textbf{38}(2), 894--942 (2010)

\bibitem{zou2005regularization}
Zou, H., Hastie, T.: Regularization and variable selection via the elastic net.
\newblock Journal of the Royal Statistical Society Series B: Statistical
  Methodology \textbf{67}(2), 301--320 (2005)

\end{thebibliography}

\end{document}